\documentclass[a4paper,12pt]{article}
\usepackage[utf8]{inputenc}
\usepackage{amsthm}
\usepackage{amsfonts}
\usepackage{amssymb}
\usepackage{graphicx}
\usepackage{rotating}
\usepackage{wrapfig}
\usepackage[all]{xy}
\usepackage{pgf,tikz}
\usetikzlibrary{arrows}
\usetikzlibrary[patterns]

\newcommand{\N}{\mathbb{N}}
\newcommand{\Z}{\mathbb{Z}}
\newcommand{\Q}{\mathbb{Q}}
\newcommand{\R}{\mathbb{R}}

\newcommand{\K}{\mathbb{K}}

\newcommand{\U}{\mathfrak{U}}

\newtheorem{prop}{Proposition}[section]
\newtheorem{cor}{Corollary}[section]
\newtheorem{obs}{Remark}[section]
\newtheorem{ex}{Example}[section]
\newtheorem{defi}{Definition}[section]

\newtheorem{teo}{Theorem}[section]

\newtheorem{lema}{Lemma}[section]

\title{A generalization of convergence actions}

\author{Lucas H. R. de Souza}

\begin{document}

\maketitle

\def\eod{\hfill$\square$}

\begin{abstract}Let a group $G$ act properly discontinuously and cocompactly on a locally compact space $X$. A Hausdorff compact space $Z$ that contains $X$ as an open subspace has the perspectivity property if the action $G\curvearrowright X$ extends to an action $G\curvearrowright Z$, by homeomorphisms, such that for every compact $K\subseteq X$ and every element $u$ of the unique uniform structure compatible with the topology of $Z$, the set $\{gK: g \in G\}$ has finitely many non $u$-small sets. We describe a correspondence between the compact spaces with the perspectivity property with respect to $X$ (and the fixed action of $G$ on it) and the compact spaces with the perspectivity property with respect to $G$ (and the left multiplication on itself). This generalizes a similar result for convergence group actions.
\end{abstract}

\tableofcontents

\section*{Introduction}

Let $G$ be a group that acts by homeomorphisms on a Hausdorff compact space $Y$ and properly discontinuously and cocompactly on a Hausdorff locally compact space $X$. Let $Z = X\dot{\cup} Y$  be a Hausdorff compact space (with canonical uniform structure $\U$) such that extends the topologies of $X$ and $Y$ and the induced action of $G$ in $Z$ is by homeomorphisms. We say that $Z$ is perspective if for every $u \in \U$ and every compact $K\subseteq X$, the set $\{g \in G: gK \notin Small(u)\}$ is finite. In other words, whenever $K$ goes closer to the boundary of $X$ (by elements of $G$), it becomes smaller.

In the case when $G = X$, with discrete topology, there is an equivalent definition of perspectivity: the right multiplication action on $G$ extends continuously to the identity on $Y$. This notion was used by Specker in \cite{Sp} and by Abels in \cite{Ab} (where $Y$ is totally disconnected) and by Toromanoff in \cite{To} (quasi-Specker compactification of $G$, where $Y$ does not need to be totally disconnected) where he proved that this notion is equivalent to the convergence property on the totally disconnected boundary and compactly generated locally compact group.

Some examples of perspectivity are:

\begin{ex}If $Z$ has the convergence property, then $Z$ has the perspective property \cite{Ge2}.
\end{ex}

\begin{ex}The visual compactification of a $CAT(0)$ space $X$ is perspective, if $G$ acts on $X$ by isometries. In particular, the compactification of $\Z^{n}$, for $n \geqslant 2$, with the visual boundary of $\R^{n}$ is perspective but it does not have the convergence property since the action on the boundary is trivial.
\end{ex}

\begin{ex}The Martin compactification of every finitely generated relatively hyperbolic group with virtually abelian parabolic subgroups is perspective (but does not have the convergence property unless all parabolic subgroups are virtually cyclic) \cite{DGGP}.
\end{ex}

A reasonable question to ask is if the Martin compactification of every finitely generated group is perspective but is not our objective on this paper.

\begin{ex}Gerasimov and Potyagailo showed in \cite{GP2} that the pullback problem for convergence actions does not have a solution in general. A pullback may not exist even when both actions are relatively hyperbolic and the group is not finitely generated. However, we show that it is not the case for perspective property. So, families of convergence actions that do not have pullback convergence actions give rises to spaces that do have the perspective property but its associate actions must not have the convergence property.
\end{ex}

Our goal is to present a behavior on spaces with the perspectivity property that already happens on convergence actions: if the action of $G$ on $Y$ has the convergence property and $Z = G\dot{\cup} Y$ and $Z'= X\dot{\cup} Y$ are the unique topological spaces that the induced actions have the convergence property, then the topology of $Z'$ can be constructed from the topology of $Z$ \cite{Ge2}.

If we consider the category $Pers(G)$ (respec. $Pers(\varphi)$, where $\varphi: G \curvearrowright X$ is the properly discontinuous cocompact action), whose objects are of the form $G\dot{\cup} Y$  (respec. $X\dot{\cup} Y$) that are perspective and the morphisms are the equivariant continuous maps whose restriction is the identity in $G$ (respec. identity in $X$), we are able to state our main theorem:

\

\textbf{Theorem \ref{main}} Let $G$ be a group, $X$ a Hausdorff locally compact space and $\varphi: G \curvearrowright X$ a properly discontinuous cocompact action. Then, the categories $Pers(G)$ and $Pers(\varphi)$ are isomorphic.

\

So, there is a correspondence between those classes of compact spaces and the functor from $Pers(G)$ to $Pers(\varphi)$ sends objects of the form $Z = G\dot{\cup} Y$ to objects of the form $Z'= X\dot{\cup} Y$ that are constructed the same way as on the compactifications that have the convergence property.

This paper is structured in four main sections. The first one contains some preliminary results in topology that are used on the other sections.

The second section is devoted to the development of the theory of sum of spaces. If $X$ and $Y$ are topological spaces, we construct from a map $f: Closed(X) \rightarrow Closed(Y)$ (with certain properties) a topological space $X+_{f}Y$ where the set is $X\dot{\cup} Y$, extends both topologies and $X$ is open. This construction appeared in \cite{Ge2} in the proof of the existence of the Attractor-Sum. It seems to be a convenient tool to work with compactifications of locally compact spaces. As an example, which is important on a later section, it is constructed in the end of this section the Freudenthal compactification of a locally compact, connected and locally connected space. All subsequent constructions on this paper take advantage of this theory

The third section is the main. It is divided into three subsections:

\begin{enumerate}
    \item The first subsection presents the general notions of boundaries of groups and of spaces where a group acts properly discontinuously and cocompactly. It also presents its respective categories, the way to interchange between this two notions functorially and an example showing why those functors do not work really well on the whole categories.
    \item The second subsection presents subcategories of the categories presented on the first one that allow the restrictions of those functors (also presented on the first part) to be isomorphisms of categories. The objects of these categories are called quasi-perspectives. These are not the best places to work, since it is allowed spaces that are not Hausdorff.
    \item In the third subsection, we restrict the categories to the objects that are Hausdorff. They are called perspective. It is stated the main theorem.
\end{enumerate}

In the fourth  section are established some general properties of those categories:

\begin{enumerate}
    \item The first part shows that the good behavior of the functors do not happen outside the categories of quasi-perspectivity.
    \item The second part shows a process that turns a compact space that contains $X$ and extends continuously the action by $G$ to an element of the category of quasi-perspectives. However, this process withdraw every possibility of a space to be Hausdorff, except on the trivial case: already perspective.
    \item The third one presents some considerations about subspaces and quotients that works as in the convergence case.
    \item The fourth one presents a proof that the category of perspectivities is closed under small limits.
    \item The fifth one presents a proof that if a group $G$ acts by isometries, properly discontinuously and cocompactly on a proper $CAT(0)$ space $X$, then the compactification of $X$ with its visual boundary is perspective. So, the boundaries of such $CAT(0)$ spaces can be transferred to be boundaries of the respective isometry groups.
    \item The sixth one presents an already known result that sum-attractor compactification for the convergence case is perspective. So, the perspectivity property actually generalizes the convergence property. Using the fact that there exist limits in the perspectivity category, there exists examples \cite{GP2} of two convergence actions that do not have a product in the category of convergence actions of a fixed group but do have a product in the category of perspective compactifications.
    \item The seventh part is just an application of the theory of perspectivities to the theory of spaces of ends. It shows how this theory seems to be a good tool to work with boundaries, doing a simple proof of one of Hopf's observations.
    \item The last part presents two examples: the first one is a non finitely generated group that admits any compact Hausdorff space with countable basis as its boundary in a compactification with the perspective property and the second one is a finitely generated group that admits any Peano space as its boundary in a compactification with the perspective property.
\end{enumerate}

\section{Preliminaries}

This section contains some well known results that are used through this paper.

We use the symbol $\Box$, besides its usual propose, on the end of a proposition, lemma or theorem to say that its proof follows immediately from the previous considerations.

\begin{defi}Let $G_{i}$ be groups, $\alpha: G_{1} \rightarrow G_{2}$ a homomorphism, $X_{i}$ sets, and $\varphi_{i}: G \curvearrowright X_{i}$ actions. We say that a map $f: X_{1} \rightarrow X_{2}$ is $\alpha$-equivariant if $\forall g \in G_{1}, \ \forall x \in X_{1}, \ f(\varphi_{1}(g,x)) = \varphi_{2}(\alpha(g),f(x))$. In the case $G_{1} = G_{2} = G$, and $\alpha = id_{G}$, we say that $f$ is $G$-equivariant or just equivariant, if there is no doubt of the meaning.
\end{defi}

\begin{prop}(RAPL - Right Adjoints Preserves Limits, Proposition 3.2.2 of \cite{Bo}) Let $F: \mathcal{C} \rightarrow \mathcal{D}$ and $G: \mathcal{D} \rightarrow \mathcal{C}$ be two functors with $G$ adjoint to $F$. If $H: \mathcal{E} \rightarrow \mathcal{D}$ is a functor that possesses a limit, then $\lim\limits_{\longleftarrow} (G\circ H)$ exists and is equal to $G(\lim\limits_{\longleftarrow} H)$.
\end{prop}

\begin{defi}Let $(X,\mathcal{U})$ be a uniform space, $Y \subseteq X$ and $u \in \mathcal{U}$. We define the $u$-neighbourhood of $Y$ by $\mathfrak{B}(Y,u) = \{x \in X: \exists y \in Y: \ (x,y)\in u\}$.
\end{defi}

\begin{prop}\label{small}Let $f:(X_{1},\U_{1}) \rightarrow (X_{2},\U_{2})$ be a uniformly continuous map, $u \in \U_{2}$ and $Y \subseteq X_{2}$. If $Y \in Small(u)$, then $f^{-1}(Y) \in Small((f^{2})^{-1}(u))$. If $f$ is surjective, then the converse is also true.
\end{prop}

\begin{proof}($\Rightarrow$) Let $x,y\in f^{-1}(Y)$. We have that $f(x),f(y) \in Y$, which implies that $(f(x),f(y))\in u$ and then $(x,y)\in(f^{-1}(f(x)),f^{-1}(f(y))) \subseteq f^{-1}(u)$. Thus, $f^{-1}(Y) \in Small((f^{2})^{-1}(u))$.

($\Leftarrow$) Let's suppose that $Y \notin Small(u)$. Then, there exists $x,y \in Y$ such that $(x,y) \notin u$. Let $x' \in f^{-1}(x)$ and $y' \in f^{-1}(y)$ (since $f$ is surjective). If $(x',y') \in (f^{2})^{-1}(u)$, then $(x,y) = (f(x'),f(y')) \in f^{2}((f^{2})^{-1}(u)) = u$, contradicting our assumption. Thus, $(x',y') \notin (f^{2})^{-1}(u)$, which implies that $f^{-1}(Y) \notin Small((f^{2})^{-1}(u))$.
\end{proof}

\begin{prop}(Proposition 10, $\S2.7$, Chapter 2 of \cite{Bou}) Let $\Gamma$ be a directed set, $\{(X_{\alpha},\U_{\alpha}), f_{\alpha_{1}\alpha_{2}}\}_{\alpha,\alpha_{1},\alpha_{2} \in \Gamma}$ an inverse system of uniform spaces, $B_{\alpha}$ a basis for $\U_{\alpha}, \ (X,\U) = \lim\limits_{\longleftarrow} X_{\alpha}$ and $\pi_{\alpha}: X \rightarrow X_{\alpha}$ the projection maps. Then, the set $\{(\pi_{\alpha}^{2})^{-1}(b): \alpha \in \Gamma, \ b \in B_{\alpha}\}$ is a basis for $\U$.
\end{prop}

\begin{defi}Let $X$ be a topological space. A family $\{F_{\alpha}\}_{\alpha \in \Gamma}$ of subsets of $X$ is locally finite if $\forall x \in X, \ \exists U$ a neighbourhood of $x$ such that $U\cap F_{\alpha} \neq \emptyset$ only for a finite subset of $\Gamma$.
\end{defi}

\begin{prop}(Proposition 4, $\S2.5$, Chapter 1 of \cite{Bou}) Let $X$ be a topological space and $\{F_{\alpha}\}_{\alpha \in \Gamma}$ a locally finite family of closed sets of $X$. Then, $\bigcup_{\alpha \in \Gamma} F_{\alpha}$ is closed.
\end{prop}

\begin{prop}(Proposition 9, $\S4.4$, Chapter 1 of \cite{Bou}) Let $\Gamma$ be a directed set, $\{X_{\alpha}, f_{\alpha_{1}\alpha_{2}}\}_{\alpha,\alpha_{1},\alpha_{2} \in \Gamma}$ an inverse system of topological spaces, $B_{\alpha}$ a basis for $X_{\alpha}, \ X = \lim\limits_{\longleftarrow}X_{\alpha}$ and $\pi_{\alpha}: X \rightarrow X_{\alpha}$ the projection maps. Then, the set $\{\pi_{\alpha}^{-1}(b): \alpha \in \Gamma, \ b \in B_{\alpha}\}$ is a basis for $X$.
\end{prop}

\begin{prop}(Corollary 3.1.20 of \cite{En}) Let $X$ be a Hausdorff compact space, $m$ an infinite cardinal and $\{X_{\alpha}\}_{\alpha \in \Gamma}$ a family of subspaces of $X$ such that $X = \bigcup_{\alpha \in \Gamma} X_{\alpha}, \ \# \Gamma \leqslant m$ and $\forall \alpha \in \Gamma, \ \omega(X_{\alpha}) \leqslant m$ (where $\omega(Y)$ is the lowest cardinality of a basis of $Y$). Then, $\omega(X) \leqslant m$.
\end{prop}

\begin{cor}\label{uniaometrizavel}Let $X$ be a Hausdorff compact space where there exists a family of subspaces $\{X_{n}\}_{n\in \N}$ such that each one has a countable basis and $X = \bigcup_{n\in \N}X_{n}$. Then, $X$ is metrizable.
\end{cor}

\begin{proof}By the last proposition, we have that $X$ has a countable basis and, since the space is compact and Hausdorff, it follows that it is metrizable.
\end{proof}

\begin{defi}Let $X$ be a topological space, $\{U_{\alpha}\}_{\alpha \in \Gamma}$ an open cover of $X$ and $a,b \in X$. A simple chain from $a$ to $b$ by elements of the cover is a set $\{U_{1},...,U_{n}\} \subseteq \{U_{\alpha}\}_{\alpha \in \Gamma}$ such that $a \in U_{1}$, $b \in U_{2}$ and $U_{i}\cap U_{j} \neq \emptyset$ if and only if $|i-j| < 2$.
\end{defi}

\begin{prop}(Theorem 26.15 of \cite{SW}) A topological space $X$ is connected if and only if for every open cover and every $a,b \in X$, there exists a simple chain from $a$ to $b$ by elements of the cover.
\end{prop}

\begin{prop}Let $X$ be a connected, locally connected and locally compact space $X$. Then, $\forall K \subseteq X$ compact there exists $K' \subseteq X$ compact and connected such that $K \subseteq K'$.
\end{prop}

\begin{proof}Let $x \in X$ and $U_{x}$ be an open neighbourhood of $x$ such that $U_{x}$ is connected and $Cl_{X}(U_{x})$ is compact (it exists since $X$ is locally compact and locally connected). We have that $\{U_{x}\}_{x\in X}$ is an open cover of $X$. Since $K$ is compact, there exists $x_{1},...,x_{n} \in X$ such that $K \subseteq U_{x_{1}} \cup ... \cup U_{x_{n}}$. Let $V_{1},...,V_{k}$ be the connected components of this union (it is finite since each $U_{i}$ is connected) and $y_{i} \in V_{i}$. Since $X$ is connected, $\forall i \in \{1,...,k\}$ there exists a simple chain consisting of elements of $\{U_{x}\}_{x\in X}$ connecting $y_{1}$ and $y_{i}$. Let $W_{i}$ be the union of this chain (it is connected since each $U_{x}$ is connected and in a chain consecutive sets are not disjoint). We have that $V_{1}\cup...\cup V_{k} \cup W_{1}\cup...\cup W_{k}$ is connected since each of then is connected and $W_{i}$ connects $V_{1}$ to $V_{i}$. Let $K' = Cl_{X}(V_{1}\cup...\cup V_{k} \cup W_{1}\cup...\cup W_{k})$. We have that $K'$ is connected, since $V_{1}\cup...\cup V_{k} \cup W_{1}\cup...\cup W_{k}$ is connected. Since $V_{i}$ and $W_{i}$ are finite unions of $\{U_{x}\}_{x \in X}$, their closures are compact, which implies that $K'$ is compact. Thus, $K \subseteq K'$ and $K'$ is connected and compact.
\end{proof}

\begin{prop}(Theorem 4.18 of \cite{Ke}) Let $M$ be a compact Hausdorff space with countable basis. Then, there exists a quotient map $\pi: K \rightarrow M$, where $K$ is the Cantor set.
\end{prop}

\begin{defi}A Peano space is a topological space that is compact, connected, locally connected and metrizable.
\end{defi}

\begin{prop}(Hahn-Mazurkiewicz - Theorem 31.5 of \cite{SW}) A space $X$ is a Peano space if and only if it is Hausdorff and there is a quotient map $f: [0,1] \rightarrow X$.
\end{prop}

As a consequence, we have that if $X,Y$ are Peano spaces with $\# X > 1$, then there exists a quotient map $f: X \rightarrow Y$.

\begin{prop}\label{quotientaction}Let $G$ be a Hausdorff locally compact group and $\varphi$ a continuous action on a compact Hausdorff space $X$. If $\sim$ is a closed equivalence relation on $X$ compatible with the action $\varphi$, then the induced action $\tilde{\varphi}: G \curvearrowright X/\sim$ is continuous.
\end{prop}

\begin{proof}Let $\pi: X \rightarrow X /\sim$ be the quotient map. Since $G$ is Hausdorff and locally compact, we have that $id\times \pi: G \times X \rightarrow G\times (X/\sim)$ is a quotient map. We define $\tilde{\varphi}: G \curvearrowright X/\sim$ such that the diagram commutes:

$$ \xymatrix{ G\times X \ar[r]^{\varphi} \ar[d]^{id\times\pi} & X \ar[d]^{\pi} \\
            G\times (X/\sim) \ar[r]^{\tilde{\varphi}} & X/\sim  } $$

Since $\pi\circ \varphi$ is continuous and $id\times \pi$ is a quotient map, we have that $\tilde{\varphi}$ is continuous.
\end{proof}

\begin{prop}(Stallings, $\S5.A.9$ and $\S5.A.10$ of \cite{St} ) Let $G$ be a finitely generated group. Then, $Ends(G)$ is infinite if and only if $G$ splits to an amalgamated product $G_{1}\ast_{H} G_{2}$, where $[G:G_{1}] \geqslant 2$, $[G:G_{2}] > 2$ and $H$ is finite, or to an HNN extension $G_{1}\ast_{H}$, where $H$ is a proper finite subgroup of $G_{1}$.
\end{prop}

\section{Sum of spaces}

\subsection{Construction}

\begin{defi}Let $X$ and $Y$ be topological spaces. We say that an application $f: Closed(X) \rightarrow Closed (Y)$ is admissible if  $\forall A,B \in Closed(X)$, $f(A\cup B) = f(A)\cup f(B)$ and $f(\emptyset) = \emptyset$. Let's fix an admissible map $f$. Let's declare $A \subseteq X\dot{\cup} Y$ as a closed set if $A \cap X \in Closed(X), \ A \cap Y \in Closed(Y)$ and $f(A\cap X) \subseteq A$. Therefore, let's denote by $\tau_{f}$ the set of the complements of this closed sets and $X+_{f}Y = (X\dot{\cup} Y,\tau_{f})$.
\end{defi}

\begin{prop}Actually, $\tau_{f}$ is a topology.
\end{prop}

\begin{proof}We have that $(X\cup Y)\cap X = X \in Closed(X), \ (X\cup Y)\cap Y = Y \in Closed(Y)$ and $f((X\cup Y)\cap X) = f(X) \subseteq X\cup Y$. So, $X\cup Y \in Closed(X+_{f}Y)$.

We have also that $\emptyset \cap X = \emptyset \in Closed(X), \ \emptyset\cap \emptyset = Y \in Closed(Y)$ and $f(\emptyset \cap X) = f(\emptyset) = \emptyset$. So, $\emptyset \in Closed(X+_{f}Y)$.

If $A,B \in Closed(X+_{f}Y)$, then $(A\cup B)\cap X = (A\cap X)\cup (B \cap X) \in Closed(X)$, $(A\cup B)\cap Y = (A\cap Y)\cup (B \cap Y) \in Closed(Y)$ and $f((A\cup B)\cap X) = f((A\cap X)\cup (B \cap X)) = f(A\cap X)\cup f(B \cap X) \subseteq A \cup B$ (because $f(A\cap X)\subseteq A$ and $f(B \cap X) \subseteq B$). So, $A\cup B \in Closed(X+_{f}Y)$.

Finally, let $\{A_{i}\}_{i\in \Gamma}$ be a family of closed sets. Then $(\bigcap\limits_{i\in \Gamma}A_{i}) \cap X = \bigcap\limits_{i\in \Gamma}(A_{i}\cap X) \in Closed(X)$, because each $A_{i}\cap X \in Closed(X)$. Analogously, $(\bigcap\limits_{i\in \Gamma}A_{i}) \cap Y \in Closed(Y)$. And $\forall i \in \Gamma, \ f((\bigcap\limits_{i\in \Gamma}A_{i}) \cap X) \subseteq f(A_{i}\cap X) \subseteq A_{i}$, which implies that $f((\bigcap\limits_{i\in \Gamma}A_{i}) \cap X)\subseteq \bigcap\limits_{i\in \Gamma}A_{i}$. So, $\bigcap\limits_{i\in \Gamma}A_{i} \in Closed(X+_{f}Y)$. \end{proof}

\begin{prop}Let $A \in Closed (X)$. Then $Cl_{X+_{f}Y}A = A\cup f(A)$.
\end{prop}

\begin{proof}We have that $(A \cup f(A)) \cap X = A \in Closed (X)$, $(A \cup f(A)) \cap Y = f(A) \in Closed (Y)$ and $f((A\cup f(A))\cap X) = f(A) \subseteq A \cup f(A)$. So, $A\cup f(A) \in Closed(X+_{f}Y)$.

Let $B\in Closed(X+_{f}Y)$ such that $A \subseteq B$. We have that $f(B\cap X)\subseteq B$. But $f(B\cap X) = f((A\cup B)\cap X) = f(A\cap X) \cup f(B\cap X) = f(A) \cup f(B\cap X)$, which implies that $f(A) \subseteq B$. So, $A \cup f(A) \subseteq B$.

Thus, $Cl_{X+_{f}Y}A = A\cup f(A)$.
\end{proof}

\begin{cor}$X$ is dense in $X+_{f}Y$ if and only if $f(X) = Y$.
\end{cor}

\begin{proof}If $f(X) = Y$, then $Cl_{X+_{f}Y}(X) = X \cup f(X) = X \cup Y$, which implies that $X$ is dense in $X+_{f}Y$. If $f(X) = Y_{1} \subsetneq Y$, then $Cl_{X+_{f}Y}(X) = X \cup f(X) = X \cup Y_{1} \subsetneq X \cup Y$, which implies that $X$ is not dense in $X+_{f}Y$.
\end{proof}

\begin{prop}$Y$ is closed in $X+_{f}Y$.
\end{prop}

\begin{proof}We have that $Y\cap X = \emptyset \in Closed(X), \ Y\cap Y = Y \in Closed(Y)$ and $f(Y\cap X) = f(\emptyset) = \emptyset \subseteq Y$. Thus, $Y \in Closed(X+_{f}Y)$.
\end{proof}

\begin{prop}The maps $id_{X}:X \rightarrow X+_{f}Y$ and $id_{Y}:Y \rightarrow X+_{f}Y$ are embeddings.
\end{prop}

\begin{proof} Let $F \in Closed(X+_{f}Y)$. Then $F\cap X \in Closed(X)$. However, $F \cap X = id_{X}^{-1}(F)$. So, $id_{X}$ is continuous. Let $F \in Closed(X)$. We have that $Cl_{X+_{f}Y}(F) = F \cup(f(F))$ and $(F \cup(f(F)))\cap X = F$, which implies that $F$ is closed in $X$ as a subspace of $X+_{f}Y$. Thus, $id_{X}$ is an embedding.

We have that $Y$ is closed in $X+_{f}Y$, so $\forall F \subseteq Y, \ F \in Closed (Y)$ if and only if $F \in Closed(X+_{f}Y)$. Thus, $id_{Y}$ is an embedding.
\end{proof}

\begin{prop}\label{bomcomp}Let $Z$ be a topological space such that $Z = X\dot{\cup} Y$ and $X$ is open. We define $f: Closed(X) \rightarrow Closed(Y)$ as $f(A) = Cl_{Z}(A) \cap Y$. So, $Z$ and $X+_{f}Y$ have the same topology.
\end{prop}

\begin{proof}Let $A,B \in Closed(X)$. So, $f(A\cup B) = Cl_{Z}(A\cup B) \cap Y =$ $\\ (Cl_{Z}(A)\cup Cl_{Z}(B)) \cap Y = (Cl_{Z}(A)\cap Y)\cup (Cl_{Z}(B) \cap Y) = f(A)\cup f(B)$ and $f(\emptyset) = Cl_{Z}(\emptyset) \cap Y = \emptyset \cap Y = \emptyset$. So, $f$ is admissible.

Let $A \in Closed (Z)$. We have that $A\cap X \in Closed (X), \ A\cap Y \in Closed (Y)$ and $f(A \cap X) = Cl_{Z}(A\cap X) \cap Y \subseteq Cl_{Z}(A\cap X) \subseteq Cl_{Z}(A) = A$. So, $A \in Closed(X+_{f}Y)$. Let $A \in Closed(X+_{f}Y)$. We have that $A\cap X \in Closed(X)$, which implies that $Cl_{X}(A\cap X) = A\cap X \subseteq A$. But $Cl_{X}(A\cap X) = Cl_{Z}(A\cap X)\cap X$, which implies that $Cl_{Z}(A\cap X)\cap X \subseteq A$. For the other hand, we have that $f(A\cap X) \subseteq A$. But $f(A \cap X) = Cl_{Z}(A\cap X) \cap Y$, which implies that $Cl_{Z}(A\cap X) \cap Y \subseteq A$. So, $Cl_{Z}(A\cap X) \subseteq A$. But $A = (A\cap X)\cup(A \cap Y)$, which implies that $Cl_{Z}(A) = Cl_{Z}(A\cap X) \cup Cl_{Z}(A\cap Y)$. Since $Y\in Closed (Z)$ and $A \cap Y \in Closed(Y)$, it follows that $A \cap Y  \in Closed (Z)$ which implies that $Cl_{Z}(A\cap Y) = A \cap Y \subseteq A$. So, $Cl_{Z}(A) \subseteq A$ and follows that $A \in Closed(Z)$.

Thus, $Closed (Z) = Closed(X+_{f}Y)$.

\end{proof}

As a simple example, we have:

\begin{prop}Let $X,Y$ be topological spaces. Then $X+_{\emptyset}Y$ is the coproduct of $X$ and $Y$ (where $\emptyset$ means the constant map equal to $\emptyset$).
\end{prop}

\begin{proof}We have that $X\cup f(X) = X \in Closed(X+_{\emptyset}Y)$. So, $X$ and $Y$ are closed, disjoint and $X\cup Y = X+_{\emptyset}Y$, which implies that $X+_{\emptyset}Y$ is the coproduct of $X$ and $Y$.
\end{proof}

\subsection{Separation}

\begin{prop}Let $X,Y$ be topological spaces and $X+_{f}Y$ Hausdorff. So, $\forall K \subseteq X$ compact, $f(K) = \emptyset$.
\end{prop}

\begin{proof}Let $K \subseteq X$ be a compact. We have that $Cl_{X+_{f}Y}(K) = K \cup f(K)$. Since $X+_{f}Y$ is Hausdorff, it follows that $K$ is closed, which means that $K \cup f(K) = K$, which implies that $f(K) = \emptyset$.
\end{proof}

\begin{prop}Let $X,Y$ be Hausdorff spaces, with $X$ locally compact. Then $X+_{f}Y$ is Hausdorff if and only if $\forall K \subseteq X$ compact, $f(K) = \emptyset$ and $\forall a,b \in Y, \ \exists A,B \in Closed(X): \ A \cup B = X, \ b \notin f(A)$ e $a \notin f(B)$.
\end{prop}

\begin{proof}($\Rightarrow$) Let $a,b \in Y$. Since $X+_{f}Y$ is a Hausdorff space, $\exists U,V \in Closed(X+_{f}Y): \ U \cup V = X+_{f}Y, \ a \notin V$ and $b \notin U$. Take $A = U \cap X$ and $B = V \cap X$. We have that $A,B \in Closed(X)$ and $A \cup B = (U\cup V) \cap X = X$. Since $U$ and $V$ are  closed, $f(A) = f(U \cap X) \subseteq U$ and $f(B) = f(V \cap X) \subseteq V$. Thus, $a \notin f(B)$ and $b \notin f(A)$. We already saw on this case that $f(K) = \emptyset$ for every compact $K \subseteq X$.

($\Leftarrow$) Let $a,b \in X$. Since $X$ is Hausdorff, there exists $U,V$ open and disjoint neighbourhoods of $a$ and $b$. But $X$ is open in $X+_{f}Y$, which implies that $U$ and $V$ are open and disjoint sets of $X+_{f}Y$ that separate $a\in U$ and $b\in V$.

Let $a\in X$ and $b \in Y$. Since $X$ is locally compact, there exists an open neighbourhood $U$ of $a$ in $X$ such that $Cl_{X}(U)$ is compact. Since $X$ is open in $X+_{f}Y$, we have that $U$ is an open neighbourhood of $a$ in $X+_{f}Y$ and, since $Cl_{X}(U)$ is compact, we have $f(Cl_{X}(U)) = \emptyset$, which implies that $Cl_{X}(U)$ is closed in $X+_{f}Y$. It follows that $U$ and $(X+_{f}Y) - Cl_{X}(U)$ separate $a$ and $b$.

Let $a,b \in Y$. So, there exists $A,B \in Closed(X)$ such that  $A \cup B = X$, $a \notin f(B)$ and $b \notin f(A)$. Since, $Y$ is Hausdorff, there exists $C,D \in Closed(Y)$ such that $C\cup D = Y, \ a \notin D$ and $b \notin C$. We have that $A\cup f(A) \cup C$ and $B\cup f(B) \cup D$ are closed sets in $X+_{f}Y$ such that $A\cup f(A)\cup C\cup B\cup f(B) \cup D = (A\cup B) \cup (f(A) \cup f(B) \cup C \cup D) = X \cup Y = X+_{f}Y, \ a \notin B \cup f(B) \cup D$ and $b \notin A \cup f(A) \cup C$. So, $(X+_{f}Y) - (B \cup f(B)\cup D)$ and $(X+_{f}Y) - (A \cup f(A)\cup C)$ are open sets that separate $a$ and $b$.

Thus, $X+_{f}Y$ is Hausdorff.
\end{proof}

\subsection{Compactness}

\begin{prop}Let $X,Y$ be topological spaces with $Y$ compact and $f$ an admissible map. Then $X+_{f}Y$ is compact if and only if $\forall A \in Closed(X)$ non compact, $f(A) \neq \emptyset$.
\end{prop}

\begin{proof}($\Rightarrow$) Let $A \in Closed(X)$ be non  compact. Since $X+_{f}Y$ is compact, we have that $A$ is not closed, which implies that $f(A)\neq \emptyset$.

($\Leftarrow$) Let $\mathcal{F}$ be a filter in $X+_{f}Y$. If $\exists K \in \mathcal{F}$ compact, then $\mathcal{F}\cap K = \{A\cap K: A \in \mathcal{F}\}$ is a filter in $K$ which has a cluster point $x$ (because $K$ is compact). Since $K \in \mathcal{F}$, we have that $\mathcal{F}\cap K$ is a basis for $\mathcal{F}$, which implies that $x$ is a cluster point of $\mathcal{F}$.

Let's suppose that $\nexists K \in \mathcal{F}: K$ is compact. Let $S \in \mathcal{F}: Cl_{X+_{f}Y}(S) \subseteq X$. Since $Cl_{X+_{f}Y}(S) = S \cup f(S)$, we have that $f(S) = \emptyset$, which implies that $S$ is compact, a contradiction. So, $\forall A \in \mathcal{F}, Cl_{X+_{f}Y}(A) \cap Y \neq \emptyset$.

So, we have that the set $\{Cl_{X+_{f}Y}(A)\cap Y\}_{A \in \mathcal{F}}$ has the finite intersection property (if $A_{1},...,A_{n} \in \mathcal{F}$, then $A_{1}\cap...\cap A_{n} \in \mathcal{F}$, which implies that $Cl_{X+_{f}Y}(A_{1} \cap ...\cap A_{n}) \cap Y \neq \emptyset$ and then $Cl_{X+_{f}Y}(A_{1}) \cap ...\cap Cl_{X+_{f}Y}(A_{n}) \cap Y \neq \emptyset$ because $Cl_{X+_{f}Y}(A_{1} \cap ...\cap A_{n}) \subseteq Cl_{X+_{f}Y}(A_{1}) \cap ...\cap Cl_{X+_{f}Y}(A_{n})$). Since $Y$ is compact, $\exists x \in \bigcap_{A \in \mathcal{F}}Cl_{X+_{f}Y}(A)\cap Y$, which implies that $\forall A \in \mathcal{F}$, $x \in Cl_{X+_{f}Y}(A)$. So, $x$ is a cluster point of $\mathcal{F}$.

Thus, every filter has a cluster point, which implies that  $X+_{f}Y$ is compact.

\end{proof}

\subsection{Continuous maps between sums of spaces}

\begin{defi}Let $X+_{f}Y$ and $Z+_{h}W$ be topological spaces and $\psi: X \rightarrow Z$ and $\phi: Y \rightarrow W$ continuous maps. So, we define $\psi + \phi: X+_{f}Y \rightarrow Z+_{h}W$ by $(\psi + \phi)(x) = \psi(x)$ if $x \in X$ and $\phi(x)$ if $x \in Y$. If $G$ is a group, $\psi: G \curvearrowright X$ and $\phi: G \curvearrowright Y$, then we define $\psi+\phi: G \curvearrowright X+_{f}Y$ by $(\psi+\phi)(g,x) = \psi(g,x)$ if $x \in X$ and $\phi(g,x)$ if $x \in Y$.
\end{defi}

Our next concern is to decide when those maps are continuous.

\begin{prop}Let $X+_{f}Y$ and $Z+_{h}W$ be topological spaces and $\\ \psi: X \rightarrow Z$ and $\phi: Y \rightarrow W$ continuous maps. Then, the application $\psi + \phi: X+_{f}Y \rightarrow Z+_{h}W$ is continuous if and only if $\forall A \in Closed(Z)$,  $f(\psi^{-1}(A)) \subseteq \phi^{-1}(h(A))$. In another words, we have the diagram:

$$ \xymatrix{ Closed(X) \ar[r]^{f} & Closed(Y) \\
            Closed(Z) \ar[r]^{h} \ar[u]^{\psi^{-1}} \ar@{}[ur]|{\subseteq} & Closed(W) \ar[u]^{\phi^{-1}} } $$

\end{prop}

\begin{proof}$(\Rightarrow)$ Let $A$ be a closed set in $Z+_{h}W$. We have that $(\psi + \phi)^{-1}(A) =$ $\psi^{-1}(A \cap Z) \cup \phi^{-1}(A \cap W)$. Let's prove that this set is closed, showing that it is equal to its closure. We have that $Cl_{X+_{f}Y}(\psi^{-1}(A \cap Z) \cup \phi^{-1}(A \cap W)) = Cl_{X+_{f}Y}(\psi^{-1}(A \cap Z)) \cup Cl_{X+_{f}Y}(\phi^{-1}(A \cap W))$. But $Cl_{X+_{f}Y}(\psi^{-1}(A \cap Z))= \psi^{-1}(A \cap Z) \cup f(\psi^{-1}(A \cap Z))$. We have that $f(\psi^{-1}(A \cap Z)) \subseteq \phi^{-1}(h(A \cap Z))$ by hypothesis and $\phi^{-1}(h(A \cap Z)) \subseteq \phi^{-1}(A \cap W)$, because $A$ is closed in $Z+_{h}W$. So, $Cl_{X+_{f}Y}(\psi^{-1}(A \cap Z)) \subseteq \psi^{-1}(A \cap Z) \cup \phi^{-1}(A \cap W)$. We have that $Cl_{X+_{f}Y}(\phi^{-1}(A \cap W)) =  \phi^{-1}(A \cap W)$ (because $A \cap W \in Closed(W)$ and $\phi$ is continuous) which implies that  $Cl_{X+_{f}Y}(\psi^{-1}(A \cap Z) \cup \phi^{-1}(A \cap W)) \subseteq \psi^{-1}(A \cap Z) \cup \phi^{-1}(A \cap W)$ and follows the equality. Thus, $\psi + \phi$ is continuous.

$(\Leftarrow)$ Let's suppose that $\psi + \phi$ is continuous. Let $A$ be a closed set in $Z$. We have that $A \cup h(A)$ is closed in $Z+_{h}W$. By continuity of the map $\psi + \phi$, we have that $(\psi + \phi)^{-1}(A\cup h(A))\in Closed(X+_{f}Y)$. But $(\psi + \phi)^{-1}(A\cup h(A)) = \psi^{-1}(A) \cup \phi^{-1}(h(A)) = Cl_{X+_{f}Y}(\psi^{-1}(A) \cup \phi^{-1}(h(A))) =$ $Cl_{X+_{f}Y}(\psi^{-1}(A)) \cup Cl_{X+_{f}Y}(\phi^{-1}(h(A)))$. So, we have that $Cl_{X+_{f}Y}(\psi^{-1}(A)) \subseteq$ $\psi^{-1}(A) \cup \phi^{-1}(h(A))$. But $Cl_{X+_{f}Y}(\psi^{-1}(A)) = \psi^{-1}(A) \cup f(\psi^{-1}(A))$ and $f(\psi^{-1}(A) \cap \psi^{-1}(A) = \emptyset$, because $\psi^{-1}(A) \subseteq X$. Thus, $f(\psi^{-1}(A)) \subseteq \phi^{-1}(h(A))$, as we wish to proof.
\end{proof}

\begin{cor}Let $X+_{f}Y$, $X+_{f'}Y$ be topological spaces. Then, the map $id: X+_{f}Y \rightarrow X+_{f'}Y$ is continuous if and only if $\forall A \in Closed (X)$, $f(A) \subseteq f'(A)$.
\eod\end{cor}

\subsection{Composition of sums of spaces}

\begin{prop}Let $X+_{f}W, \ Y$ and $Z$ be topological spaces and let $\Pi: Closed(Y) \rightarrow Closed (X)$ and $\Sigma: Closed(W) \rightarrow Closed(Z)$ be admissible maps. We define $f_{\Sigma\Pi}: Closed(Y) \rightarrow Closed(Z)$ as $f_{\Sigma\Pi} = \Sigma \circ f \circ \Pi$. Then, $f_{\Sigma\Pi}$ is admissible.
\eod\end{prop}

\begin{prop}Let $X+_{f}W, \ Y$ and $Z$ be topological spaces and consider the admissible maps:
\begin{enumerate}
    \item $\Pi: Closed(Y) \rightarrow Closed (X)$,
    \item $\Sigma: Closed(W) \rightarrow Closed(Z)$,
    \item $\Lambda: Closed(X) \rightarrow Closed(Y)$,
    \item $\Omega: Closed(Z) \rightarrow Closed(W)$.
\end{enumerate}

If $\Omega \circ \Sigma \subseteq id_{Closed(W)}$ (respec. $\supseteq id_{Closed(W)}$ or $=  id_{Closed(W)}$) and $\Pi \circ \Lambda \subseteq id_{Closed (X)}$ (respec. $\supseteq id_{Closed(X)}$ or $=  id_{Closed(X)}$) then $(f_{\Sigma \Pi})_{\Omega \Lambda} \subseteq f$ (respec. $\supseteq f$ or $= f$).
\end{prop}

\begin{proof}We have that $(f_{\Sigma \Pi})_{\Omega \Lambda}(A) =  \Omega \circ f_{\Sigma \Pi}\circ \Lambda(A) = \Omega \circ \Sigma \circ f \circ \Pi \circ \Lambda(A)$. If $\Omega \circ \Sigma \subseteq id_{Closed(W)}$ and $\Pi \circ \Lambda \subseteq id_{Closed (X)}$, then $\Omega \circ \Sigma \circ f \circ \Pi \circ \Lambda(A) \subseteq \Omega \circ \Sigma \circ f(A) \subseteq f(A)$. The other cases are analogous.
\end{proof}

\begin{prop}(Cube Lemma) Let $X_{i}+_{f_{i}}W_{i}$, $Y_{i}$ and $Z_{i}$ be topological spaces, $\Pi_{i}: Closed(Y_{i}) \rightarrow Closed(X_{i})$ and $\Sigma_{i}: Closed(W_{i}) \rightarrow Closed(Z_{i})$ admissible maps. Take $f_{i\Sigma_{i}\Pi_{i}}: Closed(Y_{i}) \rightarrow Closed(Z_{i})$ the respective induced maps. If $\mu+\nu: X_{1}+_{f_{1}}W_{1} \rightarrow X_{2}+_{f_{2}}W_{2}$, $\psi: Y_{1} \rightarrow Y_{2}$ and $\phi: Z_{1} \rightarrow Z_{2}$ are continuous maps that form the diagrams:

$$ \xymatrix{   Closed(X_{2}) \ar[r]^{\mu^{-1}} \ar@{}[dr]|{\supseteq} & Closed(X_{1}) & & Closed(W_{2}) \ar[r]^{\nu^{-1}} \ar[d]^{\Sigma_{2}} \ar@{}[dr]|{\supseteq} & Closed(W_{1}) \ar[d]_{\Sigma_{1}} \\
                Closed(Y_{2}) \ar[r]^{\psi^{-1}} \ar[u]^{\Pi_{2}} & Closed(Y_{1}) \ar[u]_{\Pi_{1}} & & Closed(Z_{2}) \ar[r]^{\phi^{-1}} & Closed(Z_{1}) } $$

Then, $\psi+\phi: Y_{1}+_{f_{1\Sigma_{1}\Pi_{1}}}Z_{1}\rightarrow Y_{2}+_{f_{2\Sigma_{2}\Pi_{2}}}Z_{2}$ is continuous.
\end{prop}

\begin{proof}Consider the diagram:

$$ \xymatrix{Closed(X_{1}) \ar[rr]^{f_{1}} & & Closed(W_{1}) \ar[dd]_<<{\Sigma_{1}} \\
            & Closed(X_{2}) \ar[rr]^<<{ \ \ \ \ \ \ \ \ \ \ \ \ \ f_{2}} \ar[lu]^{\mu^{-1}} & & Closed(W_{2}) \ar[lu]_{\nu^{-1}} \ar[dd]_{\Sigma_{2}} \\
             Closed(Y_{1}) \ar[rr]_<<{ \ \ \ \ \ \ \ \ \ \ \ f_{1\Sigma_{1}\Pi_{1}}} \ar[uu]^{\Pi_{1}} & & Closed(Z_{1}) & \\
            & Closed(Y_{2}) \ar[rr]_{f_{2\Sigma_{2}\Pi_{2}}} \ar[uu]^>>{\Pi_{2}} \ar[lu]^{\psi^{-1}} & & Closed(Z_{2}) \ar[lu]_{\phi^{-1}} &} $$

We have that $f_{1\Sigma_{1}\Pi_{1}}\circ\psi^{-1} = \Sigma_{1}\circ f_{1} \circ \Pi_{1}\circ\psi^{-1} \subseteq \Sigma_{1}\circ f_{1}\circ \mu^{-1} \circ \Pi_{2} \subseteq \Sigma_{1}\circ \nu^{-1} \circ f_{2} \circ \Pi_{2} \subseteq \phi^{-1} \circ \Sigma_{2} \circ f_{2} \circ \Pi_{2} = \phi^{-1} \circ f_{2\Sigma_{2}\Pi_{2}}$. Thus, $\psi+\phi$ is continuous.

\end{proof}

\begin{cor}\label{action} Let $G_{1},G_{2}$ be groups, $X+_{f}W$, $Y$ and $Z$ topological spaces, $\alpha: G_{1} \rightarrow G_{2}$ a group homomorphism and $\Pi: Closed(Y) \rightarrow Closed(X)$ and $\Sigma: Closed(W) \rightarrow Closed(Z)$ admissible maps. Take the induced map $f_{\Sigma\Pi}: Closed(Y) \rightarrow Closed(Z)$. If $\mu+\nu: G_{2} \curvearrowright X+_{f}W$, $\psi: G_{1} \curvearrowright Y$ and $\phi: G_{1} \curvearrowright Z$ are actions by homeomorphisms (where $\mu+\nu$ is defined as the disjoint union of the pair of actions $\mu: G_{2} \curvearrowright X$ and $\nu: G_{2} \curvearrowright W$) such that form the following diagrams for each $g \in G_{1}$:

$$ \xymatrix{   Closed(X) \ar[r]^*+{\labelstyle \ \ \mu(\alpha(g),\_)^{-1}} \ar@{}[dr]|{\supseteq} & Closed(X) & & Closed(W) \ar[r]^*+{\labelstyle \ \ \nu(\alpha(g),\_)^{-1}} \ar[d]^{\Sigma} \ar@{}[dr]|{\supseteq} & Closed(W) \ar[d]_{\Sigma} \\
                Closed(Y) \ar[r]^{ \ \psi(g,\_)^{-1}} \ar[u]^{\Pi} & Closed(Y) \ar[u]_{\Pi} & & Closed(Z) \ar[r]^{ \ \phi(g,\_)^{-1}} & Closed(Z) } $$

Then, $\psi+\phi: G_{1} \curvearrowright Y+_{f_{\Sigma \Pi}}Z$ is an action by homeomorphisms.
\end{cor}

\begin{proof}It follows by the last proposition that $\forall g \in G_{1}, \ (\psi+\phi)(g,\_)$ is continuous (and then a homeomorphism, since its inverse is $(\psi+\phi)(g^{-1},\_)$, which is also continuous). Thus, $\psi+\phi: G_{1} \curvearrowright Y+_{f_{\Sigma \Pi}}Z$ is an action by homeomorphisms.
\end{proof}

And a useful proposition about separation:

\begin{prop}Let $X+_{f}W, \ Y$ and $Z$ be Hausdorff topological spaces such that $Y$ and $X$ are locally compact and $\Pi: Closed(Y) \rightarrow Closed (X)$, $\Sigma: \ Closed(W) \ \rightarrow \ Closed(Z)$, $\ \Lambda: \ Closed(X) \ \rightarrow \ Closed(Y)$ and $\Omega: Closed(Z) \rightarrow Closed(W)$ are admissible maps. If $\forall C \subseteq Y$ compact, $\Pi(C)$ is compact, $\{\{z\}: z \in Z\} \subseteq Im \ \Sigma$, $\forall w \in W$, $w \in \Omega \circ \Sigma(\{w\})$, $\Lambda(X) = Y$ and $(f_{\Sigma\Pi})_{\Omega\Lambda} = f$, then $Y+_{f_{\Sigma\Pi}}Z$ is Hausdorff.
\end{prop}

\begin{proof}Let $C \subseteq Y$ be a compact. We have that $\Pi(C)$ is compact, which implies that $\Sigma \circ f \circ \Pi(C) = \Sigma (\emptyset) = \emptyset$, because $X+_{f}Y$ is Hausdorff.

Let $a,b \in Z$. There exists $C,D \in Closed(W): \ \Sigma (C) = \{a\}$ and $\Sigma(D) = \{b\}$. Take $c \in C$ and $d \in D$. Since $X+_{f}Y$ is Hausdorff, there exists $A,B \in Closed(X): \ A \cup B = X, \ d \notin f(A)$ and $c \notin f(B)$. We have that $\Lambda(A), \Lambda(B) \in Closed(Y)$ and $\Lambda(A)\cup \Lambda(B) = \Lambda(A\cup B) = \Lambda(X) = Y$. If $a \in f_{\Sigma\Pi}(\Lambda(B))$, then $\Omega(\{a\}) \subseteq \Omega \circ f_{\Sigma\Pi}(\Lambda(B)) = f(B)$. But $\Omega(\{a\}) = \Omega \circ \Sigma(C) \supseteq \Omega \circ \Sigma(\{c\}) \supseteq \{c\}$, which implies that $c \in f(B)$, a contradiction. So, $a \notin f_{\Sigma\Pi}(\Lambda(B))$ and, analogously, $b \notin f_{\Sigma\Pi}(\Lambda(A))$.

Thus, $Y+_{f_{\Sigma\Pi}}Z$ is Hausdorff.
\end{proof}

\begin{cor}\label{functorausdorff}Let $X+_{f}W$ and $Y$ be Hausdorff topological spaces such that $Y$ and $X$ are locally compact and $\Pi: Closed(Y) \rightarrow Closed (X)$ and  $\Lambda: Closed(X) \rightarrow Closed(Y)$ are admissible maps. If $\forall C \subseteq Y$ compact, $\Pi(C)$ is compact, $\Lambda(X) = Y$ and $(f_{id_{W}\Pi})_{id_{W}\Lambda} = f$, then $Y+_{f_{id_{W}\Pi}}W$ is Hausdorff.
\eod\end{cor}

Let's consider two special cases of composition of sums of spaces that shall be useful:

\subsubsection{Pullbacks}

\begin{defi}Let $X+_{f}W$, $Y$ and $Z$ be topological spaces, $\pi: Y \rightarrow X$ and $\varpi: Z \rightarrow W$ continuous maps. Let's consider $\Pi: Closed(Y) \rightarrow Closed(X)$ and $\Sigma: Closed(W) \rightarrow Closed(Z)$ as $\Pi(A) = Cl_{X}(\pi(A))$ and $\Sigma(A) = \varpi^{-1}(A)$. It is clear that those maps are admissible. We define the pullback of $f$ by the maps $\pi$ and $\varpi$ by $f^{\ast}(A) = f_{\Sigma \Pi}(A) = \varpi^{-1}(f(Cl_{X}(\pi(A))))$.
\end{defi}

\begin{prop}$\pi+\varpi: Y+_{f^{\ast}}Z \rightarrow X+_{f}W$ is continuous.
\end{prop}

\begin{proof}If $A \in Closed(X)$, then $f^{\ast}(\pi^{-1}(A)) = \varpi^{-1}(f(Cl_{X}(\pi(\pi^{-1}(A))))) = \varpi^{-1}(f(Cl_{X}(A))) = \varpi^{-1}(f(A))$. In another words, the diagram commutes:

$$ \xymatrix{   Closed(Y) \ar[r]^{f^{\ast}} & Closed(Z) \\
                Closed(X) \ar[r]_{f} \ar[u]^{\pi^{-1}} & Closed(W) \ar[u]_{\varpi^{-1}} } $$

Thus, $(\pi + \varpi): Y+_{f^{\ast}}Z \rightarrow X+_{f}W$ is continuous.
\end{proof}

\begin{prop}Let $Y+_{f'}Z$, for some admissible map $f'$, such that $\pi+\varpi: Y+_{f'}Z \rightarrow X+_{f}W$ is continuous. Then, $id_{Y}+id_{Z}: Y+_{f'}Z \rightarrow Y+_{f^{\ast}}Z$ is continuous.
\end{prop}

\begin{proof}Since the map $\pi+\varpi$ is continuous, it follows that $\forall B \in Closed(X)$, $f'(\pi^{-1}(B)) \subseteq \varpi^{-1}(f(B)) = f^{\ast}(\pi^{-1}(B))$. We have that $\forall A \in Closed(Y)$,  $f^{\ast}(\pi^{-1}(Cl_{X}(\pi(A)))) = \varpi^{-1}(f(Cl_{X}(\pi(\pi^{-1}(Cl_{X}(\pi(A))))))) =$ $\\ \varpi^{-1}(f(Cl_{X}(Cl_{X}(\pi(A))))) = \varpi^{-1}(f(Cl_{X}(\pi(A)))) = f^{\ast}(A)$. Consider $B = Cl_{X}(\pi(A))$. Then, $f'(\pi^{-1}(Cl_{X}(\pi(A)))) \subseteq f^{\ast}(\pi^{-1}(Cl_{X}(\pi(A)))) = f^{\ast}(A)$. But $A \subseteq \pi^{-1}(Cl_{X}(\pi(A)))$, which implies that $f'(A) \subseteq f'(\pi^{-1}(Cl_{X}(\pi(A))))$. Thus, $f'(A) \subseteq f^{\ast}(A)$, which implies that $id_{Y}+id_{Z}$ is continuous.

\end{proof}

In another words, $f^{\ast}$ induces the coarsest topology (between the topologies that extend the topologies of $Y$ and $Z$) such that the map $\pi+\varpi$ is continuous.

\begin{prop}(Cube Lemma for Pullbacks) Let $X_{i}+_{f_{i}}W_{i},  Y_{i}, Z_{i}$ be topological spaces, $\pi_{i}: Y_{i} \rightarrow X_{i}$ and $\varpi_{i}: Z_{i} \rightarrow W_{i}$ be continuous maps and  $f_{i}^{\ast}: Closed(Y_{i}) \rightarrow Closed(Z_{i})$ the respective pullbacks. Let's suppose that $\mu+\nu: X_{1}+_{f_{1}}W_{1} \rightarrow X_{2}+_{f_{2}}W_{2}$, $\psi: Y_{1} \rightarrow Y_{2}$ and $\phi: Z_{1} \rightarrow Z_{2}$ are continuous maps that commute the diagrams:

$$ \xymatrix{   Y_{1} \ar[r]^{\psi} \ar[d]^{\pi_{1}} & Y_{2} \ar[d]_{\pi_{2}} & & Z_{1} \ar[r]^{\phi} \ar[d]^{\varpi_{1}} & Z_{2} \ar[d]_{\varpi_{2}} \\
                X_{1} \ar[r]^{\mu} & X_{2} & & W_{1} \ar[r]^{\nu} & W_{2} } $$

Then, $\psi+\phi: Y_{1}+_{f_{1}^{\ast}}Z_{1}\rightarrow Y_{2}+_{f_{2}^{\ast}}Z_{2}$ is continuous.
\end{prop}

\begin{proof}We are showing that we have these diagrams:

$$ \xymatrix{   Closed(X_{2}) \ar[r]^{\mu^{-1}} \ar@{}[dr]|{\supseteq} & Closed(X_{1}) & & Closed(W_{2}) \ar[r]^{\nu^{-1}} \ar[d]^{\varpi^{-1}_{2}} \ar@{}[dr]|{\circlearrowleft} & Closed(W_{1}) \ar[d]_{\varpi^{-1}_{1}} \\
                Closed(Y_{2}) \ar[r]^{\psi^{-1}} \ar[u]^{Cl_{X_{2}} \circ \pi_{2}} & Closed(Y_{1}) \ar[u]_{Cl_{X_{1}} \circ \pi_{1}} & & Closed(Z_{2}) \ar[r]^{\phi^{-1}} & Closed(Z_{1}) } $$

Let $A \subseteq X_{2}$. So, $A \subseteq Cl_{X_{2}}(A)$, which implies that $\mu^{-1}(A) \subseteq \mu^{-1}(Cl_{X_{2}}(A))$. Since $\mu^{-1}(Cl_{X_{2}}(A))$ is a closed subset of $X_{1}$, it follows that $Cl_{X_{1}}(\mu^{-1}(A)) \subseteq \mu^{-1}(Cl_{X_{2}}(A))$. Let $B \subseteq Y_{2}$ and $x \in \psi^{-1}(B)$. We have that $\pi_{2}\circ \psi(x) \in \pi_{2}(B)$. But $\pi_{2}\circ \psi(x) = \mu \circ \pi_{1}(x)$, which implies that $x \in \pi_{1}^{-1} \circ \mu^{-1} \circ \pi_{2}(B)$. So, $\psi^{-1}(B) \subseteq \pi_{1}^{-1}\circ \mu^{-1} \circ \pi_{2}(B)$. Let now $C \in Closed(Y_{2})$. We have that $Cl_{X_{1}}(\pi_{1}(\psi^{-1}(C))) \subseteq Cl_{X_{1}}(\pi_{1}(\pi_{1}^{-1}(\mu^{-1}(\pi_{2}(C))))) =$ $Cl_{X_{1}}(\mu^{-1}(\pi_{2}(C))) \subseteq \mu^{-1}(Cl_{X_{2}}(\pi_{2}(C)))$. So, we have the first diagram.

The second diagram is immediate from the hypothesis.

By the Cube Lemma, it follows that $\psi+\phi: Y_{1}+_{f_{1}^{\ast}}Z_{1}\rightarrow Y_{2}+_{f_{2}^{\ast}}Z_{2}$ is continuous.

\end{proof}

\begin{cor}\label{pullbackaction}Let $G_{1}$ and $G_{2}$ be groups, $X+_{f}W, Y$ and $Z$ topological spaces, $\alpha: G_{1}\rightarrow G_{2}$ a homomorphism, $\pi: Y \rightarrow X$ and $\varpi: Z \rightarrow W$ continuous maps and $f^{\ast}: Closed(Y) \rightarrow Closed(Z)$ the pullback of $f$. Let $\mu+\nu: G_{2} \curvearrowright X+_{f}W$, $\psi: G_{1} \curvearrowright Y$ and $\phi: G_{1} \curvearrowright Z$ be actions by homeomorphisms such that $\pi$ and $\varpi$ are $\alpha$ - equivariant. Then, the action $\psi+\phi: G_{1} \curvearrowright Y+_{f^{\ast}}Z$ is by homeomorphisms.
\end{cor}

\begin{proof}Since $\pi$ and $\varpi$ are $\alpha$ - equivariant, we have that $\forall g \in G$ the diagrams commute:

$$ \xymatrix{   Y \ar[r]^{\psi(g,\_)} \ar[d]_{\pi} & Y \ar[d]^{\pi} & & Z \ar[r]^{\phi(g,\_)} \ar[d]_{\varpi} & Z \ar[d]^{\varpi} \\
                X \ar[r]_{\mu(\alpha(g),\_)} & X & & W \ar[r]_{\nu(\alpha(g),\_)} & W } $$

By the last proposition it follows that $(\psi+\phi)(g,\_) = \psi(g,\_)+\phi(g,\_)$ is continuous. Thus, $\psi+\phi$ is an action by homeomorphisms.

\end{proof}

There are some properties of the pullback:

\begin{prop}If $X+_{f}W,Y,Z$ are Hausdorff and $\varpi$ is injective, then $Y+_{f^{\ast}}Z$ is Hausdorff.
\end{prop}

\begin{proof}Let $x,y \in Y$. Since $Y$ is Hausdorff, there exists $U,V$ open sets that separate $x$ and $y$. But $Y$ is open in $Y+_{f^{\ast}}Z$, which implies that $U,V$ are open sets in $Y+_{f^{\ast}}Z$ that separate $x$ and $y$. Let $x \in Y$ and $y \in Z$. Take $U,V$ open sets in $X+_{f}W$ that separate $\pi(x)$ and $\varpi(y)$ (which are different points since $\pi(x) \in X$ and $\varpi(y) \in W$). So, $(\pi+\varpi)^{-1}(U)$ and $(\pi+\varpi)^{-1}(V)$ separate $x$ and $y$. Now, let $x,y \in Z$. Since $\varpi$ is injective, we have that $\varpi(x) \neq \varpi(y)$ and, since $X+_{f}W$ is Hausdorff, there exists $U$ and $V$ disjoint open sets in $X+_{f}W$ that separate $\varpi(x)$ and $\varpi(y)$. Hence, $(\pi+\varpi)^{-1}(U)$ and $(\pi+\varpi)^{-1}(V)$ are disjoint open sets in $Y+_{f^{\ast}}Z$ that separate $x$ and $y$. Thus, $Y+_{f^{\ast}}Z$ is Hausdorff.
\end{proof}

\begin{prop}If $\pi$ and $\varpi$ are closed, then $\pi+\varpi: Y+_{f^{\ast}}Z \rightarrow X+_{f}W$ is closed.
\end{prop}

\begin{proof}Let $A\in Closed(Y+_{f^{\ast}}Z)$. So, $A\cap Y \in Closed(Y)$, $A\cap Z \in Closed(Z)$ and $f^{\ast}(A\cap Y) \subseteq A$. Since $\pi$ and $\varpi$ are closed, we have that $\pi(A \cap Y) = (\pi+\varpi)(A)\cap X \in Closed(X)$, $\varpi(A \cap Z) = (\pi+\varpi)(A)\cap W \in$ $Closed(W)$ and $(\pi+\varpi)(f^{\ast}(A\cap Y)) = \varpi(f^{\ast}(A\cap Y)) \subseteq (\pi+\varpi)(A)$. But we have that $f^{\ast}(A\cap Y) = \varpi^{-1}(f(\pi(A\cap Y)))$, which implies that $\varpi(f^{\ast}(A\cap Y)) =$ $\varpi(\varpi^{-1}(f(\pi(A\cap Y)))) = f(\pi(A\cap Y)) = f((\pi+\varpi)(A)\cap X)$. Therefore $f((\pi+\varpi)(A)\cap X) \subseteq (\pi+\varpi)(A)$. Thus,  $(\pi+\varpi)(A) \in Closed(X+_{f}W)$ and then $\pi+\varpi$ is closed.
\end{proof}

\begin{prop}If $\pi$ is proper and the spaces $Z$ and $X+_{f}W$ are compact, then $Y+_{f^{\ast}}Z$ is compact.
\end{prop}

\begin{proof}Let $A \in Closed(Y)$ be non compact. Then, $Cl_{X}(\pi(A))$ is not compact (otherwise $\pi^{-1}(Cl_{X}(\pi(A))$ would be compact, since $\pi$ is proper, that contain a closed non compact subspace $A$). So, $f(Cl_{X}(\pi(A))) \neq \emptyset$ (since $Y+_{f^{\ast}}Z$ is compact), which implies that $f^{\ast}(A) = \varpi^{-1}(f(Cl_{X}(\pi(A)))) \neq \emptyset$. Thus, $Y+_{f^{\ast}}Z$ is compact.
\end{proof}

\begin{prop}Let $X+_{f}Y$ be a topological space, $X_{1} \subseteq X$ and $Y_{1} \subseteq Y$. So, the subspace topology of $Z = X_{1} \cup Y_{1}$ coincides with the topology of the space $X_{1}+_{f_{1}}Y_{1}$, with $f_{1}: Closed(X_{1}) \rightarrow Closed (Y_{1})$ such that $f_{1}(A) = f(Cl_{X}(A))\cap Y_{1}$ is the pullback of $f$ by the inclusion maps.
\end{prop}

\begin{proof}Let $\iota_{X_{1}}: X_{1} \rightarrow X, \ \iota_{Y_{1}}: Y_{1} \rightarrow Y$ and $\iota_{Z}: Z \rightarrow X+_{f}Y$ be the inclusion maps. We have that the maps $\iota_{X_{1}}+\iota_{Y_{1}}: X_{1}+_{f_{1}}Y_{1} \rightarrow X+_{f}Y$ and $\iota_{Z}: Z \rightarrow X+_{f}Y$ are both continuous. By the universal property of the subspace topology we have that $id_{X_{1}\cup Y_{1}}: X_{1}+_{f_{1}}Y_{1} \rightarrow Z$ is continuous and, by the universal property of the pullback, we have the continuity of the inverse. Thus, both topologies coincide.
\end{proof}

\subsubsection{Pushforwards}
\begin{defi}Let $X+_{f}W,Y,Z$ be topological spaces and $\pi: X \rightarrow Y$ and $\varpi: W \rightarrow Z$ continuous maps. Let's consider the admissible maps $\Pi: Closed(Y) \rightarrow Closed(X)$ and $\Sigma: Closed(W) \rightarrow Closed(Z)$ as $\Pi(A) = \pi^{-1}(A)$ and $\Sigma(A) = Cl_{Z}(\varpi(A))$. We define the pushforward of $f$ by the maps $\pi$ and $\varpi$ by $f_{\ast}(A) = f_{\Sigma \Pi}(A) = Cl_{Z}(\varpi(f(\pi^{-1}(A))))$.
\end{defi}

\begin{prop}$\pi+\varpi: X+_{f}W \rightarrow Y+_{f_{\ast}}Z$ is continuous.
\end{prop}

\begin{proof}If $A \in Closed(Y)$, then $\varpi^{-1}(f_{\ast}(A)) = \varpi^{-1}(Cl_{Z}(\varpi(f(\pi^{-1}(A))))) \supseteq \varpi^{-1}(\varpi(f(\pi^{-1}(A)))) \supseteq f(\pi^{-1}(A))$, that is, we have the diagram:

$$ \xymatrix{   Closed(X) \ar[r]^{f} \ar@{}[dr]|{\subseteq} & Closed(W) \\
                Closed(Y) \ar[r]_{f_{\ast}} \ar[u]^{\pi^{-1}} & Closed(Z) \ar[u]_{\varpi^{-1}} } $$

Thus, $\pi + \varpi: X+_{f}W \rightarrow Y+_{f_{\ast}}Z$ is continuous.
\end{proof}

\begin{prop}If $\varpi$ is closed and injective, then the diagram commutes:

$$ \xymatrix{   Closed(X) \ar[r]^{f} & Closed(W) \\
                Closed(Y) \ar[r]_{f_{\ast}} \ar[u]^{\pi^{-1}} & Closed(Z) \ar[u]_{\varpi^{-1}} } $$

\eod\end{prop}

\begin{prop}Let $Y+_{f'}Z$ for some choice of $f'$ such that the map $\pi+\varpi: X+_{f}W \rightarrow Y+_{f'}Z$ is continuous. If $\varpi$ is injective and closed, then $id_{X}+id_{W}: X+_{f_{\ast}}W \rightarrow X+_{f}W$ is continuous.
\end{prop}

\begin{proof}Since the map $\pi+\varpi$ is continuous, we have that $\forall B \in Closed(Y)$, $f(\pi^{-1}(B)) \subseteq \varpi^{-1}(f'(B))$. Since $\varpi$ is injective and closed, we have that $f(\pi^{-1}(B)) = \varpi^{-1}(f_{\ast}(B))$, which implies that $\varpi^{-1}(f_{\ast}(B)) \subseteq \varpi^{-1}(f'(B))$ and then $f_{\ast}(B) \subseteq f'(B)$. Thus, $id_{X}+id_{W}: X+_{f_{\ast}}W \rightarrow X+_{f}W$ is continuous.

\end{proof}

In another words, if $\varpi$ is injective and closed, then $f_{\ast}$ induces the finer topology (between the topologies that extend the topologies of $X$ and $W$) such that the map $\pi+\varpi$ is continuous.

\begin{prop}(Cube Lemma for Pushforwards) Let $X_{i}+_{f_{i}}W_{i}$, $Y_{i}$ and $Z_{i}$ be topological spaces, $\pi_{i}: X_{i} \rightarrow Y_{i}$ and $\varpi_{i}: W_{i} \rightarrow Z_{i}$ continuous maps and $f_{i \ast}: Closed(Y_{i}) \rightarrow Closed(Z_{i})$ the respective pushforwards. If the maps $\mu+\nu: X_{1}+_{f_{1}}W_{1} \rightarrow X_{2}+_{f_{2}}W_{2}$, $\psi: Y_{1} \rightarrow Y_{2}$ and $\phi: Z_{1} \rightarrow Z_{2}$ are continuous and commute the diagrams:

$$ \xymatrix{   X_{1} \ar[r]^{\mu} \ar[d]^{\pi_{1}} & X_{2} \ar[d]_{\pi_{2}} & & W_{1} \ar[r]^{\nu} \ar[d]^{\varpi_{1}} & W_{2} \ar[d]_{\varpi_{2}} \\
                Y_{1} \ar[r]^{\psi} & Y_{2} & & Z_{1} \ar[r]^{\phi} & Z_{2} } $$

Then, $\psi+\phi: Y_{1}+_{f_{1\ast}}Z_{1}\rightarrow Y_{2}+_{f_{2\ast}}Z_{2}$ is continuous.
\end{prop}

\begin{proof}Let $A \subseteq Z_{2}$. Since $\phi$ is continuous, $Cl_{Z_{1}}\phi^{-1}(A) \subseteq \phi^{-1}(Cl_{Z_{2}}(A))$. Let $B \subseteq W_{2}$ and $x \in \varpi_{1}(\nu^{-1}(B))$. So, $\phi(x) \in \phi(\varpi_{1}(\nu^{-1}(B))) = \varpi_{2}(\nu(\nu^{-1}(B))) = \varpi_{2}(B)$, which implies that $x \in \phi^{-1}(\varpi_{2}(B))$. Hence $\varpi_{1}(\nu^{-1}(B)) \subseteq \phi^{-1}(\varpi_{2}(B))$. Let $C \in Closed(W_{2})$. We have that $Cl_{Z_{1}}(\varpi_{1}(\nu^{-1}(C))) \subseteq Cl_{Z_{1}}(\phi^{-1}(\varpi_{2}(C))) \subseteq \phi^{-1}(Cl_{Z_{2}}(\varpi_{2}(C)))$. So, we have the diagrams (the first one is immediate from the hypothesis):

$$ \xymatrix{   Closed(X_{2}) \ar[r]^{\mu^{-1}} \ar@{}[dr]|{\circlearrowleft} & Closed(X_{1}) & & Closed(W_{2}) \ar[r]^{\nu^{-1}} \ar[d]_{Cl_{Z_{2}} \circ \varpi_{2}} \ar@{}[dr]|{\supseteq} & Closed(W_{1}) \ar[d]^{Cl_{Z_{1}} \circ \varpi_{1}} \\
                Closed(Y_{2}) \ar[r]^{\psi^{-1}} \ar[u]_{\pi^{-1}_{2}} & Closed(Y_{1}) \ar[u]^{\pi^{-1}_{1}} & & Closed(Z_{2}) \ar[r]^{\phi^{-1}} & Closed(Z_{1}) } $$

By the Cube Lemma, it follows that $\psi+\phi: Y_{1}+_{f_{1\ast}}Z_{1} \rightarrow Y_{2}+_{f_{2\ast}}Z_{2}$ is continuous.
\end{proof}

And a property of the pushforward:

\begin{prop}If $X+_{f}W$ and $Z$ are compact, then $Y+_{f_{\ast}}Z$ is compact.
\end{prop}

\begin{proof}Let $A \in Closed(Y)$ be non compact. Then, $\pi^{-1}(A)$ is not compact (otherwise $A = \pi(\pi^{-1}(A))$ would be compact). So, $f(\pi^{-1}(A)) \neq \emptyset$, which implies that $f_{\ast}(A) = Cl_{Z}(\varpi(f(\pi^{-1}(A)))) \neq \emptyset$. Thus, $Y+_{f_{\ast}}Z$ is compact.
\end{proof}

\subsection{Limits}

\begin{defi}Let $X$ be a locally compact space. We define $SUM(X)$ as the category whose objects are Hausdorff spaces of the form $X+_{f}Y$ and morphisms are continuous maps of the form $id+\phi: X+_{f_{1}}Y_{1} \rightarrow X+_{f_{2}}Y_{2}$.
\end{defi}

In this section we are going to construct the limits of this category.

\begin{prop}The one point compactification $X+_{f_{\infty}}\{\infty\}$ is the terminal object in $SUM(X)$.
\end{prop}

\begin{proof}Let $X+_{f}Y \in SUM(X)$. It is clear that if there exists a morphism $id+\phi: X+_{f}Y \rightarrow X+_{f_{\infty}}\{\infty\}$, it must be unique ($\phi$ must be the constant map). Let's check that such map is continuous (and then a morphism). Let $F \in Closed(X)$. If $F$ is compact, then $f(F) = f_{\infty}(F) = \emptyset$, which implies that $f \circ id^{-1}(F) = \emptyset = \phi^{-1} \circ f_{\infty}(F)$. If $F$ is not compact, then $f_{\infty}(F) = \{\infty\}$, which implies that $f \circ id^{-1}(F) = f(F) \subseteq Y = \phi^{-1}(\infty) = \phi^{-1} \circ f_{\infty}(F)$. Thus, $id+\phi$ is continuous and then $X+_{f_{\infty}}\{\infty\}$ is the terminal object in $Sum(X)$.
\end{proof}

\begin{prop}Let $\mathcal{C}$ be a category. We define a new category $\hat{\mathcal{C}}$ whose objects are the same as $\mathcal{C}$ and one new object $\infty$ and the morphisms are the same as $\mathcal{C}$ and for each $c \in \hat{\mathcal{C}}$, a new morphism $e_{c}: c \rightarrow \infty$. For morphisms in $\mathcal{C}$, the new composition is the same. For a morphism $\alpha: c_{1} \rightarrow c_{2}$ of $\mathcal{C}$, we define $e_{c_{2}} \circ \alpha = e_{c_{1}}$. And, finally, we define, for $c \in \mathcal{C}, \ e_{\infty} \circ e_{c} = e_{c}$. This becomes actually a category.
\end{prop}

\begin{proof}We have that the identity of an object in $\mathcal{C}$ continuous to be its identity on the new category and $id_{\infty} = e_{\infty}$. Let $\alpha_{i}: c_{i} \rightarrow c_{i+1}$ be morphisms in $\mathcal{C}$. We have that $(\alpha_{3} \circ \alpha_{2}) \circ \alpha_{1} = \alpha_{3} \circ (\alpha_{2} \circ \alpha_{1})$, since this compositions are just the same as in $\mathcal{C}, \ (e_{c_{3}}\circ \alpha_{2}) \circ \alpha_{1} = e_{c_{2}} \circ \alpha_{1} = e_{c_{1}} = e_{c_{3}} \circ (\alpha_{2} \circ \alpha_{1}), \ (e_{\infty} \circ e_{c_{2}}) \circ \alpha_{1} = e_{c_{2}} \circ \alpha_{1} = e_{c_{1}} = e _{\infty} \circ e_{c_{1}} = e_{\infty} \circ (e_{c_{2}} \circ \alpha_{1}), \ (e_{\infty} \circ e_{\infty}) \circ e_{c_{1}} = e_{\infty} \circ e_{c_{1}} = e_{\infty} \circ (e_{\infty} \circ e_{c_{1}})$ and $(e_{\infty} \circ e_{\infty}) \circ e_{\infty} = e_{\infty} = e_{\infty} \circ (e_{\infty} \circ e_{\infty})$. Thus, the composition is associative, which implies that $\hat{\mathcal{C}}$ is a category.
\end{proof}

Let $F: \mathcal{C} \rightarrow SUM(X)$ be a covariant functor, where $\mathcal{C}$ is a small category. We have that $\forall c \in \mathcal{C}, \ F(c) = X+_{f_{c}}Y_{c}$.

\begin{prop}There exists only one extension $\hat{F}: \hat{\mathcal{C}} \rightarrow SUM(X)$ of $F$ such that $\hat{F}(\infty) = X+_{f_{\infty}}\{\infty\}$.
\end{prop}

\begin{proof}It is clear that, if such extension exists, it must be unique, since it is already defined on the objects, on the morphisms in $\mathcal{C}$ and for the morphisms $e_{c}: c \rightarrow \infty$, it must be the unique morphism of the form $\hat{F}(c) \rightarrow \infty$. Let's check that $\hat{F}$ is actually a functor. If $c\in \mathcal{C}$, then $\hat{F}(id_{c}) = F_{id_{c}} = id_{F(c)}$ and $\hat{F}(id_{\infty}) = id_{X+_{f_{\infty}}\{\infty\}}$, by definition of $\hat{F}$. Let $\alpha_{i}: \ c_{i} \ \rightarrow \ c_{i+1}$ and $e_{c_{i}}: c_{i} \rightarrow \infty$ be morphisms in $\hat{\mathcal{C}}$, with $c_{i} \in \mathcal{C}$ (there is no morphism of the form $\gamma: \infty \rightarrow c_{i}$ and the only morphism $\gamma: \infty \rightarrow \infty$ is the identity). We have that $\hat{F}(\alpha_{2}\circ \alpha_{1}) = F(\alpha_{2}\circ \alpha_{1}) = F(\alpha_{2}) \circ F(\alpha_{1}) = \hat{F}(\alpha_{2}) \circ \hat{F}(\alpha_{1})$, and $\hat{F}(e_{c_{2}} \circ \alpha_{1}) = \hat{F}(e_{c_{1}}) = \hat{F}(e_{c_{2}}) \circ \hat{F}(\alpha_{1})$, since both are the unique morphism $F(c_{1}) \rightarrow X+_{f_{\infty}}\{\infty\}$. Thus, $\hat{F}$ is a functor.
\end{proof}

Let $\iota_{c}: X \rightarrow X+_{f_{c}}Y_{c}$ be the inclusion map. It is clear that $(X,\{\iota_{c}\}_{c\in \mathcal{C}})$ is a cone of the functor $\tilde{F}: \tilde{\mathcal{C}} \rightarrow Top$  that does the same as $\hat{F}$. So, it induces the diagonal map $\Delta: X \rightarrow \lim\limits_{\longleftarrow} \tilde{F}$.

\begin{prop}$\Delta$ is an open embedding.
\end{prop}

\begin{proof}Let $\pi_{c}: \lim\limits_{\longleftarrow} \tilde{F} \rightarrow X+_{f_{c}}Y_{c}$ be the projections and the equivalence relation $x\sim y$, for $x,y \in \lim\limits_{\longleftarrow} \tilde{F}$ if $x = y$ or $\pi_{\infty}(x) = \pi_{\infty}(y) \in X$. In the case, $\forall c \in \hat{\mathcal{C}}, \ \hat{F}(e_{c}) \circ \pi_{c}(x) = \pi_{\infty}(x) = \pi_{\infty}(y) = \hat{F}(e_{c}) \circ \pi_{c}(y)$, which implies that $\pi_{c}(x) = \pi_{c}(y)$, since $\hat{F}(e_{c})|_{X}$ is injective (in a fact it is the identity in $X$). So, the diagram commutes $\forall c,c' \in \hat{\mathcal{C}}$ and $\alpha: c \rightarrow c'$ morphism:

$$ \xymatrix{   & \lim\limits_{\longleftarrow} \tilde{F} \ar[ldd]_{\pi_{c}} \ar[rdd]^{\pi_{c'}} \ar[d]^{\omega} &  \\
                  & (\lim\limits_{\longleftarrow} \tilde{F})/ \!\! \sim \ar[ld]^{\omega_{c}} \ar[rd]_{\omega_{c'}} & \\ \tilde{F}(c) \ar[rr]^{\alpha} & & \tilde{F}(c') } $$

Where $\omega: \lim\limits_{\longleftarrow} \tilde{F} \rightarrow (\lim\limits_{\longleftarrow} \tilde{F})/\sim$ is the quotient map and $\omega_{c}$ and $\omega_{c'}$ are the maps that commutes the upper triangles (they are continuous because of the quotient topology). So, the whole diagram commutes, which implies, by the universal property of the limit, that $\omega$ is an homeomorphism, which implies that $\sim$ is trivial and then $\pi_{\infty}|_{\pi_{\infty}^{-1}(X)}$ is injective.

Since $\pi_{\infty} \circ \Delta = \iota_{\infty}$ and $\pi_{\infty}$ is injective, it follows that $\Delta$ is injective. Let $U \subseteq X$ be an open set. We have that $\pi_{\infty}^{-1}(U)$ is open in $\lim\limits_{\longleftarrow} \tilde{F}$. But $\pi_{\infty}^{-1}(U) = \Delta(U)$, since $\pi_{\infty} \circ \Delta(U) = U$ and $\pi_{\infty}|_{\pi_{\infty}^{-1}(X)}$ is injective. So, $\Delta(U)$ is open. Thus, $\Delta$ is open.

\end{proof}

So, $\lim\limits_{\longleftarrow} \tilde{F} \cong X+_{f}Y$ for some topological space $Y$ and an admissible map $f$. Thus, $\lim\limits_{\longleftarrow} \tilde{F}$ can be seen as an object of $SUM(X)$.

\begin{prop}\label{limitesegundotermo}Let $\acute{F}: \hat{C} \rightarrow Top$ defined by $\acute{F}(c) = Y_{c}$ and, for $\alpha: c \rightarrow d$ a morphism, $\acute{F}(\alpha) = \tilde{F}(\alpha)|_{Y_{c}}: Y_{c} \rightarrow Y_{d}$. Then, $Y \cong \lim\limits_{\longleftarrow} \acute{F}$.
\end{prop}

\begin{proof}Let, for $c \in \hat{\mathcal{C}}, \ \nu_{c}: \acute{F}(c) \rightarrow \tilde{F}(c)$ be the inclusion map. By the definition of $\acute{F}$, we have that $\{v_{c}\}_{c\in \hat{\mathcal{C}}}$ is a natural transformation, which implies that it induces a continuous map $\nu: \lim\limits_{\longleftarrow} \acute{F} \rightarrow \lim\limits_{\longleftarrow} \tilde{F}$.

Let $x,y \in \lim\limits_{\longleftarrow} \acute{F}$ such that $\nu(x) = \nu(y)$. Then, $\forall c \in \hat{\mathcal{C}}, \ \pi_{c} \circ \nu(x) = \pi_{c} \circ \nu(y)$. But $\pi_{c} \circ \nu = \nu_{c} \circ \varpi_{c}$, where $\varpi_{c}: \lim\limits_{\longleftarrow} \acute{F} \rightarrow \acute{F}(c)$ is the projection map. So,  $\forall c \in \hat{\mathcal{C}}, \ \nu_{c} \circ \varpi_{c}(x) =  \nu_{c} \circ \varpi_{c}(y)$, which implies that $\forall c \in \hat{\mathcal{C}}, \ \varpi_{c}(x) = \varpi_{c}(y)$, since $\nu_{c}$ is injective. So, $x = y$ and then $\nu$ is injective.

Let $x \notin Im \ \Delta$. If there exists $c_{0} \in \hat{\mathcal{C}}: \pi_{c_{0}}(x) \in X$, then $\pi_{\infty}(x) =  \pi_{c_{0}}(x) \in X$, which implies that $\forall c \in \hat{\mathcal{C}}, \ \pi_{c}(x) \in X$, contradicting the fact that $x \notin Im \ \Delta$. So, $\forall c \in \hat{\mathcal{C}}, \ \pi_{c}(x) \in Y_{c}$. Since $\forall \alpha: c \rightarrow d, \ \acute{F}(\alpha)(\pi_{c}(x)) = \tilde{F}(\alpha)(\pi_{c}(x)) = \pi_{d}(x)$, there exists $y \in \lim\limits_{\longleftarrow} \acute{F}$, such that $\forall c \in \hat{\mathcal{C}}, \ \varpi_{c}(y) = \pi_{c}(x)$. So, $\forall c \in \hat{\mathcal{C}}, \ \pi_{c} \circ \nu(y) = \nu_{c} \circ \varpi_{c}(y) = \varpi_{c}(y) = \pi_{c}(x)$, which implies that $\nu(y) = x$ and then $(\lim\limits_{\longleftarrow} \tilde{F}) - Im \ \Delta \subseteq Im \ \nu$. Let $x \in Im \ \nu$ and $y \in \lim\limits_{\longleftarrow} \acute{F}: \nu(y) = x$. Then, $\forall c \in \hat{\mathcal{C}}, \ \pi_{c}(x) = \pi_{c} \circ \nu(y) = \nu_{c} \circ \varpi_{c}(y) = \varpi_{c}(y) \in Y_{c}$, which implies that $x \notin Im \ \Delta$. So, $Im \ \nu \subseteq (\lim\limits_{\longleftarrow} \tilde{F}) - Im \ \Delta$ and then $Im \ \nu = (\lim\limits_{\longleftarrow} \tilde{F}) - Im \ \Delta$.

Since $\lim\limits_{\longleftarrow} \acute{F}$ is compact and $\nu$ is injective, we have that $\lim\limits_{\longleftarrow} \acute{F} \cong Im \ \nu = (\lim\limits_{\longleftarrow} \tilde{F}) - Im \ \Delta \cong Y$.
\end{proof}

\begin{prop}$\lim\limits_{\longleftarrow} \hat{F}$ exists and $\lim\limits_{\longleftarrow} \tilde{F} \cong \lim\limits_{\longleftarrow} \hat{F}$.
\end{prop}

\begin{proof}We have that $\lim\limits_{\longleftarrow} \tilde{F}$ satisfies the limit conditions since it satisfies the conditions in $Top$ with more morphisms.
\end{proof}

\begin{prop}$\lim\limits_{\longleftarrow} F$ exists and $\lim\limits_{\longleftarrow} \hat{F} \cong \lim\limits_{\longleftarrow} F$.
\end{prop}

\begin{proof}Both functors have the same cones because they agree in $\mathcal{C}$ and $\hat{F}(\infty)$ is the terminal object in $SUM(X)$. So, they have the same limit.
\end{proof}

Observe that $X$ does not need to be dense on the limit, even when $X$ is dense in $F(c)$, $\forall c \in \mathcal{C}$:

\begin{ex}Consider two copies of the two point compactification space $\R+_{f}\{-\infty,\infty\}$. The product is a compact of the form $\R+_{g}Y$, with $\# Y = 4$, since $Y \cong \{-\infty,\infty\} \times \{-\infty,\infty\}$. But there is no Hausdorff compactification of $\R$ with four points. Thus, $\R$ is not dense in $\R+_{g}Y$.
\end{ex}

So, let's consider $Sum(X)$ the full subcategory of $SUM(X)$ whose objects are the spaces where $X$ is dense.

\begin{prop}Let $I: Sum(X) \rightarrow SUM(X)$ be the inclusion functor and $J: SUM(X) \rightarrow Sum(X)$ the functor that sends a space $X+_{f}Y$ to $Cl_{X+_{f}Y}(X)$ and a map $id+\phi: X+_{f_{1}}Y_{1} \rightarrow X+_{f_{2}}Y_{2}$ to its restriction $id+\phi|_{f_{1}(X_{1})}: Cl_{X+_{f_{1}}Y_{1}}(X) \rightarrow Cl_{X+_{f_{2}}Y_{2}}(X)$. Then, $J$ is right adjoint to $I$.
\end{prop}

\begin{proof}Let $X+_{f}Y \in Sum(X)$, $X+_{g}Z \in SUM(X)$ and an application $id+\phi: X+_{f}Y \rightarrow J(X+_{g}Z) = Cl_{X+_{g}Z}(X)$. If $\iota:  Cl_{X+_{g}Z}(X) \rightarrow  X+_{g}Z$ is the inclusion map, then $\iota \circ (id+\phi): I(X+_{f}Y) = X+_{f}Y \rightarrow X+_{g}Z$ is the only morphism that commutes the diagram:

$$ \xymatrix{ X+_{f}Y \ar[r]^-{id}  \ar[d]^{id+\phi} & J \circ I(X+_{f}Y)  \ar[ld]^{ \ J(\iota\circ(id+\phi))} \\
            J(X+_{g}Z) & } $$

So, $X+_{f}Y$ and the map $id+id: X+_{f}Y \rightarrow X+_{f}Y$ form a reflection of $X+_{f}Y$ along $J$. Thus, $I$ is left adjoint to $J$.

 \end{proof}

\begin{prop}Let $F: \mathcal{C} \rightarrow Sum(X)$ be a functor, $\breve{F}: \mathcal{C} \rightarrow SUM(X)$ a functor that do the same thing as $F$ and $Z = \lim\limits_{\longleftarrow} \breve{F}$. Then,  $\lim\limits_{\longleftarrow} F$ exists and $Cl_{Z}X = \lim\limits_{\longleftarrow} F$.
\end{prop}

\begin{proof}RAPL.
\end{proof}

\begin{prop}\label{densonolimite}Let $F: \mathcal{C} \rightarrow Sum(X)$ be a functor, where $\mathcal{C}$ is a codirected poset. Then, $X$ is dense in $X+_{f}Y = \lim\limits_{\longleftarrow}\breve{F}$ and $\lim\limits_{\longleftarrow}F \cong \lim\limits_{\longleftarrow}\breve{F}$.
\end{prop}

\begin{proof}Let $y \in Y - f(X)$. Since $Y - f(X)$ is open in $X+_{f}Y$ and the set $\{\pi_{c}^{-1}(U): c \in \mathcal{C}$ and $U$ is open in $F(c)\}$ is a basis for the topology of $X+_{f}Y$ (because $\mathcal{C}$ is codirected), we have that there exists $c \in \mathcal{C}$ and $U$ an open set of $F(c) = X+_{f_{c}}Y_{c}$ such that $y \in \pi^{-1}_{c}(U) \subseteq Y - f(X)$. So, $\pi_{c}(y) \in U$. Let $x \in X \cap U$. We have that $x \in \pi_{c}^{-1}(U) \subseteq Y$, a contradiction. So, $X \cap U = \emptyset$. But $X \subseteq (X+_{f_{c}}Y_{c}) - U$, a closed set, implies that $X \cup f_{c}(X) = Cl_{X+_{f_{c}}Y_{c}}(X) \subseteq (X+_{f_{c}}Y_{c}) - U$. Since $f_{c}(X) = Y_{c}$, it follows that $X+_{f_{c}}Y_{c} \subseteq (X+_{f_{c}}Y_{c}) - U$ and then $U = \emptyset$, a contradiction, since $\pi_{c}(y) \in U$. Thus, $f(X) = Y$ and then $X$ is dense in $X+_{f}Y$. Since $\lim\limits_{\longleftarrow}F \cong Cl_{\lim\limits_{\longleftarrow}\breve{F}}(X)$, it follows that $\lim\limits_{\longleftarrow}F \cong \lim\limits_{\longleftarrow}\breve{F}$
\end{proof}

\begin{cor}\label{limitesegundotermominimal}Let $F: \mathcal{C} \rightarrow Sum(X)$ be a functor, where $\mathcal{C}$ is codirected poset and $X+_{f}Y = \lim\limits_{\longleftarrow}F$. Then, $Y \cong \lim\limits_{\longleftarrow} \acute{F}$.
\eod\end{cor}

\subsection{Freudenthal compactification}

Let $X$ be a Hausdorff connected and locally connected space. We  construct the Freudenthal compactification of $X$ using the language of sum of spaces. For $K \subseteq X$ let's define $\pi_{0}^{u}(X-K)$ as the set of unbounded connected components of $X - K$ with the discrete topology (boundedness here means that its closure in $X$ is compact).

\begin{prop}$\forall K \subseteq X$ compact, $\pi_{0}^{u}(X-K)$ is finite.
\end{prop}

\begin{proof}Let $V$ be an open set such that $K \subseteq V$ and $Cl_{X}(V)$ is compact (it exists since $X$ is locally compact). Let $S = \{U \in \pi_{0}^{u}(X-K): U \cap V \neq \emptyset\}$. Since $X$ is connected, $\partial V \neq \emptyset$. Let, $\forall p \in  \partial V, \ V_{p}$ be an open and connected neighbourhood of $p$ such that $V_{p} \cap K = \emptyset$. Let $\{V_{p_{1}},...,V_{p_{n}}\}$ be a finite subcover of $\partial V$ ($\partial V$ is compact, since $Cl_{X}V$ is compact). If $U \in S$, then there exists $i \in \{1,...,n\}: V_{p_{i}} \subseteq U$. However, for $U\neq U' \in S$ and $i,i'\in\{1,...,n\}: V_{p_{i}} \subseteq U$ and $V_{p_{i'}} \subseteq U'$, we have that $i \neq i'$ (since $U$ and $U'$ are disjoint). So, $S$ is finite.

Let $W \in \pi_{0}^{u}(X-K)-S$. Then, $W \subseteq V \cup (X - Cl_{X}V)$. Since $W$ is connected, we have that $W \subseteq V$ or $W \subseteq X - Cl_{X}V$. Since $V$ is bounded and $W$ is not, it follows that $W \subseteq X - Cl_{X}V$. Since $X-K$ is locally connected, it follows that $W$ is open in $X-K$ and then open in $X$. But $X-W = V \cup \bigcup_{U \in \pi_{0}^{u}(X-K)-\{W\}} U$ is an open set as well, contradicting the fact that $X$ is connected. Thus, $\pi_{0}^{u}(X-K) = S$, which implies that $\pi_{0}^{u}(X-K)$ is finite.
\end{proof}

If $K_{1}, K_{2}$ are two compact subspaces of $X$ with $K_{1} \subseteq K_{2}$, take the map $\psi_{K_{1}K_{2}}: \pi_{0}^{u}(X-K_{2}) \rightarrow \pi_{0}^{u}(X-K_{1})$ defined by $\psi_{K_{1}K_{2}}(U)$ as the connected component of $U$ in $\pi_{0}^{u}(X-K_{1})$, for $U \in \pi_{0}^{u}(X-K_{2})$. We have that, for $K_{1} \subseteq K_{2} \subseteq K_{3} \subseteq X, \ \psi_{K_{1}K_{2}} \circ \psi_{K_{2}K_{3}} = \psi_{K_{1}K_{3}}$. So we are able to define the end space of $X$ as $Ends(X) = \lim\limits_{\longleftarrow} \pi_{0}^{u}(X-K)$.

Let's consider, for $K \subseteq X$ a compact, the space $X+_{f_{K}}\pi_{0}^{u}(X-K)$ with $f_{K}(F) = \{U \in \pi_{0}^{u}(X-K): U \cap F$ is unbounded$\}$.

\begin{prop}$f_{K}$ is admissible.
\end{prop}

\begin{proof}We have that $f_{K}(\emptyset) = \{U \in \pi_{0}^{u}(X-K): U \cap \emptyset$ is unbounded$\} = \emptyset$. Let $F_{1},F_{2} \in Closed(X)$. If $U \in f_{K}(F_{1})$, then $U \cap F_{1}$ is unbounded, which implies that $U \cap (F_{1}\cup F_{2})$ is unbounded and then $U \in f_{K}(F_{1} \cup F_{2})$. Analogously, $f_{K}(F_{2}) \subseteq f_{K}(F_{1}\cup F_{2})$, which implies that $f_{K}(F_{1}) \cup f_{K}(F_{2}) \subseteq f_{K}(F_{1}\cup F_{2})$. If $U \in \pi_{0}^{u}(X-K) - (f_{K}(F_{1}) \cup f_{K}(F_{2}))$, then $U \cap F_{1}$ and $U \cap F_{2}$ are bounded. Hence $(U \cap F_{1}) \cup (U \cap F_{2}) = U \cap (F_{1} \cup F_{2})$ is bounded, which implies that $U \notin f_{K}(F_{1}\cup F_{2})$. So, $f_{K}(F_{1}\cup F_{2}) \subseteq f_{K}(F_{1}) \cup f_{K}(F_{2})$, which implies that $f_{K}(F_{1}\cup F_{2}) = f_{K}(F_{1}) \cup f_{K}(F_{2})$. Thus, $f_{K}$ is admissible.
\end{proof}

\begin{prop}$X+_{f_{K}}\pi_{0}^{u}(X-K)$ is compact.
\end{prop}

\begin{proof}Let $F \in Closed(X)$ be non compact. If $f_{K}(F) = \emptyset$, then $\forall U \in \pi_{0}^{u}(X-K), \ U\cap F$ is bounded, which implies that $K \cup \bigcup\limits_{U \in \pi_{0}^{u}(X-K)} (U\cap F)$ is bounded (since $\pi_{0}^{u}(X-K)$ is finite). But $F \subseteq K \cup \bigcup\limits_{U \in \pi_{0}^{u}(X-K)} (U\cap F)$, contradicting the fact that $F$ is not compact (and then unbounded). Thus, $f_{K}(F) \neq \emptyset$, which implies that $X+_{f_{K}}\pi_{0}^{u}(X-K)$ is compact.
\end{proof}

\begin{prop}$X+_{f_{K}}\pi_{0}^{u}(X-K)$ is Hausdorff.
\end{prop}

\begin{proof}Let $F$ be a compact subset of $X$. We have that $\forall U \in \pi_{0}^{u}(X-K), \ U \cap F$ is bounded, which implies that $f_{K}(F) = \emptyset$. Let $U,V \in \pi_{0}^{u}(X-K)$. Since $U$ and $V$ are open in $X$, we have that $X-U, X-V \in Closed(X)$. But $X = (X-U) \cup (X-V)$ and $f_{K}(X-U) = \pi_{0}^{u}(X-K) - \{U\}, \ f_{K}(X-V) = \pi_{0}^{u}(X-K) - \{V\}$, which implies that $U \notin f_{K}(X-U)$ and $V \notin f_{K}(X-V)$. Thus, $X+_{f_{K}}\pi_{0}^{u}(X-K)$ is Hausdorff.
\end{proof}

If $K_{1} \subseteq K_{2} \subseteq X$ are two compact subspaces, we are able to consider the map $id+\psi_{K_{1}K_{2}}: X+_{f_{K_{1}}}\pi_{0}^{u}(X-K_{2}) \rightarrow X+_{f_{K_{1}}}\pi_{0}^{u}(X-K_{1})$.

\begin{prop}$id+\psi_{K_{1}K_{2}}$ is continuous.
\end{prop}

\begin{proof}Let $F \in Closed(X)$ and $U \in \pi_{0}^{u}(X-K_{2})$. If $U \cap F$ is unbounded, then $\psi_{K_{1}K_{2}}(U) \cap F$ is unbounded (since $U \subseteq \psi_{K_{1}K_{2}}(U)$). So, if $U \in f_{K_{2}}(F)$, then $\psi_{K_{1}K_{2}}(U) \in f_{K_{1}}(F)$. Hence, $f_{K_{2}}(F) \subseteq \psi_{K_{1}K_{2}}^{-1} \circ f_{K_{1}}(F)$. In another words, we have the diagram:

$$ \xymatrix{ Closed(X) \ar[r]^<<{ \ \ \ \ f_{K_{2}}} & Closed(\pi_{0}^{u}(X-K_{2})) \\
            Closed(X) \ar[r]^<<{ \ \ \ \ f_{K_{1}}} \ar[u]^{id^{-1}} \ar@{}[ur]|{\subseteq \ \ \ \ } & Closed(\pi_{0}^{u}(X-K_{1})) \ar[u]^{\psi_{K_{1}K_{2}}^{-1}} } $$

Thus, $id+\psi_{K_{1}K_{2}}$ is continuous.
\end{proof}

Let $\K$ be the category defined by the poset of compact subspaces of $X$ with the partial order defined by inclusions. Let $\digamma: \K \rightarrow Sum(X)$ defined by $\digamma(K) = X+_{f_{K}}\pi_{0}^{u}(X-K)$ and, if $K_{1} \subseteq K_{2} \subseteq X$ are compact, $\digamma(K_{1} \subseteq K_{2}) = id+\psi_{K_{1}K_{2}}: X+_{f_{K_{1}}}\pi_{0}^{u}(X-K_{2}) \rightarrow X+_{f_{K_{1}}}\pi_{0}^{u}(X-K_{1})$.

\begin{prop}$\lim\limits_{\longleftarrow}\digamma \cong X+_{f}Ends(X)$, for some admissible map $f$.
\end{prop}

\begin{proof}It is immediate from the \textbf{Propositions \ref{limitesegundotermo}} and \textbf{\ref{densonolimite}}.
\end{proof}

\begin{lema}Let $X+_{g}Z \in Sum(X)$ be a compact space with $Z$ totally disconnected. For $K \subseteq X$ a compact, we define, for $z_{1},z_{2} \in Z, \ z_{1}\sim_{K}z_{2}$ if $z_{1}$ and $z_{2}$ are in the same connected component in $(X+_{g}Z) - K$ and extend trivially to $X+_{g}Z$. Let's define $Z_{K} = Z/ \sim_{K}$, $g_{K}$ an admissible map such that $X+_{g_{K}} Z_{K} = X+_{g}Z/\sim_{K}$ (via the identification of $X$ with its classes) and the functor $\digamma_{g}: \K \rightarrow Sum(X)$ defined by $\digamma_{g}(K) = X+_{g_{K}}Z_{K}$ and $\digamma_{g}(K_{1} \subseteq K_{2}): X+_{g_{K_{2}}}Z_{K_{2}} \rightarrow X+_{g_{K_{1}}}Z_{K_{1}}$ the quotient map. Then, $\lim\limits_{\longleftarrow} \digamma_{g} = X+_{g}Z$.
\end{lema}

\begin{obs}We have that $Z = g(Cl_{X}(X-K)) = \bigcup\limits_{U\in \pi_{0}^{u}(X-K)}g(Cl_{X}(U))$, which implies that every element of $Z$ is in the closure of some element of $\pi_{0}^{u}(X-K)$ and then in a connected component of some element of $\pi_{0}^{u}(X-K)$. Since $\pi_{0}^{u}(X-K)$ is finite and each element of $Z$ must be in a component of an element of $\pi_{0}^{u}(X-K)$ , it follows that $Z_{K}$ is finite. We have also that every connected component is closed, which implies that every class of $Z$ is closed. So $\sim_{K} = \Delta^{2}(X+_{g}Z) \cup \bigcup_{z\in Z} [z]$ is closed, which implies that $X+_{g_{K}}Z_{K}$ is Hausdorff, and then, an element of $Sum(X)$.
\end{obs}

\begin{proof}Let, for $K \subseteq X$ compact, $id+\eta_{K}: X+_{g}Z \rightarrow X+_{g_{K}}Z_{K}$ be the quotient map. We have that $X+_{g}Z$, together with the family $\{id+\eta_{K}\}_{K\in \K}$, is a cone of $\digamma$. So, it inducts a continuous map $id+\eta: X+_{g}Z \rightarrow X+_{\tilde{g}}\tilde{Z} = \lim\limits_{\longleftarrow}  \digamma$ (with $\tilde{Z} \cong \lim\limits_{\longleftarrow} Z_{K}$). Since $\K$ is codirected and $\forall K\in \K, \ X$ is dense in $X+_{g_{K}}Z_{K}$, it follows that $X$ is dense in $X+_{\tilde{g}}\tilde{Z}$ and then, the map $id+\eta$ is surjective. Let $x\neq y \in Z$. Since $Z$ is totally disconnected, there exists a clopen set $A$ of $Z$ such that $x \in A$ and $y \in Z-A$. Since $X+_{g}Z$ is normal, there exists $\tilde{A},\tilde{B}$, open sets of $X+_{g}Z$, such that $A \subseteq \tilde{A}, \ Z-A \subseteq \tilde{B}$ and $\tilde{A} \cap \tilde{B} = \emptyset$. Take $K = (X+_{g}Z) - (\tilde{A}\cup \tilde{B})$. We have that $K$ is compact and $K \subseteq X$. Since $(X+_{g}Z) - K = \tilde{A}\cup \tilde{B}$ and $\tilde{A}$ and $\tilde{B}$ are open in $X+_{g}Z$, we have that $\tilde{A}$ and $\tilde{B}$ are clopen in $(X+_{g}Z) - K$. So, $x$ and $y$ are not in the same connected component of $(X+_{g}Z) - K$, which implies that $x\nsim_{K}y$ and then $\eta_{K}(x) \neq \eta_{K}(y)$. So, $(id+\eta)(x) \neq (id+\eta)(y)$, which implies that $id+\eta$ is injective. Thus, $id+\eta$ is bijective and, since $X+_{g}Z$ is compact and $X+_{\tilde{g}}\tilde{Z}$ is Hausdorff, it is a homeomorphism.
\end{proof}

\begin{prop}\label{universalends}Let $X+_{g}Z \in Sum(X)$ be a compact space with $Z$ totally disconnected. Then, there exists only one continuous surjective map of the form $id:\phi: X+_{f}Ends(X) \rightarrow X+_{g}Z$.
\end{prop}

\begin{proof}Let $K \subseteq X$ be a compact. We define $\zeta_{K}: \pi_{0}^{u}(X-K) \rightarrow Z_{K}$ as $\zeta_{K}(U) = [z]$, where $z$ is in the same component of $U$ in $X-K$. By the definition of $\sim_{K}$, the map $\zeta_{K}$ is well defined, it is continuous because the space $\pi_{0}^{u}(X-K)$ is discrete and is surjective since every element of $Z$ is in some connected component of an element of $\pi_{0}^{u}(X-K)$.

Let $F\in Closed(X)$ and $U \in f_{K}(F)$. We have that $F\cap U$ is unbounded. Let $z \in Cl_{X+_{g}Z}(F\cap U)\cap Z$ (it exists since $F\cap U$ is unbounded). Since $U$ is connected, $z$ is in the connected component of $U$, which implies that $\zeta_{K}(U) = [z]$, and then, $U \in \zeta_{K}^{-1}([z])$. However, $[z] \in Cl_{X+_{g_{K}}Z_{K}}(F\cap U)\cap Z_{K} = g_{K}(F)$ (since $z \in Cl_{X+_{g}Z}(F\cap U)\cap Z$ and the quotient map is continuous). So, $f_{K}(F) \subseteq \zeta_{K}^{-1}(g_{K}(F))$. In another words, we have the diagram:

$$ \xymatrix{ Closed(X) \ar[r]^<<{ \ \ \ \ f_{K}} & Closed(\pi_{0}^{u}(X-K)) \\
            Closed(X) \ar[r]^<<{ \ \ \ \ \ \ g_{K}} \ar[u]^{id^{-1}} \ar@{}[ur]|{\subseteq} & Closed(Z_{K}) \ar[u]^{\zeta_{K}^{-1}} } $$

So, $id+\zeta_{K}: X+_{f_{K}}\pi_{0}^{u}(X-K) \rightarrow X+_{g_{K}}Z_{K}$ is continuous.

Let $K_{1} \subseteq K_{2} \subseteq X$ be compact subspaces. It is clear that the diagram commutes:

$$ \xymatrix{ X \!\! +_{f_{K_{2}}} \! \! \pi_{0}^{u}(X \! - \! K_{2}) \ar[r]_*+{\labelstyle \ \ id + \psi_{K_{1}K_{2}}} \ar[d]^{id+\zeta_{K_{2}}} & X \!\! +_{f_{K_{1}}} \! \! \pi_{0}^{u}(X \!- \! K_{1}) \ar[d]^{id+\zeta_{K_{1}}} \\
            X \!\! +_{g_{K_{2}}} \! \! Z_{K_{2}} \ar[r]^{\digamma_{g}(K_{1} \subseteq K_{2})} & X \!\! +_{g_{K_{1}}} \! \! Z_{K_{1}}  }$$

So, it induces a continuous map $id+\zeta: X+_{f}Ends(X) \rightarrow X+_{g}Z$. Since $X$ is dense in $X+_{f}Ends(X)$, it is the only map that extends $id_{X}$ and since $X$ is dense in $X+_{g}Z$, it follows that the map must be surjective.

\end{proof}

\begin{cor}\label{universalends}Let $X+_{g}Z \in SUM(X)$ be a compact space with $Z$ totally disconnected. Then, there exists only one continuous surjective map of the form $id:\phi: X+_{f}Ends(X) \rightarrow X+_{g}Z$.
\end{cor}

\begin{proof}Just apply the last proposition to the subspace $X \cup g(X)$.
\end{proof}

So, $X+_{f}Ends(X)$ is the Freudenthal compactification of $X$.

\begin{prop}\label{extensaoends}Let $X_{1},X_{2}$ be locally compact, connected and locally connected spaces and $j: X_{1} \rightarrow X_{2}$ be a proper continuous map. Then, there exists a unique continuous extension to the Freudenthal compactifications: $X_{1}+_{f_{1}}Ends(X_{1}) \rightarrow X_{2}+_{f_{2}} Ends(X_{2})$.
\end{prop}

\begin{proof}Since $j$ is proper, $\forall K \subseteq X_{2}$ compact, $j^{-1}(K)$ is also compact. If $U \in \pi_{0}^{u}(X_{1} - j^{-1}(K)), \ j(U)$ is connected, which implies that it is contained in a connected component of $X_{2}-K$. If $j(U)$ is bounded, then $Cl_{X_{2}}(j(U))$ is compact, which implies that $j^{-1}(Cl_{X_{2}}(j(U)))$ is compact (since $j$ is proper). But $j^{-1}(Cl_{X_{2}}(j(U))) \supseteq U$, contradicting the fact that $U \in \pi_{0}^{u}(X_{1}-j^{-1}(K))$. So, $j(U)$ is contained in an (unique) element of $\pi_{0}^{u}(X_{2}-K)$. Consider the map $j_{K}: \pi_{0}^{u}(X_{1}-j^{-1}(K)) \rightarrow \pi_{0}^{u}(X_{2} - K)$ defined by $j_{K}(U)$ equal to the connected component of $j(U)$. Since $\pi_{0}^{u}(X_{1} - j^{-1}(K))$ is discrete, it follows that $j_{K}$ is continuous.

Let $F$ be a closed subset of $X_{2}$. Then, we have that $f_{j^{-1}(K)} \circ j^{-1}(F) = \{U \in \pi_{0}^{u}(X_{1}-j^{-1}(K)): U \cap j^{-1}(F)$ is unbounded$\}$. But $U \cap j^{-1}(F)$ unbounded implies that $j(U \cap j^{-1}(F)) \subseteq j(U) \cap F \subseteq j_{K}(U) \cap F$ is unbounded. Hence, $U \in f_{j^{-1}(K)} \circ j^{-1}(F)$ implies $j_{K}(U) \in f_{K}(F)$ and then $U \in j_{K}^{-1} \circ f_{K}(F)$. So, $f_{j^{-1}(K)} \circ j^{-1}(F) \subseteq j_{K}^{-1} \circ f_{K}(F)$. In another words, we have the diagram:

$$ \xymatrix{ Closed(X_{1}) \ar[r]^<<{ \ \ \ \ f_{j^{-1}(K)}} & Closed(\pi_{0}^{u}(X_{1}-j^{-1}(K))) \\
            Closed(X_{2}) \ar[r]^<<{ \ \ \ \ \ \ f_{K}} \ar[u]^{j} \ar@{}[ur]|{\subseteq} & Closed(\pi_{0}^{u}(X_{2}-K)) \ar[u]^{j_{K}^{-1}} } $$

Then, the map $j+j_{K}: X_{1}+_{f_{j^{-1}(K)}} \pi_{0}^{u}(X_{1}-j^{-1}(K)) \rightarrow X_{2}+_{f_{K}} \pi_{0}^{u}(X_{2}-K)$ is continuous. It is clear that, $\forall K_{1} \subseteq K_{2} \subseteq X_{2}$ compact subspaces, the diagram commutes:

$$ \xymatrix{ X_{1} \!\! +_{f_{j^{-1}(K_{2})}} \!\! \pi_{0}^{u}(X_{1} \! - \! j^{-1}(K_{2})) \ar[r]_*++{\labelstyle \ id+ \psi_{j^{-1}(K_{1})j^{-1}(K_{2})} } \ar[d]^{j+j_{K_{2}}} & X_{1} \!\! +_{f_{j^{-1}(K_{1})}} \!\! \pi_{0}^{u}(X_{1} \! - \! j^{-1}(K_{1})) \ar[d]^{j+j_{K_{1}}}  \\
           X_{2} \!\! +_{f_{K_{2}}} \!\! \pi_{0}^{u}(X_{2} \! - \! K_{2}) \ar[r]^{id+ \psi_{K_{1}K_{2}}}  & X_{2} \!\! +_{f_{K_{1}}} \!\! \pi_{0}^{u}(X_{2} \! - \! K_{1}) } $$

So, it induces a continuous map $\tilde{j}: X_{1}+_{f_{1}}Ends(X_{1}) \rightarrow X_{2}+_{f_{2}}Ends(X_{2})$ that extends $j$. The uniqueness comes from the fact that the space is dense on its compactification.

\end{proof}

This construction of the Freudenthal compactification is going to be useful later. However, the existence of limits described on the last section gives us an easier way to construct the Freudenthal compactification that works for every locally compact Hausdorff space:

\begin{prop}Let $X$ be a locally compact Hausdorff space. Then, there exists a universal compactification of $X$ such that the remainder is totally disconnected.
\end{prop}

\begin{proof}Let $\{X+_{f_{i}}Y_{i}\}_{i \in \Gamma}$ be the collection of all compactifications of $X$ such that $Y_{i}$ is totally disconnected. Since all of them are quotients of the Stone-$\check{C}$ech compactification, it follows that this collection is actually a set. So, there exists a product $X+_{f}Y$ for those spaces on the category $Sum(X)$. Since this product is the closure of $X$ in the pullback on the category $Top$, we have that $X+_{f}Y$ is compact and $Y$ is a subspace of $\prod Y_{i}$, which implies that $Y$ is totally disconnected. The projection maps are the unique maps to the spaces $X+_{f_{i}}Y_{i}$ that are the identity on $X$, since $X$ is dense. Thus, $X+_{f}Y$ is the universal compactification of $X$ such that the remainder is totally disconnected.
\end{proof}

\section{Attractor-Sum functors}

\begin{defi}Let $G$ be a group, $X$ and $Y$ Hausdorff topological spaces with $X$ locally compact and $Y$ compact, $L: G \curvearrowright G$ the left multiplication action, $\varphi: G \curvearrowright X$ a  properly discontinuous cocompact action, $\psi: G \curvearrowright Y$ an action by homeomorphisms and $K \subseteq X$ a compact such that $\varphi(G,K) = X$. We define $\Pi_{K}: Closed(X) \rightarrow  Closed(G)$ as $\Pi_{K}(S) = \{g\in G: \varphi(g,K)\cap S \neq \emptyset\}$ and $\Lambda_{K}: Closed(G) \rightarrow Closed(X)$ as $\Lambda_{K}(F) = \varphi(F,K)$.
\end{defi}

\begin{prop}$\forall F \in Closed(G), \ \varphi(F,K)\in Closed(X)$.
\end{prop}

\begin{proof}Let $x \in X$ and $U\ni x$ be an open set such that $Cl_{X}(U)$ is compact. Since $\varphi$ is properly discontinuous, the set $\{g\in F: \varphi(g,K)\cap Cl_{X}(U)\}$ is finite. So, the family of closed sets $\{\varphi(g,K)\}_{g \in F}$ is locally finite, which implies that $\varphi(F,K) = \bigcup_{g \in F}\varphi(g,K)$ is closed.
\end{proof}

\begin{prop}$\Pi_{K}$ and $\Lambda_{K}$ are admissible.
\end{prop}

\begin{proof}We have that $\Pi_{K}(\emptyset) = \{g \in G: \varphi(g,K)\cap \emptyset \neq \emptyset\} = \emptyset$ and, for $S_{1},S_{2}\in Closed(X)$, $\Pi_{K}(S_{1}\cup S_{2}) = \{g \in G: \varphi(g,K)\cap (S_{1}\cup S_{2}) \neq \emptyset\} =$ $\{g \in G: (\varphi(g,K)\cap S_{1})\cup (\varphi(g,K)\cap S_{2}) \neq \emptyset\} =$ $\\ \{g \in G: \varphi(g,K)\cap S_{1} \neq \emptyset\} \cup \{g \in G: \varphi(g,K)\cap S_{2} \neq \emptyset\} = \Pi_{K}(S_{1}) \cup \Pi_{K}(S_{2})$. Thus, $\Pi_{K}$ is admissible.

We have that $\Lambda_{K}(\emptyset) = \emptyset$ and, for $F_{1},F_{2}\in Closed(G)$, $\Lambda_{K}(F_{1}\cup F_{2}) = \varphi(F_{1}\cup F_{2},K) = \varphi(F_{1},K) \cup \varphi(F_{2},K) = \Lambda_{K}(F_{1})\cup \Lambda_{K}(F_{2})$. Thus, $\Lambda_{K}$ is admissible.
\end{proof}

We consider, for the admissible maps $\partial: \ Closed(G) \ \rightarrow \ Closed(Y)$ and $f: Closed(X) \rightarrow Closed(Y)$, the compositions $\partial_{id \Pi_{K}} = \partial \circ \Pi_{K}$ and $f_{id \Lambda_{K}} = f \circ \Lambda_{K}$. To simplify the notation, we denote them by $\partial_{\Pi_{K}}$ and $f_{\Lambda_{K}}$, respectively.

\begin{prop}$\Pi_{K}\circ \Lambda_{K} \supseteq id_{Closed(G)}$ and $\Lambda_{K} \circ \Pi_{K} \supseteq id_{Closed(X)}$.
\end{prop}

\begin{proof}We have that $\Pi_{K}\circ \Lambda_{K}(F) = \{g\in G: \varphi(g,K)\cap \varphi(F,K)\neq \emptyset\} \supseteq F$. We have also that $\Lambda_{K} \circ \Pi_{K}(S) = \Lambda_{K}(\{g \in G: \varphi(g,K)\cap S \neq \emptyset\}) = \bigcup \{\varphi(g,K): \varphi(g,K)\cap S \neq \emptyset\} \supseteq S$, since $\varphi(G,K) = X$.
\end{proof}

\begin{cor}Let $G+_{\partial}Y$ be a topological space. Then, the application $id: G+_{\partial}Y \rightarrow G+_{(\partial_{\Pi_{K}})_{\Lambda_{K}}}Y$ is continuous.
\eod\end{cor}

\begin{cor}Let $X+_{f}Y$ be a topological space. Then, the application $id: X+_{f}Y \rightarrow X+_{(f_{\Lambda_{K}})_{\Pi_{K}}}Y$ is continuous.
\eod\end{cor}

\begin{prop}\label{functorobjeto1}If $L+\psi: G \curvearrowright G+_{\partial}Y$ is an action by homeomorphisms, then $\varphi+\psi: G \curvearrowright X+_{\partial_{\Pi_{K}}}Y$ is an action by homeomorphisms.
\end{prop}

\begin{proof}Let $S$ be a closed subset of $X$ and $g \in G$. Then, $\Pi_{K} \circ \varphi(g,\_)^{-1}(S) = \Pi_{K} (\varphi(g^{-1},S)) = \{h \in G: \varphi(h,K)\cap \varphi(g^{-1},S) \neq \emptyset\} =$ $\\ \{g^{-1}h' \in G: \varphi(g^{-1}h',K)\cap \varphi(g^{-1},S) \neq \emptyset\} = \\ \{g^{-1}h' \in G: \varphi(g^{-1},\varphi(h',K))\cap \varphi(g^{-1},S) \neq \emptyset\} = \\ \{g^{-1}h' \in G: \varphi(h',K)\cap S \neq \emptyset\} = L(g,\_)^{-1}(\{h' \in G: \varphi(h',K)\cap S \neq \emptyset\}) = L(g,\_)^{-1}\circ \Pi_{K}(S)$. So, the diagrams commute ($\forall g \in G$):

$$ \xymatrix{   Closed(G) \ar[r]^{\ \ L(g,\_)^{-1}} & Closed(G) & & Closed(Y) \ar[r]^{\ \ \psi(g,\_)^{-1}} \ar[d]^{id} & Closed(Y) \ar[d]_{id} \\
                Closed(X) \ar[r]^{\ \ \varphi(g,\_)^{-1}} \ar[u]^{\Pi_{K}} & Closed(X) \ar[u]_{\Pi_{K}} & & Closed(Y) \ar[r]^{\ \ \psi(g,\_)^{-1}} & Closed(Y) } $$

By the \textbf{Corollary \ref{action}}, $\varphi+\psi: G \curvearrowright X+_{\partial_{\Pi_{K}}}Y$ is an action by homeomorphisms.

\end{proof}

\begin{prop}\label{functorobjeto2}If $\varphi+\psi: G \curvearrowright X+_{f}Y$ is an action by homeomorphisms, then $L+\psi: G \curvearrowright G+_{f_{\Lambda_{K}}}Y$ is an action by homeomorphisms.
\end{prop}

\begin{proof}Let $F \in Closed(G)$ and $g \in G$. Then, $\Lambda_{K} \circ L(g,\_)^{-1}(F) = \Lambda_{K}(g^{-1}F) = \varphi(g^{-1}F,K) = \varphi(g^{-1},\varphi(F,K)) = \varphi(g,\_)^{-1} \circ \Lambda_{K}(F)$. So, the diagrams commute ($\forall g \in G$):

$$ \xymatrix{   Closed(X) \ar[r]^{\ \ \varphi(g,\_)^{-1}} & Closed(X) & & Closed(Y) \ar[r]^{\ \ \psi(g,\_)^{-1}} \ar[d]^{id} & Closed(Y) \ar[d]_{id} \\
                Closed(G) \ar[r]^{\ \ L(g,\_)^{-1}} \ar[u]^{\Lambda_{K}} & Closed(G) \ar[u]_{\Lambda_{K}} & & Closed(Y) \ar[r]^{\ \ \psi(g,\_)^{-1}} & Closed(Y) } $$

By the \textbf{Corollary \ref{action}}, $L+\psi: G \curvearrowright G+_{\partial_{\Lambda_{K}}}Y$ is an action by homeomorphisms.

\end{proof}

\begin{prop}If $G+_{\partial}Y$ is compact, then $X+_{\partial_{\Pi_{K}}}Y$ is compact.
\end{prop}

\begin{proof}Let $S\in Closed(X)$ be non compact. Let's suppose that $\Pi_{K}(S)$ is compact. Then, $S \subseteq \Lambda_{K} \circ \Pi_{K}(S) = \varphi(\Pi_{K}(S),K)$, which is compact, absurd. So, $\Pi_{K}(S)$ is not compact. Since $G+_{\partial}Y$ is compact, it follows that $\partial(\Pi_{K}(S)) \neq \emptyset$. Thus, $\partial_{\Pi_{K}}(S) = \partial(\Pi_{K}(S)) \neq \emptyset$, which implies that $X+_{\partial_{\Pi_{K}}}Y$ is compact.
\end{proof}

\begin{prop}If $X+_{f}Y$ is compact, then $G+_{f_{\Lambda_{K}}}Y$ is compact.
\end{prop}

\begin{proof}Let $F\in Closed(G)$ be non compact. Let's suppose that $\Lambda_{K}(F)$ is compact. Then, $F \subseteq \Pi_{K} \circ \Lambda_{K}(F)$, which is compact, absurd. So, $\Lambda_{K}(F)$ is not compact. Since $X+_{f}Y$ is compact, it follows that $f(\Lambda_{K}(F)) \neq \emptyset$. Thus, $f_{\Lambda_{K}}(F) = f(\Lambda_{K}(F)) \neq \emptyset$, which implies that $G+_{f_{\Lambda_{K}}}Y$ is compact.
\end{proof}

\begin{prop}If $G$ is dense in $G+_{\partial}Y$, then $X$ is dense in $X+_{\partial_{\Pi_{K}}}Y$.
\end{prop}

\begin{proof}We have that $\partial_{\Pi_{K}}(X) = \partial(\{g \in G: \varphi(g,K)\cap X \neq \emptyset\}) = \partial(G) = Y$. Thus, $X$ is dense in $X+_{\partial_{\Pi_{K}}}Y$.
\end{proof}

\begin{prop}If $X$ is dense in $X+_{f}Y$, then $G$ is dense in $G+_{f_{\Lambda_{K}}}Y$.
\end{prop}

\begin{proof}We have that $f_{\Lambda_{K}}(G) = f(\varphi(G,K)) = f(X) = Y$. Thus, $G$ is dense in $G+_{f_{\Lambda_{K}}}Y$.
\end{proof}

\begin{prop}\label{functor}Let $G_{1}$ and $G_{2}$ be groups, $X_{1}$ and $X_{2}$ Hausdorff locally compact spaces, $Y_{1}$ and $Y_{2}$ Hausdorff compact spaces, $\varphi_{i}: G_{i} \curvearrowright X_{i}$ properly discontinuous cocompact actions, $\psi_{i}: G_{i} \curvearrowright Y_{i}$ actions by homeomorphisms, $\alpha: G_{1} \rightarrow G_{2}$ a homomorphism, $\mu: X_{1} \rightarrow X_{2}$, $\nu: Y_{1} \rightarrow Y_{2}$ continuous maps with $\mu$ $\alpha$-equivariant, $K_{i} \subseteq X_{i}$ compact subspaces such that $\varphi_{i}(G,K_{i}) = X_{i}$ and $\mu(K_{1}) \subseteq K_{2}$ and $\Pi_{K_{i}}: Closed(X_{i}) \rightarrow Closed(G_{i})$. If the application $\alpha+\nu: G_{1}+_{\partial_{1}}Y_{1} \rightarrow G_{2}+_{\partial_{2}}Y_{2}$ is continuous, then the application $\\ \mu+\nu: X_{1}+_{\partial_{1\Pi_{K_{1}}}} Y_{1} \rightarrow X_{2}+_{\partial_{2\Pi_{K_{2}}}} Y_{2}$ is continuous.
\end{prop}

\begin{proof}Let $S \in Closed(X_{2})$ and $g \in \Pi_{K_{1}}(\mu^{-1}(S))$. So, $\varphi_{1}(g,K_{1})\cap \mu^{-1}(S) \neq \emptyset$, which implies that $\mu(\varphi_{1}(g,K_{1})\cap \mu^{-1}(S)) \neq \emptyset$. But $\mu(\varphi_{1}(g,K_{1})\cap \mu^{-1}(S)) \subseteq \mu(\varphi_{1}(g,K_{1}))\cap \mu(\mu^{-1}(S)) = \varphi_{2}(\alpha(g),\mu(K_{1}))\cap S \subseteq \varphi_{2}(\alpha(g),K_{2})\cap S$, which implies that $\varphi_{2}(\alpha(g),K_{2})\cap S \neq \emptyset$ and then $\alpha(g) \in \Pi_{K_{2}}(S)$. So, $\alpha(\Pi_{K_{1}}(\mu^{-1}(S))) \subseteq \Pi_{K_{2}}(S)$, which implies that $\Pi_{K_{1}}(\mu^{-1}(S)) \subseteq \alpha^{-1}(\Pi_{K_{2}}(S))$.

So, we have the diagrams:

$$ \xymatrix{ Closed(G_{2}) \ar[r]^{\alpha^{-1}} & Closed(G_{1}) & & Closed(Y_{2}) \ar[r]^{\nu^{-1}} & Closed(Y_{1}) \\
            Closed(X_{2}) \ar[r]^{\mu^{-1}} \ar[u]^{\Pi_{K_{2}}} \ar@{}[ur]|{\supseteq} & Closed(X_{1}) \ar[u]_{\Pi_{K_{1}}} & & Closed(Y_{2}) \ar[r]^{\nu^{-1}} \ar[u]^{id} \ar@{}[ur]|{\circlearrowleft} & Closed(Y_{1}) \ar[u]^{id} } $$

By the Cube Lemma, it follows that $\mu+\nu: X_{1}+_{\partial_{1\Pi_{K_{1}}}}Y_{1} \rightarrow X_{2}+_{\partial_{2\Pi_{K_{2}}}}Y_{2}$ is continuous.

\end{proof}

\begin{prop}\label{functor2}Let $G_{1}$ and $G_{2}$ be groups, $X_{1}$ and $X_{2}$ Hausdorff locally compact spaces, $Y_{1}$ and $Y_{2}$ Hausdorff compact spaces, $\varphi_{i}: G_{i} \curvearrowright X_{i}$ properly discontinuous cocompact actions, $\psi_{i}: G_{i} \curvearrowright Y_{i}$ actions by homeomorphisms, $\alpha: G_{1} \rightarrow G_{2}$ a homomorphism, $\mu: X_{1} \rightarrow X_{2}$, $\nu: Y_{1} \rightarrow Y_{2}$ continuous maps with $\mu$ $\alpha$-equivariant, $K_{i} \subseteq X_{i}$ compacts such that $\varphi_{i}(G,K_{i}) = X_{i}$ and $\mu(K_{1}) \subseteq K_{2}$ and $\Lambda_{K_{i}}: Closed(G_{i}) \rightarrow Closed(X_{i})$ given by $\Lambda_{K_{i}}(F) = \varphi_{i}(G,K_{i})$. If $\mu+\nu: X_{1}+_{f_{1}}Y_{1} \rightarrow X_{2}+_{f_{2}}Y_{2}$ is continuous, then the map $\alpha+\nu: G_{1}+_{f_{1\Lambda_{K_{1}}}} Y_{1} \rightarrow G_{2}+_{f_{2\Lambda_{K_{2}}}} Y_{2}$ is continuous.
\end{prop}

\begin{proof}Let $F \in Closed(G_{2})$. Then, $\Lambda_{K_{1}}(\alpha^{-1}(F)) = \varphi_{1}(\alpha^{-1}(F),K_{1}) \subseteq \mu^{-1}(\mu(\varphi_{1}(\alpha^{-1}(F),K_{1}))) = \mu^{-1}(\varphi_{2}(\alpha(\alpha^{-1}(F)),\mu(K_{1}))) \subseteq \mu^{-1}(\varphi_{2}(F,K_{2})) = \mu^{-1}(\Lambda_{K_{2}}(F))$. So, we have the diagrams:

$$ \xymatrix{ Closed(X_{2}) \ar[r]^{\mu^{-1}} & Closed(X_{1}) & & Closed(Y_{2}) \ar[r]^{\nu^{-1}} & Closed(Y_{1}) \\
            Closed(G_{2}) \ar[r]^{\alpha^{-1}} \ar[u]^{\Lambda_{K_{2}}} \ar@{}[ur]|{\supseteq} & Closed(G_{1}) \ar[u]_{\Lambda_{K_{1}}} & & Closed(Y_{2}) \ar[r]^{\nu^{-1}} \ar[u]^{id} \ar@{}[ur]|{\circlearrowleft} & Closed(Y_{1}) \ar[u]^{id} } $$

By the Cube Lemma, it follows that $\alpha+\nu: G_{1}+_{f_{1\Lambda_{K_{1}}}} Y_{1} \rightarrow G_{2}+_{f_{2\Lambda_{K_{2}}}} Y_{2}$ is continuous.

\end{proof}

Let's consider what happens by the categorical point of view:

\begin{defi}Let $G$ be a group and $L: G \curvearrowright G$ the left multiplication action. We denote by $Comp(G)$ the category whose  objects are compact spaces of the form $G+_{\partial}Y$, for some Hausdorff compact $Y$, with an action by homeomorphisms of the form $L+\psi: G \curvearrowright G+_{\partial}Y$, and the morphisms are continuous maps of the form $id+\phi: G+_{\partial_{1}}Y_{1}\rightarrow G+_{\partial_{2}}Y_{2}$ , such that $\phi$ is equivariant with respect to the actions of $G$ in $Y_{1}$ and $Y_{2}$.
\end{defi}

\begin{defi}Let $G$ be a group, $X$ a locally compact and Hausdorff space and $\varphi: G \curvearrowright X$ a properly discontinuous  cocompact action. We denote by $Comp(\varphi)$ the category whose  objects are compact spaces of the form $X+_{f}Y$, for some Hausdorff compact $Y$, with an action by homeomorphisms of the form $\varphi+\psi: G \curvearrowright X+_{f}Y$, and the morphisms are continuous maps of the form $id+\phi: X+_{f_{1}}Y_{1}\rightarrow X+_{f_{2}}Y_{2}$ , such that $\phi$ is equivariant with respect to the actions of $G$ in $Y_{1}$ and $Y_{2}$.
\end{defi}

\begin{prop}Let $K \subseteq X$ be a compact such that $\varphi(G,K) = X$. The map $\bar{\Pi}_{K}: Comp(G) \rightarrow Comp(\varphi)$, that takes $G+_{\partial}Y$ to $X+_{\partial_{\Pi_{K}}}Y$, the action $L+\psi:G \curvearrowright G+_{\partial}Y$ to $\varphi+\psi: G \curvearrowright X+_{\partial_{\Pi_{K}}}Y$ and $id+\phi: G+_{\partial_{1}}Y_{1}\rightarrow G+_{\partial_{2}}Y_{2}$ to $id+\phi: X+_{\partial_{1\Pi_{K}}}Y_{1}\rightarrow X+_{\partial_{2\Pi_{K}}}Y_{2}$, is a functor.
\end{prop}

\begin{proof}By the \textbf{Proposition \ref{functorobjeto1}}, if $L+\psi:G \curvearrowright G+_{\partial}Y$ is an action by  homeomorphisms, then $\varphi+\psi: G \curvearrowright X+_{\partial_{\Pi_{K}}}Y$ is also an action by homeomorphisms. So, $\bar{\Pi}_{K}$ maps objects to objects. By the \textbf{Proposition \ref{functor}}, if $id+\phi: G+_{\partial_{1}}Y_{1}\rightarrow G+_{\partial_{2}}Y_{2}$ is continuous, then the application $id+\phi: X+_{\partial_{1\Pi_{K}}} Y_{1}\rightarrow X+_{\partial_{2\Pi_{K}}}Y_{2}$ is also continuous. So, $\bar{\Pi}_{K}$ maps morphisms to morphisms. It is  trivial that it maps identity to identity and preserves compositions. Thus, it is a functor.
\end{proof}

\begin{prop}Let $K \subseteq X$ be a compact such that $\varphi(G,K) = X$. The map $\bar{\Lambda}_{K}: Comp(\varphi) \rightarrow Comp(G)$, that takes $X+_{f}Y$ to $G+_{f_{\Lambda_{K}}}Y$, the action $\varphi+\psi:G \curvearrowright X+_{f}Y$ to $L+\psi: G \curvearrowright G+_{f_{\Lambda_{K}}}Y$ and $id+\phi: X+_{f_{1}}Y_{1}\rightarrow X+_{f_{2}}Y_{2}$ to $id+\phi: G+_{f_{1\Lambda_{K}}}Y_{1}\rightarrow G+_{f_{2\Lambda_{K}}}Y_{2}$ is a functor.
\end{prop}

\begin{proof}By the \textbf{Proposition \ref{functorobjeto2}}, if $\varphi+\psi:G \curvearrowright X+_{f}Y$ is an action by homeomorphisms, then $L+\psi: G \curvearrowright G+_{f_{\Lambda_{K}}}Y$ is also an action by homeomorphisms. So, $\bar{\Lambda}_{K}$ maps objects to objects. By the \textbf{Proposition \ref{functor2}}, if $id+\phi: X+_{f_{1}}Y_{1}\rightarrow X+_{f_{2}}Y_{2}$ is continuous, then the application  $id+\phi: G+_{f_{1\Lambda_{K}}}Y_{1}\rightarrow G+_{f_{2\Lambda_{K}}}Y_{2}$ is also continuous. So, $\bar{\Lambda}_{K}$ maps morphisms to morphisms. It is trivial that it maps identity to identity and preserves compositions. Thus, it is a functor.
\end{proof}

\begin{prop}The family of maps $id: G+_{\partial}Y  \rightarrow G+_{(\partial_{\Pi_{K}})_{\Lambda_{K}}}Y$ forms a natural transformation $id_{K}: id_{Comp(G)}\Rightarrow \bar{\Lambda}_{K} \circ \bar{\Pi}_{K}$.
\end{prop}

\begin{proof}Let $id+\phi: G+_{\partial_{1}}Y_{1}\rightarrow G+_{\partial_{2}}Y_{2}$ be a morphism. We have that $\bar{\Lambda}_{K} \circ \bar{\Pi}_{K}(id+\phi) = id+\phi: G+_{(\partial_{1\Pi_{K}})_{\Lambda_{K}}}Y_{1}\rightarrow G+_{(\partial_{2\Pi_{K}})_{\Lambda_{K}}}Y_{2}$ and the diagram commutes:

$$ \xymatrix{  G+_{\partial_{1}}Y_{1} \ar[r]^<<{ \ \ \ \ id} \ar[d]^{id+\phi} & G+_{(\partial_{1\Pi_{K}})_{\Lambda_{K}}} \!\! Y_{1} \ar[d]^{id+\phi} \\
            G+_{\partial_{2}}Y_{2}  \ar[r]^<<{ \ \ \ \ id} & G+_{(\partial_{2\Pi_{K}})_{\Lambda_{K}}} \!\! Y_{2} }$$

Thus, $id_{K}$ is a natural transformation.

\end{proof}

\begin{prop}The family of maps $id: X+_{f}Y  \rightarrow X+_{(f_{\Lambda_{K}})_{\Pi_{K}}}Y$ forms a natural transformation $id'_{K}: id_{Comp(\varphi)}\Rightarrow \bar{\Pi}_{K} \circ \bar{\Lambda}_{K}$.
\end{prop}

\begin{proof}Let $id+\phi: X+_{f_{1}}Y_{1}\rightarrow X+_{f_{2}}Y_{2}$ be a morphism. We have that $\bar{\Pi}_{K} \circ \bar{\Lambda}_{K}(id+\phi) = id+\phi: X+_{(f_{1\Lambda_{K}})_{\Pi_{K}}}Y_{1}\rightarrow X+_{(f_{2\Lambda_{K}})_{\Pi_{K}}}Y_{2}$ and the diagram commutes:

$$ \xymatrix{  X+_{f_{1}}Y_{1} \ar[r]^<<{ \ \ \ \ id} \ar[d]^{id+\phi} & X+_{(f_{1\Lambda_{K}})_{\Pi_{K}}} \!\! Y_{1} \ar[d]^{id+\phi} \\
            X+_{f_{2}}Y_{2}  \ar[r]^<<{ \ \ \ \ id} & X+_{(f_{2\Lambda_{K}})_{\Pi_{K}}} \!\! Y_{2} }$$

Thus, $id'_{K}$ is a natural transformation.

\end{proof}

However, this is not an equivalence of categories since the first natural transformation might not be an isomorphism (and even do not be replaced by someone), as the example below shows us:

\begin{ex}\label{exemplochave}Let $G = \Z, \ Y = \{x_{0},x_{1},x_{2}\}$ with the  discrete topology, $\psi: \Z \curvearrowright Y$ given by $\psi(z,x_{i}) = x_{i+z mod 3}, \ X = \R$ and $\varphi: \Z \curvearrowright \R$ given by $\varphi(z,r) = z+r$. Let's take the one point compactification $A_{i} = (i+3\Z)\cup \{x_{i}\} \\$ and $\Z+_{\partial}Y = A_{1} \dot{\cup} A_{2} \dot{\cup} A_{3}$. Clearly $L+\psi: \Z \curvearrowright \Z+_{\partial}Y$ is an action by homeomorphisms. So, we have an Hausdorff space that is an object of $Comp(G)$. Take $K = [0,3]$. So, we have that $\partial(3 \Z) = \{x_{0}\}$, but $(\partial_{\Pi_{[0,3]}})_{\Lambda_{[0,3]}}(3\Z) = \partial_{\Pi_{[0,3]}}(3\Z+[0,3]) = \partial_{\Pi_{[0,3]}}(\R) = \partial(\{z \in \Z: (z+[0,3])\cap\R \neq \emptyset\}) = \partial(\Z) = Y$. Thus, $\partial \neq (\partial_{\Pi_{[0,3]}})_{\Lambda_{[0,3]}}$, which implies that $id: G+_{\partial}Y \rightarrow G+_{(\partial_{\Pi_{[0,3]}})_{\Lambda_{[0,3]}}}Y$ is not a homeomorphism, and then $id_{[0,3]}$ is not a natural isomorphism. Since $id: G+_{\partial}Y \rightarrow G+_{(\partial_{\Pi_{[0,3]}})_{\Lambda_{[0,3]}}}Y$ is continuous, it is not a homeomorphism and $G+_{\partial}Y$ is compact, it follows that $G+_{(\partial_{\Pi_{[0,3]}})_{\Lambda_{[0,3]}}}Y$ is not Hausdorff and then it is not homeomorphic to $G+_{\partial}Y$. Thus, there is no natural isomorphism between $id_{Comp(G)}$ and $\bar{\Lambda}_{[0,3]}\circ \bar{\Pi}_{[0,3]}$.

On the same example, take a compact $K\subseteq \R$ such that $\Z+K = \R$. We have that $\partial_{\Pi_{K}}(\R) = \partial(\{z \in \Z: (z+K)\cap\R\neq \emptyset\}) = \partial(\Z) = Y$, which implies that $\R$ is dense in $\R+_{\partial_{\Pi_{K}}}Y$. If $\R+_{\partial_{\Pi_{K}}}Y$ were Hausdorff, it would be a compactification of $\R$ by three points, which does not exist since $\#End(\R) = 2$. Thus, $\R+_{\partial_{\Pi_{K}}}Y$ is not Hausdorff.
\end{ex}

One way to obtain natural isomorphisms (and then an equivalence of categories) would be restrict the categories to Hausdorff spaces and shows that the functors maps Hausdorff spaces to Hausdorff spaces (so, the maps of the natural transformations would be automatically homeomorphisms). However, this fact does not happen, as the example above shows us.

In order to work around the problem, we are looking for a good property on the objects, that we are able to restrict the categories and permits us to have the isomorphism. This is the quasi-perspectivity property.

\subsection{Quasi-perspectivity}

\begin{defi}Let $G$ be a group, $Y$ a Hausdorff compact space, $L: G \curvearrowright G$ the left multiplication action, $R: G \curvearrowright G$ the right multiplication action and $\psi: G \curvearrowright Y$ an action by homeomorphisms. We say that a compact space $G+_{\partial}Y$ is a quasi-perspectivity if $L+\psi: G \curvearrowright G+_{\partial}Y$ and $R+id: G \curvearrowright G+_{\partial}Y$ are continuous. We denote by $qPers(G)$ the full subcategory of $Comp(G)$ whose objects are quasi-perspectivities.
\end{defi}

Let $\varphi: G \curvearrowright X$ be a properly discontinuous  cocompact action on a locally compact Hausdorff space and $K \subseteq X$ a fundamental domain.

\begin{prop}Let $G+_{\partial}Y \in Comp(G)$ and $F \in Closed(G)$. Then, $(\partial_{\Pi_{K}})_{\Lambda_{K}}(F) = \bigcup_{z \in Z_{K}} \partial(Fz)$, where $Z_{K} = \{z \in G: \varphi(z,K)\cap K \neq \emptyset\}$.
\end{prop}

\begin{proof}Since $\varphi$ is properly discontinuous, it follows that $Z_{K}$ is finite. Let $F \in Closed(G), \ x \in F$ and $z \in Z_{K}$. So, $\varphi(z,K) \cap K \neq \emptyset$, which implies that $\varphi(xz,K)\cap \varphi(x,K) \neq \emptyset$ and then $\varphi(xz,K)\cap \varphi(F,K) \neq \emptyset$. For the other side, let $g\in G$ such that $\varphi(g,K)\cap \varphi(F,K) \neq \emptyset$. So, $\exists x \in F:$ $\varphi(g,K)\cap \varphi(x,K) \neq \emptyset$, which implies that $\varphi(x^{-1}g,K)\cap K \neq \emptyset$ and then $x^{-1}g \in Z_{K}$, which implies that $g = xx^{-1}g \in FZ_{K}$. It follows that $FZ_{K} = \{g \in G: \varphi(g,K)\cap \varphi(F,K) \neq \emptyset\}$.

Thus, $(\partial_{\Pi_{K}})_{\Lambda_{K}}(F) = \partial_{\Pi_{K}}(\varphi(F,K)) = \\ \partial(\{g \in G: \varphi(g,K)\cap \varphi(F,K) \neq \emptyset\}) =$ $\partial(FZ_{K}) = \bigcup_{z \in Z_{K}}\partial(Fz)$, because $Z_{K}$ is finite.

\end{proof}

\begin{lema}If $K \subseteq K' \subseteq X$ are fundamental domains, then $\partial_{\Pi_{K}} \subseteq \partial_{\Pi_{K'}}$.
\end{lema}

\begin{proof}Let $S \in Closed(X)$. So, $\partial_{\Pi_{K'}}(S) = \partial(\{g \in G: \varphi(g,K')\cap S \neq \emptyset\}) \supseteq \partial(\{g \in G: \varphi(g,K)\cap S \neq \emptyset\}) = \partial_{\Pi_{K}}(S)$. Thus, $\partial_{\Pi_{K}} \subseteq \partial_{\Pi_{K'}}$.
\end{proof}

\begin{lema}If $K,K' \subseteq X$ are fundamental domains, then $\forall S \in Closed(X), \\ \partial_{\Pi_{K\cup K'}}(S) = \partial_{\Pi_{K}}(S) \cup \partial_{\Pi_{K'}}(S)$.
\end{lema}

\begin{proof}Let $S \in Closed(X)$. Then:

$\partial_{\Pi_{K\cup K'}}(S) = \partial(\{g \in G: \varphi(g,K \cup K') \cap S \neq \emptyset\}) = \\ \partial(\{g \in G: (\varphi(g,K) \cap S) \cup (\varphi(g,K') \cap S) \neq \emptyset\}) = \\ \partial(\{g \in G: \varphi(g,K) \cap S \neq \emptyset\} \cup \{g \in G: \varphi(g, K') \cap S \neq \emptyset\}) = \\ \partial(\{g \in G: \varphi(g,K) \cap S \neq \emptyset\}) \cup \partial(\{g \in G: \varphi(g,K') \cap S \neq \emptyset\}) = \\ \partial_{\Pi_{K}}(S) \cup \partial_{\Pi_{K'}}(S)$.
\end{proof}

\begin{lema}If $K \subseteq X$ is a fundamental domain and $h \in G$, then $\forall S \in Closed(X), \ \partial_{\Pi_{\varphi(h,K)}}(S) =  \partial(\Pi_{K}(S)h^{-1})$.
\end{lema}

\begin{proof}Let $S \in Closed(X)$. Then:

$\partial_{\Pi_{\varphi(h,K)}}(S) = \partial(\{g \in G: \varphi(g,\varphi(h,K)) \cap S \neq \emptyset\}) = \\ \partial(\{gh^{-1} \in G: \varphi(gh^{-1},\varphi(h,K)) \cap S \neq \emptyset\}) = \\ \partial(\{gh^{-1} \in G: \varphi(g,K) \cap S \neq \emptyset\}) = \\ \partial(\{g \in G: \varphi(g,K) \cap S \neq \emptyset\}h^{-1}) = \partial(\Pi^{}_{K}(S)h^{-1})$.
\end{proof}

Let's suppose that $G+_{\partial}Y$ is quasi-perspective.

$\partial_{\Pi_{K}}$ do not depend of the choice of the compact $K$:

\begin{prop}\label{dependencia1}Let $K,K' \subseteq X$ be fundamental domains. Then, $\partial_{\Pi_{K}} = \partial_{\Pi_{K'}}$.
\end{prop}

\begin{proof}Let $K'' = K \cup K'$. We have that $K''$ is a fundamental domain and $\partial_{\Pi_{K}} \subseteq \partial_{\Pi_{K''}}$. Let $K''' = K \cup \varphi(\Pi_{K}(K'),K)$. Since $\varphi$ is properly discontinuous, we have that $\Pi_{K}(K')$ is finite. So, $K'''$ is compact and then, a fundamental domain. Since $K$ is a fundamental domain, $\forall k' \in K'$, $\exists g \in G$ and $\exists k \in K$ such that $\varphi(g,k) = k'$. So, $K' \subseteq \varphi(\Pi_{K}(K'),K)$, which implies that $K'' \subseteq K'''$ and then, $\partial_{\Pi_{K''}} \subseteq \partial_{\Pi_{K'''}}$. Let $S \in Closed(X)$. We have that $\partial_{\Pi_{K'''}}(S) = \partial_{\Pi_{K}}(S) \cup \bigcup_{g \in \Pi_{K}(K')}\partial_{\Pi_{\varphi(g,K)}}(S) = \partial_{\Pi_{K}}(S) \cup \bigcup_{g \in \Pi_{K}(K')}\partial(\Pi_{K}(S)g^{-1}) =$ $\\ \partial_{\Pi_{K}}(S) \cup \bigcup_{g \in \Pi_{K}(K')}\partial_{\Pi_{K}}(S) = \partial_{\Pi_{K}}(S)$. So, $\partial_{\Pi_{K'''}} = \partial_{\Pi_{K}}$, which implies that $\partial_{\Pi_{K}} = \partial_{\Pi_{K''}}$. Analogously, we have that $\partial_{\Pi_{K'}} = \partial_{\Pi_{K''}}$. Thus, $\partial_{\Pi_{K}} = \partial_{\Pi_{K'}}$.
\end{proof}

\begin{prop}\label{identidade1} $\partial = (\partial_{\Pi_{K}})_{\Lambda_{K}}$.
\end{prop}

\begin{proof}Since $R+id$ is continuous, we have that $\forall g \in G, \ \partial(Fg) = \partial(F)$. Thus, $(\partial_{\Pi_{K}})_{\Lambda_{K}}(F) = \bigcup_{z \in Z_{K}}\partial(Fz) = \bigcup_{z \in Z_{K}}\partial(F) = \partial(F)$.
\end{proof}

We are not able to generalize the definition of quasi-perspectivity for the compact spaces of the form $X+_{f}Y$, since there is no clear right action established on $X$. So, let's work a bit formally:

\begin{defi}We denote by $qPers(\varphi)$ the full subcategory of $Comp(\varphi)$ where the objects are images of the functor $\bar{\Pi}_{K}|_{qPers(G)}$ (we saw that this do not depend of the choice of the compact $K$).
\end{defi}

\begin{prop}\label{identidade2}Let $X+_{f}Y \in qPers(\varphi)$. Then, $\forall K$ fundamental domain, $f = (f_{\Lambda_{K}})_{\Pi_{K}}$.
\end{prop}

\begin{proof}Since $X+_{f}Y \in qPers(\varphi)$, we have that $f = \partial_{\Pi_{K}}$ for some $\partial$ such that $G+_{\partial}Y \in qPers(G)$ (and any choice of the fundamental domain $K$). So, $f_{\Lambda_{K}} = (\partial_{\Pi_{K}})_{\Lambda_{K}} = \partial$, which implies that $(f_{\Lambda_{K}})_{\Pi_{K}} = \partial_{\Pi_{K}} = f$.
\end{proof}

\begin{prop}\label{dependencia2}Let $X+_{f}Y \in qPers(\varphi)$ and $K,K'\subseteq X$ fundamental domains. Then, $f_{\Lambda_{K}} = f_{\Lambda_{K'}}$.
\end{prop}

\begin{proof}Since $X+_{f}Y \in qPers(\varphi)$, there exists $G+_{\partial}Y \in qPers(G)$ such that $\partial_{\Pi_{K}} = \partial_{\Pi_{K'}} = f$. Thus, $f_{\Lambda_{K}} = (\partial_{\Pi_{K}})_{\Lambda_{K}} = \partial = (\partial_{\Pi_{K'}})_{\Lambda_{K'}} = f_{\Lambda_{K'}}$.
\end{proof}

And then, we are able to present our first isomorphism of categories:

\begin{teo}$qPers(G)$ and $qPers(\varphi)$ are isomorphic.
\end{teo}

\begin{proof}Let $\Pi: qPers(G) \rightarrow qPers(\varphi)$ and $\Lambda: qPers(\varphi) \rightarrow qPers(G)$ be the restrictions of $\bar{\Pi}_{K}$ and $\bar{\Lambda}_{K}$ respectively (by the definition of $qPers(\varphi)$ the codomain of $\Pi$ is correct and by \textbf{Proposition \ref{identidade1}} the codomain of $\Lambda$ is correct). We have, by \textbf{Propositions \ref{dependencia1}} and \textbf{\ref{dependencia2}}, that such restrictions do not depend of the choice of the compact $K$. And, by \textbf{Propositions \ref{identidade1}} and \textbf{\ref{identidade2}}, it follows that $\Lambda \circ \Pi = id_{qPers(G)}$ and $\Pi \circ \Lambda = id_{qPers(\varphi)}$.
\end{proof}

For geometric and even topological purposes, it is interesting to consider only Hausdorff compactifications of groups and spaces in general. So, we restrict the categories even further.

\subsection{Perspectivity}

\begin{defi} We say that $G+_{\partial}Y \in qPers(G)$ is perspective if it is Hausdorff. We denote by $Pers(G)$ the full subcategory of $qPers(G)$ whose objects are perspectives and $MPers(G)$ the full subcategory of $Pers(G)$ whose objects are metrizable.
\end{defi}

\begin{prop}\label{somahausdorff}If $G+_{\partial}Y$ is perspective, then $X+_{\partial_{\Pi_{K}}}Y$ is Hausdorff.
\end{prop}

\begin{proof}Let $C \subseteq X$ be a compact. Since $\varphi$ is properly discontinuous, we have that $\Pi_{K}(C) = \{g \in G: \varphi(g,K)\cap C \neq \emptyset\}$ is finite and then compact. We have also that $\Lambda_{K}(G) = \varphi(G,K) = X$ and $\partial = (\partial_{\Pi_{K}})_{\Lambda_{K}}$. Then, by the \textbf{Corollary \ref{functorausdorff}} it follows that $X+_{\partial_{\Pi_{K}}}Y$ is Hausdorff.
\end{proof}

To have a direct comprehension of the objects in $qPers(\varphi)$ that comes from objects in $Pers(G)$, let's present a equivalent definition of perspectivity:

\begin{prop}$G+_{\partial}Y$ is perspective if and only if it is Hausdorff and $\forall u \in \mathcal{U}_{\partial}, \ \forall C \subseteq G$ compact, $\#\{g \in G: gC \notin Small(u)\} < \aleph_{0}$, where $\mathcal{U}_{\partial}$ is the only uniform structure compatible with the topology of $G+_{\partial}Y$.

\end{prop}

\begin{proof}$(\Leftarrow)$ Let $F \in Closed(G)$. If $F$ is compact, then $Fg^{-1}$ is compact and then $\partial(F) = \partial(Fg^{-1}) = \emptyset$. So, let's suppose that $\partial(Fg^{-1}) \nsubseteq \partial(F)$. This implies that there exists $x \in \partial(Fg^{-1}) - \partial(F)$. Take $u \in \mathcal{U}_{\partial}$ such that $\forall y \in \partial(F), \ (x,y) \notin u$ (it exists since $\partial(F)$ is compact). Let $v \in \mathcal{U}_{\partial}$ symmetric such that $v^{3} \subseteq u$ and take $A = \mathfrak{B}(x,v)\cap Fg^{-1}$. We have that $A$ is infinite, since $x \in Cl_{G+_{\partial}Y}(Fg^{-1})$. Since $v$ is symmetric, we have that the set $\{h \in G: (h,hg)\notin v\}$ is finite, which implies that there exists $A' \subseteq A$ such that $A - A'$ is finite and $\forall h \in A', (h,hg)\in v$. We have that $A'g$ is infinite (so, it is not compact), $A'g \subseteq F$ and $\partial(A'g) \subseteq \partial(F)$. Since $G+_{\partial}Y$ is compact, we have that $\partial(A'g) \neq \emptyset$. Let $y \in \partial(A'g)$ and $hg \in A'g$ such that $(hg,y)\in v$. So, $(x,h),(h,hg),(hg,y)\in v$, which implies that $(x,y)\in v^{3}\subseteq u$, a contradiction.

Thus, $\forall F \in Closed(G), \ \partial(Fg^{-1}) \subseteq \partial(F)$, which implies that the map $R(\_,g)+id: G+_{\partial}Y \rightarrow G+_{\partial}Y$ is continuous (and $R+id$ is continuous, since $G$ is discrete).

$(\Rightarrow)$ Let $K \subseteq G$ be a compact and $u \in \mathcal{U}_{\partial}$.

Let's consider the case that $K = \{k_{1},k_{2}\}$ (the case $\#K = 1$ is trivial). Let $A = \{g\in G: gK \notin Small(u)\}, \ A_{1} = Ak_{1}$ and $A_{2} = Ak_{2}$. Let's suppose that $\partial(A_{1}) \neq \emptyset$ and let $y \in \partial(A_{1})$. Let $v\in \mathcal{U}_{\partial}$ symmetric such that $v^{2} \subseteq u$. Since $y \in Cl_{G+_{\partial}Y}(A_{1})$, there exists a non empty set $B \subseteq A$ such that $\forall h \in B$, $(hk_{1},y) \in v$ and $\forall h \in A-B, \ (hk_{1},y) \notin v$. Let $B_{1} = Bk_{1}$ and $B_{2} = Bk_{2}$. We have that $B_{1} \subseteq A_{1}$ and $B_{2} \subseteq A_{2}$. We have that $y \in Cl_{G+_{\partial}Y}(B_{1})$, which implies that $y \in f(B_{1})$. Let $g \in G$ such that $k_{1}g = k_{2}$. Since $R(\_,g)+id$ is a homeomorphism, we have that $\forall F \in Closed(G), \ \partial(Fg) = \partial(F)$. In particular, we have that $\partial(B_{1}) = \partial(B_{1}g) = \partial(Bk_{1}g) = \partial(Bk_{2}) = \partial(B_{2})$. Since $y \in \partial(B_{2})$, there exists $h \in B$ such that $(y,hk_{2})\in v$. But $(hk_{1},y) \in v$, by the definition of $B$. So, $(hk_{1},hk_{2}) \in v^{2} \subseteq u$, a contradiction (because $B \subseteq A$). So, $\partial(A_{1}) = \emptyset$, which implies that $A_{1}$ is compact (since $G+_{\partial}Y$ is compact) and then $A_{1}$ is finite. Since $A_{1} = Ak_{1}$, we have that $A$ is finite.

In the general case, we have that $\{g \in G: gK \notin Small(u)\} =$ $\\ \bigcup_{k,k'\in K}\{g \in G: g\{k,k'\}\notin Small(u)\}$, that is finite because it is a finite union (since $K$ is finite) of finite sets.
\end{proof}

Now, we are able to generalize the definition of perspectivity:

\begin{defi}Let $G$ be a group, $X,Y$ Hausdorff spaces with $X$ locally compact and $Y$ compact, $\psi: G \curvearrowright Y$ an action by homeomorphisms and $\varphi: G \curvearrowright X$ a properly discontinuous cocompact action. We say that a compact space of the form $X+_{f}Y$ is perspective if it is Hausdorff, the action $\varphi+\phi: G \curvearrowright X+_{f}Y$ is by homeomorphisms and $\forall u \in \mathcal{U}_{f}, \ \forall K \subseteq X$ compact, $\#\{g \in G: \varphi(g,K)\notin Small(u)\} < \aleph_{0}$, where $\mathcal{U}_{f}$ is the only uniform structure compatible with the topology of $X+_{f}Y$. We denote by $Pers(\varphi)$ the full subcategory of $Comp(\varphi)$ whose objects are perspectives and $MPers(\varphi)$ the full subcategory of $Pers(\varphi)$ whose objects are metrizable.
\end{defi}

\begin{lema}\label{fechopequeno}Let $G+_{\partial}Y \in Comp(G)$ be a Hausdorff space, $\mathcal{U}_{Y}$ the uniform structure of $Y$ and $v \in \mathcal{U}_{Y}$. If $Z \subseteq G$ is an infinite subset, then there exists $Z' \subseteq Z$ infinite such that $\partial(Z') \in Small(v)$.
\end{lema}

\begin{proof}Let $u \in \mathcal{U}_{\partial}$ such that $v = u \cap Y^{2}$ and $u' \in \mathcal{U}_{\partial}$ symmetric such that $u'^{3} \subseteq u$. Since $G+_{\partial}Y$ is compact, there exists $\{C_{1},...C_{n}\}$ cover such that $\forall i \in \{1,...,n\}, \ C_{i} \in Small(u')$. Since $Z$ is infinite, there exists $i_{0}$ such that $Z \cap C_{i_{0}}$ is infinite. Take $Z' = Z \cap C_{i_{0}}$. Let $y,y' \in \partial (Z')$ and $U,U'$ $u'$-small neighbourhoods of $y$ and $y'$, respectively. We have that $Z'\cap U \neq \emptyset$ and $Z'\cap U' \neq \emptyset$. So, take $z \in Z'\cap U \neq \emptyset$ and $z'\in Z'\cap U' \neq \emptyset$. We have that $(y,z)\in u', \ (z,z')\in u'$ and $(z',y') \in u'$, which implies that $(y,y')\in u'^{3} \subseteq u$ and then $(y,y')\in v$.
\end{proof}

\begin{prop}\label{perspectividade1}If $G+_{\partial}Y$ is perspective, then $X+_{\partial_{\Pi_{K}}}Y$ is perspective.
\end{prop}

\begin{proof}Let's suppose that there exists an element $u \in \mathcal{U}_{\partial_{\Pi_{K}}}$ such that the set $Z = \{g \in G: \varphi(g,K)\notin Small(u)\}$ is infinite. Let $v \in \mathcal{U}_{\partial_{\Pi_{K}}}$ symmetric such that $v^{3}\subseteq u$. By the lemma above, there exists $Z' \subseteq Z$ a infinite subset such that $\partial(Z') \in Small(v\cap Y^{2}) \subseteq Small(v)$. Since $\forall z \in Z', \ \varphi(z,K)\notin Small(u)$, take $k_{z},k'_{z}\in K: \ (\varphi(z,k_{z}),\varphi(z,k'_{z})) \notin u$. Take $\mathfrak{B}(\partial_{\Pi_{K}}(\varphi(Z',K)),v)$. We have that $\{\varphi(z,k_{z})\}_{z \in Z'} - \mathfrak{B}(\partial_{\Pi_{K}}(\varphi(Z',K)),v)$ is finite, otherwise the set would have some cluster point outside of $\partial_{\Pi_{K}}(\varphi(Z',K))$, contradicting the fact that $\partial_{\Pi_{K}}(\{\varphi(z,k_{z})\}_{z\in Z'}) \subseteq \partial_{\Pi_{K}}(\varphi(Z',K))$. Analogously, we have that $\{\varphi(z,k'_{z})\}_{z\in Z'} - \mathfrak{B}(\partial_{\Pi_{K}}(\varphi(Z',K)),v)$ is finite. So, there exists $z \in Z'$ such that $\varphi(z,k_{z}),\varphi(z,k'_{z}) \in \mathfrak{B}(\partial_{\Pi_{K}}(\varphi(Z',K)),v)$, which implies that there are $y,y' \in \partial_{\Pi_{K}}(\varphi(Z',K))$ such that $(\varphi(z,k_{z}),y),(y',\varphi(z,k'_{z})\in v$. But $\partial_{\Pi_{K}}(\varphi(Z',K)) = (\partial_{\Pi_{K}})_{\Lambda_{K}}(Z') = \partial(Z') \in Small(v)$, which implies that $(y,y')\in v$. So, $(\varphi(z,k_{z}),\varphi(z,k'_{z})) \in v^{3} \subseteq u$, an absurd. Thus, $Z$ is finite.

Let $K' \subseteq X$ be a compact. We have that $K'' = K \cup K'$ is a fundamental domain. Since $G+_{\partial}Y$ is perspective, we have that $\partial_{ \Pi_{K}} = \partial_{\Pi_{K''}}$, which implies that $\forall u \in \mathcal{U}_{\partial_{\Pi_{K}}} = \mathcal{U}_{\partial_{ \Pi_{K''}}}$, the set $\{g \in G: \varphi(g,K'')\notin Small(u)\}$ is finite. But $\{g \in G: \varphi(g,K')\notin Small(u)\} \subseteq \{g \in G: \varphi(g,K'')\notin Small(u)\}$, which implies that $\{g \in G: \varphi(g,K')\notin Small(u)\}$ is finite. Thus, $X+_{\partial_{\Pi_{K}}}Y$ is perspective.

\end{proof}

Let's suppose that $X+_{f}Y$ is perspective.

\begin{prop}\label{identidade3}$f = (f_{\Lambda_{K}})_{\Pi_{K}}$
\end{prop}

\begin{proof}Let $S \in Closed(X)$. We have that $(f_{\Lambda_{K}})_{\Pi_{K}}(S) = f_{\Lambda_{K}}(\Pi_{K}(S)) = f(\bigcup\{\varphi(g,K): g\in \Pi_{K}(S)\})$ Let $x \in (f_{\Lambda_{K}})_{\Pi_{K}}(S)$ and $u,v \in \mathcal{U}_{f}$ such that $v^{2} \subseteq u$. Since $\{g \in G: \varphi(g,K) \notin Small(v)\}$ is finite, $\exists g \in \Pi_{K}(S)$, $\exists k \in K:$ $(\varphi(g,k),x) \in v$ and $\varphi(g,K) \in Small(v)$. Let $s \in \varphi(g,K)\cap S$ (this set is not empty since $g \in \Pi_{K}(S)$). We have that $(s,\varphi(g,k)),(\varphi(g,k),x) \in v$, which implies that $(s,x) \in v^{2} \subseteq u$. So, $\forall u \in \mathcal{U}_{f}$, there exists $s \in S$ such that $(s,x) \in u$, which implies that $x$ is a cluster point of $S$ (in $X+_{f}Y$) and then $x \in f(S)$. Thus, $(f_{\Lambda_{K}})_{\Pi_{K}}\subseteq f$. Since $f \subseteq (f_{\Lambda_{K}})_{\Pi_{K}}$ (without the hypothesis of perspectivity), it follows that $f = (f_{\Lambda_{K}})_{\Pi_{K}}$.
\end{proof}

A general result:

\begin{prop}\label{preservahausdorff}If $X+_{f}Y \in Comp(\varphi)$ is Hausdorff and $f = (f_{ \Lambda_{K}})_{\Pi_{K}}$, then $G+_{f_{\Lambda_{K}}}Y$ is Hausdorff.
\end{prop}

\begin{proof}Let $C \subseteq G$ be a compact. Since $G$ is discrete, $C$ is finite, which implies that $\Lambda_{K}(C) = \varphi(C,K)$ is compact. We have also that $\Pi_{K}(X) = \{g \in G: \varphi(g,K)\cap X \neq \emptyset\} = G$ and $f = (f_{ \Lambda_{K}})_{\Pi_{K}}$. Thus, by the \textbf{Corollary \ref{functorausdorff}}, it follows that $G+_{f_{\Lambda_{K}}}Y$ is Hausdorff.
\end{proof}

To apply in our case (perspective):

\begin{cor}$G+_{f_{\Lambda_{K}}}Y$ is Hausdorff.
\eod\end{cor}

$f_{\Lambda_{K}}$ does not depend of the choice of the compact $K$:

\begin{cor}\label{dependencia3}Let $K,K' \subseteq X$ be fundamental domains. Then, $f_{\Lambda_{K}} = f_{\Lambda_{K'}}$.
\end{cor}

\begin{proof}Let $K'' = K \cup K'$. We have that $K''$ is a fundamental domain. Let $F \in Closed(G)$. We have that $f_{\Lambda_{K''}}(F) = f(\varphi(F,K'')) \supseteq f(\varphi(F,K)) = f_{\Lambda_{K}}(F)$. But $f_{\Lambda_{K}} \subseteq f_{\Lambda_{K''}}$ implies that $id: G+_{f_{\Lambda_{K}}}Y \rightarrow G+_{f_{\Lambda_{K''}}}Y$ is continuous. Since the spaces are Hausdorff and compact, it follows that $id$ is a homeomorphism, which implies that $f_{\Lambda_{K}} = f_{\Lambda_{K''}}$. Analogously, we have that $f_{\Lambda_{K'}} = f_{\Lambda_{K''}}$. Thus, $f_{\Lambda_{K}} = f_{\Lambda_{K'}}$.
\end{proof}

\begin{prop}\label{perspectividade2}$G+_{f_{\Lambda_{K}}}Y$ is perspective.
\end{prop}

\begin{proof}We already have that $G+_{f_{\Lambda_{K}}}Y$ is Hausdorff. Let $F \in Closed(G)$. We have that $f_{\Lambda_{K}}(Fg) = f(\varphi(Fg,K)) = f(\varphi(F,\varphi(g,K))) = f_{\Lambda_{\varphi(g,K)}}(F)$, since $\varphi(g,K)$ is a compact such that $\varphi(G,\varphi(g,K)) = X$. However, we have that $f_{\Lambda_{K}} = f_{\Lambda_{\varphi(g,K)}}$, which implies that $f_{\Lambda_{K}}(Fg) = f_{\Lambda_{K}}(F)$. Thus, $G+_{f_{\Lambda_{K}}}Y$ is perspective.
\end{proof}

\begin{prop}$Pers(\varphi)$ is a full subcategory of $qPers(\varphi)$.
\end{prop}

\begin{proof}Let $X+_{f}Y \in Pers(\varphi)$. By the last proposition $G+_{f_{\Lambda_{K}}}Y \in Pers(G) \subseteq qPers(G)$, which implies that $X+_{(f_{\Lambda_{K}})_{\Pi_{K}}}Y \in qPers(\varphi)$. But, by the \textbf{Proposition \ref{identidade3}}, $(f_{\Lambda_{K}})_{\Pi_{K}} = f$, which implies that $X+_{f}Y \in qPers(\varphi)$. So, all objects of $Pers(\varphi)$ are in $qPers(\varphi)$. Since all morphisms are always maintained, we have that $Pers(\varphi)$ is a full subcategory of $qPers(\varphi)$.
\end{proof}

Now we are able to present the correspondence between the two forms of perspectivity:

\begin{teo}\label{main}Let $G$ be a group, $X$ a Hausdorff  locally compact space and $\varphi: G \curvearrowright X$ a properly discontinuous cocompact action. Then, the categories $Pers(G)$ and $Pers(\varphi)$ are isomorphic.
\end{teo}

\begin{proof}Take the functors $\Pi$ and $\Lambda$. By \textbf{Propositions \ref{perspectividade1}} and \textbf{\ref{perspectividade2}}, both send perspectivities to perspectivities. Let's take $\Pi': Pers(G) \rightarrow Pers(\varphi)$ and $\Lambda': Pers(\varphi) \rightarrow Pers(G)$ the restrictions of $\Pi$ and $\Lambda$,  respectively, and we have that $\Lambda' \circ \Pi' = id_{Pers(G)}$ and $\Pi' \circ \Lambda' = id_{Pers(\varphi)}$.
\end{proof}

\begin{cor}Let $G$ be a countable group, $X$ a locally compact Hausdorff space with countable basis and $\varphi: G \curvearrowright X$ a properly discontinuous cocompact action. Then, the categories $MPers(G)$ and $MPers(\varphi)$ are isomorphic.
\end{cor}

\begin{proof}Take $\Pi$ and $\Lambda$ as the theorem above. Since $G$ and $X$ have countable basis, we have, by \textbf{Corollary \ref{uniaometrizavel}}, that every Hausdorff compact space of the form $G+_{\partial}Y$ or the form $X+_{f}Y$ is metrizable. So, $\Pi$ and $\Lambda$ send metrizable spaces to metrizable spaces. Let's take $\Pi'': MPers(G) \rightarrow MPers(\varphi)$ and $\Lambda'': MPers(\varphi) \rightarrow MPers(G)$ the restrictions of $\Pi$ and $\Lambda$,  respectively, and we have that $\Lambda'' \circ \Pi'' = id_{MPers(G)}$ and $\Pi'' \circ \Lambda'' = id_{MPers(\varphi)}$.
\end{proof}

At last, a proposition showing that all Hausdorff objects of $qPers(\varphi)$ belongs to $Pers(\varphi)$:

\begin{prop}Let $X+_{f}Y \in qPers(\varphi)$. If it is Hausdorff, then it belongs to $Pers(\varphi)$.
\end{prop}

\begin{proof}Since $X+_{f}Y \in qPers(\varphi)$, we have that $f = (f_{\Lambda_{K}})_{\Pi_{K}}$ (\textbf{Proposition \ref{identidade2}}). So, by \textbf{Proposition \ref{preservahausdorff}}, $G+_{f_{\Lambda_{K}}}Y$ is Hausdorff, which implies that $G+_{f_{\Lambda_{K}}}Y \in Pers(G)$. Thus, $X+_{f}Y = X+_{(f_{\Lambda_{K}})_{\Pi_{K}}}Y \in Pers(\varphi)$.
\end{proof}

\section{More about (quasi)-perspectivity}

\subsection{Non extendability of the isomorphism}

Here we see that, on the subcategories where the functors agree (do not depend of the choice of the fundamental domain), it is not possible to extend $qPers(G)$ and $qPers(\varphi)$ and still obtain an isomorphism of categories:

\begin{prop}Let $\psi: G \curvearrowright Y$, $G+_{\partial}Y \in Comp(G)$ and $\varphi: G\curvearrowright X$. If $\forall K$ fundamental domain in $X$, $\partial = (\partial_{\Pi_{K}})_{\Lambda_{K}}$, then $G+_{\partial}Y  \in qPers(G)$.
\end{prop}

\begin{proof}Let $K$ be a fundamental domain, $k \in K$ and $g \in G$. Let's define $K' = K \cup \{\varphi(g,k)\}$. We have that $\forall F \in Closed(G), \ \partial(F) = (\partial_{\Pi_{K'}})_{\Lambda_{K'}}(F) = \bigcup_{z\in Z_{K'}}\partial(Fz)$. But $g \in Z_{K'}$, since $\varphi(g,k)\in \varphi(g,K')\cap K'$, which implies that $\partial(Fg)\subseteq \partial(F)$. However, $\partial(Fg) = (\partial_{\Pi_{K'}})_{\Lambda_{K'}}(Fg) = \bigcup_{z\in Z_{K'}}\partial(Fgz)$, which implies that $\partial(F) = \partial(Fgg^{-1}) \subseteq \partial(Fg)$, because $g^{-1} \in Z_{K'}$, since $k \in \varphi(g,K')\cap K'$. Thus, $\forall F \in Closed(G), \ \forall g \in G, \ \partial(Fg) = \partial(F)$, which implies that $G+_{\partial}Y \in qPers(G)$.
\end{proof}

\begin{cor}If $G+_{\partial}Y$ is Hausdorff and $\forall K$ fundamental domain in $X$, $\partial = (\partial_{\Pi_{K}})_{\Lambda_{K}}$, then $G+_{\partial}Y  \in Pers(G)$.
\eod\end{cor}

So, the property of quasi-perspectivity (or perspectivity if considered just Hausdorff spaces) is the most general one that expects a correspondence between the compactifications for every chosen fundamental domain.

\subsection{Adjunction}

\begin{defi}Let $\psi: G \curvearrowright Y$ be an action by homeomorphisms and $G+_{\partial}Y \in Comp(G)$. We define the map $\tilde{\partial}: Closed(G) \rightarrow Closed(Y)$ as $\tilde{\partial}(F) = Cl_{Y}(\bigcup_{g\in G}\partial(Fg))$.
\end{defi}

\begin{prop}$\tilde{\partial}$ is admissible.
\end{prop}

\begin{proof}We have that $\tilde{\partial}(\emptyset) = Cl_{Y}(\bigcup_{g\in G}\partial(\emptyset g)) = Cl_{Y}(\bigcup_{g\in G}\partial(\emptyset)) =$ $\\ Cl_{Y}(\bigcup_{g\in G}\emptyset) = Cl_{Y}(\emptyset) = \emptyset$. Let $F_{1},F_{2} \in Closed(G)$. We have that $\tilde{\partial}(F_{1} \cup F_{2}) = Cl_{Y}(\bigcup_{g\in G}\partial((F_{1} \cup F_{2}) g)) = Cl_{Y}(\bigcup_{g\in G}\partial(F_{1}g \cup F_{2} g)) =$ $\\ Cl_{Y}(\bigcup_{g\in G}\partial(F_{1}g) \cup \partial(F_{2} g)) = Cl_{Y}(\bigcup_{g\in G}\partial(F_{1}g) \cup \bigcup_{g\in G}\partial(F_{1}g)) =$ $\\ Cl_{Y}(\bigcup_{g\in G}\partial(F_{1}g)) \cup Cl_{Y}(\bigcup_{g\in G}\partial(F_{1}g)) = \tilde{\partial}(F_{1}) \cup \tilde{\partial}(F_{2})$. Thus, $\tilde{\partial}$ is admissible.
\end{proof}

We immediately have that $id+id: G+_{\partial}Y \rightarrow G+_{\tilde{\partial}}Y$  is continuous, since $\partial \subseteq \tilde{\partial}$. We have also that if $G+_{\partial}Y \in qPers(G)$, then $\tilde{\partial} = \partial$, since $\forall F \in Closed(G), \ \forall g \in G, \ \partial(Fg) = \partial(F)$.

\begin{prop}$G+_{\tilde{\partial}}Y \in Comp(G)$.
\end{prop}

\begin{proof}Since $id+id: G+_{\partial}Y \rightarrow G+_{\tilde{\partial}}Y$  is continuous and $G+_{\partial}Y$ is compact, it follows that $G+_{\tilde{\partial}}Y$ is compact.

Let $F \in Closed(G)$ and $h \in G$. So, $\tilde{\partial}(hF) = Cl_{Y}(\bigcup_{g\in G}\partial(hFg))$. Since $L+\psi: G \curvearrowright G+_{\partial}Y$ is continuous, we have that, $\forall g \in G$, $\partial(hFg) = \psi(h,\partial(Fg))$. So, $\tilde{\partial}(hF) = Cl_{Y}(\bigcup_{g\in G}\psi(h,\partial(Fg))) = Cl_{Y}(\psi(h,\bigcup_{g\in G}\partial(Fg))) = \psi(h, Cl_{Y}(\bigcup_{g\in G}\partial(Fg))) = \psi(h,\tilde{\partial}(F))$, which implies that $L+\psi: G \curvearrowright G+_{\tilde{\partial}}Y$ is continuous.

Thus, $G+_{\tilde{\partial}}Y \in Comp(G)$.
\end{proof}

\begin{prop}$G+_{\tilde{\partial}}Y \in qPers(G)$.
\end{prop}

\begin{proof}Let $F$ be a subset of $G$ and $h$ an element of $G$. We have that $\tilde{\partial}(Fh) = Cl_{Y}(\bigcup_{g\in G}\partial(Fhg)) = Cl_{Y}(\bigcup_{g\in G}\partial(Fg)) = \tilde{\partial}(F)$. Thus, $G+_{\tilde{\partial}}Y \in qPers(G)$.
\end{proof}

\begin{prop}If $id+\phi: G+_{\partial_{1}}Y_{1} \rightarrow  G+_{\partial_{2}}Y_{2}$ is continuous, then $id+\phi: G+_{\tilde{\partial_{1}}}Y_{1} \rightarrow  G+_{\tilde{\partial_{2}}}Y_{2}$ is continuous.
\end{prop}

\begin{proof}Let $F \in Closed(G)$. We have that $\tilde{\partial_{1}}(F) = Cl_{Y}(\bigcup_{g\in G}\partial_{1}(Fg))$. But $\partial_{1}(Fg) \subseteq \phi^{-1}(\partial_{2}(Fg))$, because $id+\phi: G+_{\partial_{1}}Y_{1} \rightarrow  G+_{\partial_{2}}Y_{2}$ is continuous. So, $\tilde{\partial_{1}}(F) \subseteq Cl_{Y}(\bigcup_{g\in G}\phi^{-1}(\partial_{2}(Fg))) = Cl_{Y}(\phi^{-1}(\bigcup_{g\in G}\partial_{2}(Fg))) \subseteq \phi^{-1}(Cl_{Y}(\bigcup_{g\in G}\partial_{2}(Fg))) = \phi^{-1}(\tilde{\partial_{2}}(F))$. So, we have the diagram:

$$ \xymatrix{  Closed(G) \ar[r]^{\tilde{\partial_{1}}}  & Closed(Y_{1}) \ar@{}[dl]|{\subseteq} \\
            Closed(G) \ar[r]^{\tilde{\partial_{2}}} \ar[u]^{id} & Closed(Y_{2}) \ar[u]^{\phi^{-1}} }$$

Thus, $id+\phi: G+_{\tilde{\partial_{1}}}Y_{1} \rightarrow  G+_{\tilde{\partial_{2}}}Y_{2}$ is continuous.

\end{proof}

So, we have a functor $\mathcal{P}: Comp(G) \rightarrow qPers(G)$ such that $\mathcal{P}(G+_{\partial}Y) = G+_{\tilde{\partial}}Y$ and, for $id+\phi: G+_{\partial_{1}}Y_{1} \rightarrow  G+_{\partial_{2}}Y_{2}$ a morphism, $\mathcal{P}(id+\phi) = id+\phi: G+_{\tilde{\partial_{1}}}Y_{1} \rightarrow  G+_{\tilde{\partial_{2}}}Y_{2}$.

\begin{prop}Let $G+_{\partial_{1}}Y_{1} \in Comp(G), \ G+_{\partial_{2}}Y_{2} \in qPers(G)$ and $id+\phi: \ G+_{\partial_{1}}Y_{1} \ \rightarrow \ G+_{\partial_{2}}Y_{2}$ be a continuous map, Then, the map $id+\phi: G+_{\tilde{\partial_{1}}}Y_{1} \rightarrow G+_{\partial_{2}}Y_{2}$ is also continuous.
\end{prop}

\begin{proof}Since $id+\phi: G+_{\partial_{1}}Y_{1} \rightarrow G+_{\partial_{2}}Y_{2}$ is continuous, we have that $id+\phi: G+_{\tilde{\partial_{1}}}Y_{1} \rightarrow G+_{\tilde{\partial_{2}}}Y_{2}$ is continuous. But $G+_{\partial_{2}}Y_{2} \in qPers(G)$, which implies that $\tilde{\partial_{2}} = \partial_{2}$. Thus, $id+\phi: G+_{\tilde{\partial_{1}}}Y_{1} \rightarrow G+_{\partial_{2}}Y_{2}$ is continuous.
\end{proof}

So, $id+\phi: G+_{\tilde{\partial_{1}}}Y_{1} \rightarrow G+_{\partial_{2}}Y_{2}$ is the only continuous map that commutes the diagram:

$$ \xymatrix{  G+_{\partial_{1}}Y_{1} \ar[r]^{id+id} \ar[d]^{id+\phi} & G+_{\tilde{\partial_{1}}}Y_{1} \ar[ld]^{id+\phi} \\
            G+_{\partial_{2}}Y_{2}  & }$$

So, for $\mathcal{I}: qPers(G) \rightarrow Comp(G)$ the inclusion functor, we have that $G+_{\tilde{\partial_{1}}}Y_{1}$ with the map $id+id: G+_{\partial_{1}}Y_{1} \rightarrow G+_{\tilde{\partial_{1}}}Y_{1}$, form a reflection of $G+_{\partial_{1}}Y_{1}$ along $\mathcal{I}$. Thus, we have:

\begin{teo}$\mathcal{P}$ is left adjoint to $\mathcal{I}$.
\eod\end{teo}

A similar thing happens between the categories $Comp(\varphi)$ and $qPers(\varphi)$. Let the functor $\mathcal{P}_{K}': Comp(\varphi) \rightarrow qPers(\varphi)$ be defined by $\mathcal{P}_{K}' = \Pi \circ \mathcal{P} \circ \Lambda_{K}$ and $\mathcal{I}': qPers(\varphi) \rightarrow Comp(\varphi)$ be the inclusion functor.

\begin{prop}Let $X+_{f}Y \in Comp(\varphi)$. Then, the application $\\ id+id: X+_{f}Y \rightarrow \mathcal{I}' \circ \mathcal{P}_{K}'(X+_{f}Y) = X+_{(\widetilde{f_{\Lambda_{K}}})_{\Pi_{K}}}Y$ is continuous.
\end{prop}

\begin{proof}We have that $f_{\Lambda_{K}} \subseteq \widetilde{f_{\Lambda_{K}}}$, which implies that $f \subseteq (f_{\Lambda_{K}})_{\Pi_{K}} \subseteq (\widetilde{f_{\Lambda_{K}}})_{\Pi_{K}}$ and then $id+id: X+_{f}Y \rightarrow X+_{(\widetilde{f_{\Lambda_{K}}})_{\Pi_{K}}}Y$ is continuous.
\end{proof}

\begin{prop}Let $X+_{f_{1}}Y_{1} \in Comp(\varphi), \ X+_{f_{2}}Y_{2} \in qPers(\varphi)$ and $id+\phi: X+_{f_{1}}Y_{1} \rightarrow X+_{f_{2}}Y_{2}$ be a continuous map, Then, the application $id+\phi: X+_{(\widetilde{f_{1\Lambda_{K}}})_{\Pi_{K}}}Y_{1} \rightarrow X+_{f_{2}}Y_{2}$ is also continuous.
\end{prop}

\begin{proof}Since $id+\phi: X+_{f_{1}}Y_{1} \rightarrow X+_{f_{2}}Y_{2}$ is continuous, we have that $id+\phi: G+_{f_{1\Lambda_{K}}}Y_{1} \rightarrow G+_{f_{2\Lambda_{K}}}Y_{2}$ is continuous, which implies that the map $id+\phi: G+_{\widetilde{f_{1\Lambda_{K}}}}Y_{1} \rightarrow G+_{\widetilde{f_{2\Lambda_{K}}}}Y_{2} = G+_{f_{2\Lambda_{K}}}Y_{2}$ is continuous, which implies that $id+\phi: X+_{(\widetilde{f_{1\Lambda_{K}}})_{\Pi_{K}}}Y_{1} \rightarrow X+_{(f_{2\Lambda_{K}})_{\Pi_{K}}}Y_{2} = X+_{f_{2}}Y_{2}$ is continuous.
\end{proof}

So, $id+\phi: X+_{(\widetilde{f_{1\Lambda_{K}}})_{\Pi_{K}}}Y_{1} \rightarrow X+_{f_{2}}Y_{2}$ is the only continuous map that commutes the diagram:

$$ \xymatrix{  X+_{f_{1}}Y_{1} \ar[r]^<<{ \ \ \ \ \ id+id} \ar[d]^{id+\phi} & X+_{(\widetilde{f_{\Lambda_{K}}})_{\Pi_{K}}} \!\!\! Y \ar[ld]^{id+\phi} \\
            X+_{f_{2}}Y_{2}  & }$$

Hence, $X+_{(\widetilde{f_{\Lambda_{K}}})_{\Pi_{K}}}Y$ with the map $id+id: X+_{f_{1}}Y_{1} \rightarrow X+_{(\widetilde{f_{\Lambda_{K}}})_{\Pi_{K}}}Y$, form a reflection of $X+_{f_{1}}Y_{1}$ along $\mathcal{I}'$. Thus, we have:

\begin{teo} $\mathcal{P}_{K}'$ is left adjoint to $\mathcal{I}'$.
\eod\end{teo}

\begin{obs}By the uniqueness of the adjunction, we have that $\mathcal{P}_{K}'$ do not depend of the choice of the fundamental domain $K$.
\end{obs}

\subsection{Subspaces and quotients}

\begin{prop}Let $\varphi: G \curvearrowright Y$ and $G+_{\partial}Y$ be a perspectivity, $H < G$ and $F \in Closed(Y)$ such that $\partial(H) \subseteq F$. So, $\varphi|_{H\times F}: H \curvearrowright F$ and $H+_{\partial^{\ast}}F$ form a perspectivity, where $\partial^{\ast}$ is a pullback with respect to the inclusions of $H$ in $G$ and $F$ in $Y$.
\end{prop}

\begin{proof}The topology of $H+_{\partial^{\ast}}F$ coincides with the topology of $H \cup F$ as a subspace of $G+_{\partial}Y$. So, $H+_{\partial^{\ast}}F$ is Hausdorff. Since $\partial(H) \subseteq F$, we have that $H\cup F \in Closed(G+_{\partial}Y)$, with implies that it is compact. And since $L+\varphi$ and $R+id$ are continuous, we have that $(L+\varphi)|_{H\times (H+_{\partial^{\ast}}F)}$ and $(R+id)|_{H\times (H+_{\partial^{\ast}}F)}$ are continuous. Thus, $H+_{\partial^{\ast}}F$ is perspective.
\end{proof}

\begin{prop}Let, for $i =1,2$, $\varphi_{i}+\psi_{i}: G \curvearrowright X_{i}+_{f_{i}}Y_{i}$ be actions by homeomorphisms on Hausdorff spaces and $m+n: X_{1}+_{f_{1}}Y_{1} \rightarrow X_{2}+_{f_{2}}Y_{2}$ a continuous and surjective map with $m$ equivariant. If $X_{1}+_{f_{1}}Y_{1}$ is perspective, then $X_{2}+_{f_{2}}Y_{2}$ is perspective.
\end{prop}

\begin{proof}Let $K$ be a compact subset of $X_{2}$ and $u$ an element of $\mathcal{U}_{f_{2}}$. We have that $(m+n)^{-1}(K)\in Closed(X_{1}+_{f_{1}}Y_{1})$, which implies that it is compact. But $(m+n)^{-1}(K) = m^{-1}(K) \subseteq X$. Since $X_{1}+_{f_{1}}Y_{1}$ is perspective, the set $\{g \in G: \varphi_{1}(g,m^{-1}(K))\notin Small(((m+n)^{2})^{-1}(u))\}$ is finite.

Let $h \in \{g \in G: \varphi_{2}(g,K)\notin Small(u)\}$. Then, $\varphi_{2}(h,K)\notin Small(u)$, which implies that $(m+n)^{-1}(\varphi_{2}(h,K)) \notin Small(((m+n)^{2})^{-1}(u))$. However, $(m+n)^{-1}(\varphi_{2}(h,K)) = m^{-1}(\varphi_{2}(h,K)) = \varphi_{1}(h,m^{-1}(K))$. So, $\varphi_{1}(h,m^{-1}(K)) \notin Small(((m+n)^{2})^{-1}(u))$, which implies that $h \in \{g \in G: \varphi_{1}(g,m^{-1}(K))\notin Small(((m+n)^{2})^{-1}(u))\}$. So, $\{g \in G: \varphi_{2}(g,K)\notin Small(u)\} \subseteq$ $\\ \{g \in G: \varphi_{1}(g,m^{-1}(K))\notin Small(((m+n)^{2})^{-1}(u))\}$, which implies that it is finite.

Thus, $X_{2}+_{f_{2}}Y_{2}$ is perspective.

\end{proof}

So, the closure of a subgroup and equivariant quotients maintain the perspective property, the same behavior that appears on the convergence case.

\subsection{Limits}

Let $\varphi: G \curvearrowright X$ be a properly discontinuous  cocompact action on a locally compact Hausdorff space $X$ and $F: \mathcal{C} \rightarrow T_{2}Comp(\varphi)$ be a covariant functor, where $\mathcal{C}$ is a small category and $T_{2}Comp(\varphi)$ is the full subcategory of $Comp(\varphi)$ whose objects are Hausdorff. If $\tilde{F}: \mathcal{C} \rightarrow SUM(X)$ is a functor that do the same as $F$ but forgetting the action, then it is easy to see that $\lim\limits_{\longleftarrow} \tilde{F}$ has a natural action of $G$ by homeomorphisms that extends $\varphi$ and this pair is the limit of $F$.

\begin{prop}If $F: \mathcal{C} \rightarrow SUM(X)$ is a covariant functor such that $\mathcal{C}$ is small and $\forall c \in \mathcal{C}, \ F(c)$ is compact with uniform structure $\mathcal{U}_{c}$, we have that $\forall K \subseteq X, \ \forall u \in \mathcal{U}_{c}, \ \{g \in G: \varphi(g,K) \notin Small(u)\}$ is finite, then $\forall K \subseteq X, \forall u \in \mathcal{U}, \ \{g \in G: \varphi(g,K) \notin Small(u)\}$ is finite, where $\mathcal{U}$ is the uniform structure of $\lim\limits_{\longleftarrow} F$.
\end{prop}

\begin{obs}Since $F(c)$ is compact $\forall c \in \mathcal{C}$, we have that $\lim\limits_{\longleftarrow} F$ is compact, and then, it has a unique uniform structure.
\end{obs}

\begin{proof}If $\pi_{c}: \lim\limits_{\longleftarrow} F \rightarrow F(c)$ are the projection maps, we have that the set $\{\pi_{c_{1}}^{-1}(u_{1})\cap ... \cap\pi_{c_{n}}^{-1}(u_{n}): u_{i} \in \mathcal{U}_{c_{i}}\}$ is a basis for $\mathcal{U}$.

Let $K \subseteq X$ be a compact and $u$ an element of $\mathcal{U}$. There exists $c_{1},...,c_{n} \in \mathcal{C}$ and $u_{i} \in \mathcal{U}_{c_{i}}:$ $\pi_{c_{1}}^{-1}(u_{1})\cap ... \cap\pi_{c_{n}}^{-1}(u_{n}) \subseteq u$. We have that, $\forall i \in \{1,...,n\}$, the set $\{g \in G: \varphi(g,K) \notin Small(u_{i})\}$ is finite. If $g \in G:$ $\varphi(g,K)\notin Small(\pi_{c_{i}}^{-1}(u_{i}))$, then $\varphi(g,K) = \pi_{c}(\varphi(g,K)) \notin Small(u_{i})$, which implies that $\{g \in G: \varphi(g,K) \notin Small(\pi_{c_{i}}^{-1}(u_{i}))\} \subseteq \{g \in G: \varphi(g,K) \notin Small(u_{i})\}$ that is finite. However, $\{g \in G: \varphi(g,K) \notin Small(\pi_{c_{1}}^{-1}(u_{1})\cap ... \cap \pi_{c_{n}}^{-1}(u_{n}))\} \subseteq \bigcup_{i = 1}^{n}\{g \in G : \varphi(g,K) \notin Small(\pi_{c_{i}}^{-1}(u_{i}))\}$  (in a fact, if $g \in G$ such that $\varphi(g,K) \notin Small(\pi_{c_{1}}^{-1}(u_{1})\cap ... \cap \pi_{c_{n}}^{-1}(u_{n}))$, then $\exists x,y \in \varphi(g,K):$  $(x,y) \notin \pi_{c_{1}}^{-1}(u_{1})\cap ... \cap \pi_{c_{n}}^{-1}(u_{n})$, which implies that $\exists i\in \{1,...,n\}:$ $(x,y) \notin \pi_{c_{i}}^{-1}(u_{i})$ and then $\varphi(g,K) \notin Small(\pi_{c_{i}}^{-1}(u_{i}))$) that is also finite. However, the set $\{g: \varphi(g,K) \notin Small(u)\} \subseteq \{g: \varphi(g,K) \notin Small(\pi_{c_{1}}^{-1}(u_{1})\cap ... \cap \pi_{c_{n}}^{-1}(u_{n}))\}$, which implies that it is finite.

Thus, $\forall K \subseteq X, \forall u \in \mathcal{U}, \ \{g \in G: \varphi(g,K) \notin Small(u)\}$ is finite.

\end{proof}

\begin{cor}Let $F: \mathcal{C} \rightarrow Pers(X)$ be a covariant functor such that $\mathcal{C}$ is small. Then, there exists $\lim\limits_{\longleftarrow} F$.\end{cor}

\begin{proof}Take $\bar{F}: \mathcal{C} \rightarrow T_{2}Comp(X)$ and $\tilde{F}: \mathcal{C} \rightarrow SUM(X)$ be the functors that do the same as $F$. We have that $\lim\limits_{\longleftarrow} \bar{F}$ exists, its space is Hausdorff and equal to $\lim\limits_{\longleftarrow} \tilde{F}$. So, by the last proposition, it is perspective, which implies that $\lim\limits_{\longleftarrow} \bar{F} \in Pers(\varphi)$. Thus, there exists $\lim\limits_{\longleftarrow} F$ and $\lim\limits_{\longleftarrow} F = \lim\limits_{\longleftarrow} \bar{F}$.
\end{proof}

\subsection{CAT(0)}

Let $(X,d)$ be a CAT(0) space and $x \in X$. We define, for $r > s > 0$, $\pi_{rs}:  Cl_{X}(\mathfrak{B}(x,r)) \rightarrow Cl_{X}(\mathfrak{B}(x,s))$ as $\pi_{rs}(y) = y$ if $y \in Cl_{X}(\mathfrak{B}(x,s))$ and $\pi_{rs}(y)$ the unique point on the set $[x,y]\cap \mathcal{S}(x,s)$, where $[x,y]$ is the unique geodesic with this two extreme points and $\mathcal{S}(x,r) = \{z \in X: d(x,z) = r\}$. This forms a codirected set and the limit is the compactification of $X$ with its visual boundary: $X+_{f}\partial X$, for some admissible map $f$. The inducted maps $\pi_{r}:  X+_{f}\partial X \rightarrow Cl_{X}(\mathfrak{B}(x,r))$ are just the projections when restricted to $X$.

\begin{prop}The set $\{w_{r,\epsilon}: r,\epsilon > 0\}$ is a basis of the uniform structure $\U_{f}$ of $X+_{f}\partial X$, where $w_{r,\epsilon} = \{(y,z) \in (X+_{f}\partial X)^2: d(\pi_{r}(y),\pi_{r}(z)) < \epsilon\}$.
\end{prop}

\begin{proof}Since this limit comes from a codirected set, we have that, if $\U_{r}$ is the uniform structure of $Cl_{X}(\mathfrak{B}(x,r))$ and $B_{r}$ is a basis of $\U_{r}$, the uniform structure of $X+_{f}\partial X$ has a basis $\{(\pi_{r}^{2})^{-1}(u): u \in B_{r}, r > 0\}$. Take $B_{r} = \{u_{r,\epsilon}\}_{\epsilon > 0}$, where $u_{r,\epsilon} = \{(y,z) \in Cl_{X}(\mathfrak{B}(x,r))^{2}: d(y,z) < \epsilon\}$. So, the set $\{(\pi_{r}^{2})^{-1}(u_{r,\epsilon}): r > 0, \epsilon > 0\}$ is a basis of $\U_{f}$. However, $(\pi_{r}^{2})^{-1}(u_{r,\epsilon}) = \{(y,z) \in (X+_{f}\partial X)^2: (\pi_{r}(y),\pi_{r}(z)) \in u_{r,\epsilon}\} =$ $\\ \{(y,z) \in (X+_{f}\partial X)^2: d(\pi_{r}(y),\pi_{r}(z)) < \epsilon \} = w_{r,\epsilon}$. Thus, $\{w_{r,\epsilon}: r,\epsilon > 0\}$ is a basis of $\U_{f}$.
\end{proof}

\begin{lema}Let $(\R^{2},d)$ be the Euclidean plane and let's fix $x \in \R^{2}$ as the based point. Let $p,q \in \R^{2}$, and $r,\epsilon > 0$. If there is $d > 0$ such that $d(p,q) \leqslant d, \ d(x,p) \geqslant \frac{dr}{\epsilon}$ and $d(x,q) \geqslant \frac{dr}{\epsilon}$, then $(p,q) \in w_{r,\epsilon}$.
\end{lema}

\begin{proof}Let $\theta = \angle pxq$, $a = d(x,p), \ b = d(x,q)$ and $c = d(p,q)$. By the cosine law on the triangle $\Delta pxq$, we have that $c^{2} = a^{2}+b^{2} - 2ab cos(\theta)$, which implies that $-2cos(\theta) = \frac{c^{2} -  a^{2} - b^{2}}{ab}$ and then $2(1-cos(\theta)) = \frac{c^{2} - a^{2} - b^{2} + 2ab}{ab} = \frac{c^{2} -  (a-b)^{2}}{ab} \leqslant \frac{c^{2}}{ab}$. But $c \leqslant d, \ a \geqslant \frac{dr}{\epsilon}$ and $b \geqslant \frac{dr}{\epsilon}$, which implies that $\\ 2(1-cos(\theta))\leqslant \frac{d^{2}}{\frac{d^{2}r^{2}}{\epsilon^{2}}} = \frac{\epsilon^{2}}{r^{2}}$. Let $y$ be a point in $[x,p]$ and $z$ a point in $[x,q]$ such that $d(x,y) = d(x,z) = r$. Applying the cosine law on the triangle $\Delta yxz$, we have that $d(y,z)^{2} = d(x,y)^{2}+d(x,z)^{2} -2d(x,y)d(x,z)cos(\theta) = 2r^{2}-2r^{2}cos(\theta) = 2r^{2}(1-cos(\theta)) \leqslant r^{2}\frac{\epsilon^{2}}{r^{2}} = \epsilon^{2}$. Thus, $d(y,z) \leqslant \epsilon$, which implies that $(p,q) \in w_{r,\epsilon}$.
\end{proof}

\begin{lema}Let $(X,d)$ be a CAT(0) space and let's fix $x \in X$ as the based point. Let $p,q \in X$, and $r,\epsilon > 0$. If there is $d > 0$ such that $d(p,q) \leqslant d$, $d(x,p) \geqslant \frac{dr}{\epsilon}$ and $d(x,q) \geqslant \frac{dr}{\epsilon}$, then $(p,q) \in w_{r,\epsilon}$.
\end{lema}

\begin{proof}Let's consider the triangle $\Delta xpq \subseteq (X,d)$ and its comparative triangle $\Delta \bar{x}\bar{p}\bar{q} \subseteq (\R^{2},\bar{d})$. Let's suppose that there are $d,r,\epsilon > 0$: $d(p,q) \leqslant d$, $d(x,p) \geqslant \frac{dr}{\epsilon}$ and $d(x,q) \geqslant \frac{dr}{\epsilon}$. We have that $\bar{d}(\bar{p},\bar{q}) = d(p,q) \leqslant d$,  $\bar{d}(\bar{x},\bar{p}) = d(x,p) \geqslant \frac{dr}{\epsilon}$ and $\bar{d}(\bar{x},\bar{q}) \geqslant \frac{dr}{\epsilon}$, which implies, by the lemma above, that if $z_{1} \in [\bar{x},\bar{p}]$ and $z_{2} \in [\bar{x},\bar{q}]$ are such that $\bar{d}(\bar{x},z_{1}) = \bar{d}(\bar{x},z_{2}) = r$, then $\bar{d}(z_{1},z_{2}) < \epsilon$. But $z_{1} = \overline{\pi_{r}(p)}$ and $z_{2} = \overline{\pi_{r}(q)}$, which implies that $d(\pi_{r}(p),\pi_{r}(p)) \leqslant d(z_{1},z_{2}) < \epsilon$, because $X$ is CAT(0). Thus, $(p,q) \in w_{r,\epsilon}$.
\end{proof}

\begin{prop}Let $(X,d)$ be a CAT(0) space, $x \in X$, $X+_{f}\partial X$ its compactification with its visual boundary with respect to $x$ and $\U_{f}$ the uniform structure of the compactification. Let $\varphi: G \curvearrowright X$ be an action that is properly discontinuous, cocompact and by isometries. Then, $X+_{f}\partial X \in Pers(\varphi)$.
\end{prop}

\begin{proof}Let $K$ be a compact subspace of $X$ and $u \in \mathcal{U}_{f}$. There are $r,\epsilon > 0:$ $w_{r\epsilon} \subseteq u$. Let $d = diam \ K$ and $g \in G$ such that $d(\varphi(g,K),x) \geqslant \frac{dr}{\epsilon}$. If $k_{1},k_{2} \in K$, then $d(\varphi(g,k_{1}),\varphi(g,k_{2})) = d(k_{1},k_{2}) \leqslant d, \ d(x,\varphi(g,k_{1})) \geqslant \frac{dr}{\epsilon}$ and $d(x,\varphi(g,k_{2})) \geqslant \frac{dr}{\epsilon}$. By the lemma above, we have that $(\varphi(g,k_{1}), \varphi(g,k_{2})) \in w_{r\epsilon}$. So, $\{g \in G: \varphi(g,K) \notin Small(w_{r\epsilon})\} \subseteq \{g \in G: d(\varphi(g,K),x) \geqslant \frac{dr}{\epsilon}\}$ which is finite, since $\varphi$ is properly discontinuous and $X$ is proper. However, $\{g \in G: \varphi(g,K) \notin Small(u) \} \subseteq  \{g \in G: \varphi(g,K) \notin Small(w_{r\epsilon})\}$, which implies that $\{g \in G: \varphi(g,K) \notin Small(u) \}$ is finite. Thus, $X+_{f}\partial X \in Pers(\varphi)$
\end{proof}

\subsection{Convergence}

Let's consider here the action $\psi: G \curvearrowright Y$ with the convergence property (i.e. the induced action on the set of distinct triples is properly discontinuous) and the action $\varphi: G \curvearrowright X$ properly discontinuous and cocompact.

\begin{prop}(Attractor-Sum - Gerasimov, Proposition 8.3.1 of \cite{Ge2}) There exists a unique space  $X+_{f_{c}} Y$ such that $\varphi+\psi$  has the convergence property.
\end{prop}

Furthermore, Gerasimov constructed such topology from the special case $G+_{\partial_{c}}Y$ (with the topology where $L+\psi$ has the convergence property) using the same process as the functor $\Pi_{K}$ (actually this was the motivation to define such functor). So, $f = \partial_{c\Pi_{K}}$. In the same article it was proved that $G+_{\partial_{c}}Y$ is perspective (\textit{Proposition 7.5.4 of \cite{Ge2}}). So, we have:

\begin{prop}If $X+_{f_{c}} Y$ has the convergence property, then it is perspective.
\end{prop}

\begin{proof}We have that $G+_{\partial_{c}}Y$ is perspective, which implies that $X+_{\partial_{c\Pi_{K}}}Y$ is perspective. But $f = \partial_{c\Pi_{K}}$, which implies that $X+_{f_{c}} Y$ is perspective.
\end{proof}

And also:

\begin{prop}If $X+_{f_{c}} Y$ has the convergence property, then, $\forall K \subseteq X$ fundamental domain, $G+_{f_{c\Lambda_{K}}} Y$ has the convergence property.
\end{prop}

\begin{proof}There exists $G+_{\partial_{c}}Y$ with the convergence property and $f_{c} = \partial_{c\Pi_{K}}$. Since $X+_{f_{c}} Y$ is perspective, $f_{c\Lambda_{K}} = (\partial_{c\Pi_{K}})_{\Lambda_{K}} = \partial_{c}$, which implies that $G+_{f_{c\Lambda_{K}}} Y$ has the convergence property.
\end{proof}

Summarising, we get:

\begin{prop}\label{summary}If $\psi: G \curvearrowright Y$ is an action with the convergence property, then $G+_{\partial_{c}}Y \in Pers(G)$ and  $X+_{f_{c}}Y \in Pers(\varphi)$. Furthermore, the functors $\Pi$ and $\Lambda$ preserve the convergence property.
\eod\end{prop}

\begin{prop}(Gerasimov and Potyagailo, Lemma 5.3 of \cite{GP2}) Let $\\ \psi_{i}: G \curvearrowright Y_{i}$ be convergence actions, with $\psi_{2}$ minimal and $Y_{2}$ with more than two points, and $\nu: Y_{1} \rightarrow Y_{2}$ a continuous equivariant map. Then, $id_{G}+\nu: G+_{\partial_{c1}}Y_{1} \rightarrow G+_{\partial_{c2}}Y_{2}$ is a continuous equivariant map.
\end{prop}

\begin{cor}Let $\varphi: G \curvearrowright X$ be a properly discontinuous cocompact action, $\psi_{i}: G \curvearrowright Y_{i}$ be convergence actions, with $\#Y_{2} \neq 2$ and $\psi_{2}$ minimal, and $\\ \nu: Y_{1} \rightarrow Y_{2}$ a continuous equivariant map. Then, the application $\\ id_{X}+\nu: X+_{f_{c2}}Y_{1} \rightarrow X+_{f_{c2}}Y_{2}$ is continuous and equivariant.
\end{cor}

\begin{proof}It follows from the fact that $id_{G}+\nu: G+_{\partial_{c1}}Y_{1} \rightarrow G+_{\partial_{c2}}Y_{2}$ is a continuous equivariant map (and then a morphism of $Pers(G)$) and $id_{X}+\nu = \Pi(id_{G}+\nu$).
\end{proof}

Let $Conv(G)$ be the category of minimal convergence actions on spaces with cardinality different than $2$ and $G$-equivariant continuous maps. By the proposition and the corollary above, we have the full and faithful functors $\Gamma_{G}: Conv(G) \rightarrow Pers(G)$ and $\Gamma_{\varphi}: Conv(G) \rightarrow Pers(\varphi)$ that send an object $\psi:G \curvearrowright Y$ to $G+_{\partial_{c}}Y$ and $X+_{f_{c}}Y$, respectively, and a morphism $\nu: Y_{1} \rightarrow Y_{2}$ to $id_{G}+\nu: G+_{\partial_{c1}}Y_{1} \rightarrow G+_{\partial_{c2}}Y_{2}$ and $id_{X}+\nu: X+_{f_{c2}}Y_{1} \rightarrow X+_{f_{c2}}Y_{2}$, respectively. The \textbf{Proposition \ref{summary}} says that the diagram is commutative:

$$ \xymatrix{ & Conv(G)\ar[ld]_{\Gamma_{G}} \ar[rd]^{\Gamma_{\varphi}} &  \\
            Pers(G) \ar@<2pt>[rr]^{\Pi} & & Pers(\varphi) \ar@<2pt>[ll]^{\Lambda} }$$

It is considered in \cite{GP2} the pullback problem: for two convergence actions $\psi_{1}: G \curvearrowright Y_{1}$ and $\psi_{2}: G \curvearrowright Y_{2}$, is there a convergence action $\psi_{3}: G \curvearrowright Y_{3}$ and two equivariant continuous maps $\nu_{1}: Y_{3} \rightarrow Y_{1}$ and $\nu_{2}: Y_{3} \rightarrow Y_{2}$? On the article, Gerasimov and Potyagailo answered negatively this question, \textit{Proposition 5.2 of \cite{GP2}}, even when the actions are both relatively hyperbolic (and the group is not finitely generated), \textit{Proposition 5.5 of \cite{GP2}}. On $Pers(G)$ (and equivalently on $Pers(\varphi)$) this problem has a positive answer, since the category is closed under finite limits, and then, closed under finite products. So, convergence actions have products on the category $Pers(G)$ (which, in general, do not have the convergence property).

\subsection{End Spaces}

It is clear that the compactification of a finitely generated group with its end space is perspective, since it has the convergence property (\textit{Proposition 9.3.2 of \cite{Ge2}}), but we have a direct proof.

Let $X$ be a connected, locally connected and locally compact Hausdorff space and $\varphi: G \curvearrowright X$ a properly discontinuous cocompact action.

\begin{prop}$\varphi$ extends uniquely to an action by homeomorphisms $\varphi+\psi: G \curvearrowright X+_{f}Ends(X)$, where $X+_{f}Ends(X)$ is the Freudenthal compactification of $X$.
\end{prop}

\begin{proof}By the \textbf{Proposition \ref{extensaoends}}, for every $g \in G$, there exists a unique continuous map $(\varphi+\psi)(g,\_): X+_{f}Ends(X) \rightarrow X+_{f}Ends(X)$ extending $\varphi(g,\_)$. Since $id_{X+_{f}Ends(X)}$ is an extension of the map $\varphi(1,\_)$, it follows that $(\varphi+\psi)(1,\_) =$ $id_{X+_{f}Ends(X)}$. We have also that, $\forall g,h\in G$, $(\varphi+\psi)(gh,\_)$ and $(\varphi+\psi)(g,\_) \circ (\varphi+\psi)(h,\_)$ are both extensions of $\varphi(gh,\_)$, which implies that $(\varphi+\psi)(gh,\_) = (\varphi+\psi)(g,\_) \circ (\varphi+\psi)(h,\_)$. Thus, $\varphi+\psi$ is a group action that extends $\varphi$. The uniqueness comes from the uniqueness of the extension of each map.
\end{proof}

\begin{prop}If $X+_{g}Y \in T_{2}Comp(\varphi)$ is such that $Y$ is totally disconnected, then there exists a unique continuous and equivariant application $id+\phi: X+_{f}Ends(X) \rightarrow X+_{g}Y$.
\end{prop}

\begin{proof}We already know, by the \textbf{Corollary \ref{universalends}}, that this map exists and it is unique. Since this map is equivariant in $X$ and $X$ is dense in $X+_{f}Ends(X)$, it follows that it is equivariant in the whole space.
\end{proof}

\begin{prop}$\forall K \subseteq X$ compact, $X+_{f_{K}}\pi^{u}_{K}(X-K)$ has the property that $\forall K' \subseteq X$ compact, $\forall u \in \mathcal{U}_{f_{K}}$, the set $\{g \in G: \varphi(g,K')\notin Small(u)\}$ is finite.
\end{prop}

\begin{proof}Let $u \in \mathcal{U}_{f_{K}}$ and $K' \subseteq X$ be a compact. We have that there exists $K''$ compact and connected such that $K' \subseteq K''$. Let $v \in \mathcal{U}_{f_{K}}$ such that $v \subseteq u$ and $\forall x_{U} \in \pi^{u}_{K}(X-K)$ (the point of the set corresponding to the component $U$) and $\forall V$ a component different to $U, \ (x_{U},y) \notin v, \ \forall y \in V$ (it exists since $V$ is bounded and $\pi^{u}_{K}(X-K)$ is finite). Take $v' \in \mathcal{U}: \ v'^{2} \subseteq v$ and $v'' = int \ v$. We have that $v''$ is open, which implies that $\mathfrak{B}(\pi^{u}_{K}(X-K), v'')$ is open and then its complement $C$ is compact. By the choice of $v''$, we have that $\forall x_{U} \in \pi^{u}_{K}(X-K), \ \mathfrak{B}(x_{U},v'') \subseteq U$. Since $\varphi$ is properly discontinuous, the set $A = \{g \in G: \varphi(g,K'') \cap C \neq \emptyset\}$ is finite. Since $K''$ is connected, we have that $\forall g \notin A, \ \varphi(g,K'')$ is in a component $U$ of $X-K$, which implies that $\varphi(g,K'') \subseteq  \mathcal{B}(x_{U},v'')$ and then $\varphi(g,K'') \in Small(v''^{2}) \subseteq Small(u)$. So, $\{g \in G: \varphi(g,K'') \notin Small(u)\} \subseteq A$, which implies that it is finite. But $K' \subseteq K''$, which implies that $\{g \in G: \varphi(g,K') \notin Small(u)\} \subseteq$ $\\ \{g \in G: \varphi(g,K'') \notin Small(u)\}$, and then it is also finite.

\end{proof}

\begin{prop}$X+_{f}Ends(X) \in Pers(\varphi)$.
\end{prop}

\begin{proof}We have that $X+_{f}Ends(X) \in Comp(\varphi)$ and it is Hausdorff. Since $X+_{f}Ends(X) = \lim\limits_{\longleftarrow} (X+_{f_{K}}\pi^{u}_{K}(X-K))$, and each one preserves the property that $\forall K' \subseteq X$ compact, $\forall u \in \mathcal{U}_{f_{K}}$, the set $\{g \in G: \varphi(g,K')\notin Small(u)\}$ is finite, we have that $X+_{f}Ends(X) \in Pers(\varphi)$.
\end{proof}

Now it seems natural to generalize the definition of the Freudenthal compactification:

\begin{defi}We say that a space $X+_{f_{F}}Ends(\varphi)$ is the Freudenthal compactification of $X$ with respect to the action $\varphi: G \curvearrowright X$ if it is perspective and for another space $X+_{g}Y \in Pers(\varphi)$, with $Y$ totally disconnected, there exists a unique morphism $id_{X}+\phi: X+_{f_{F}}Ends(\varphi) \rightarrow X+_{g}Y$.
\end{defi}

It is clear, by the proposition above, that, for $X$ connected and locally connected, $X+_{f_{F}}Ends(\varphi)$ is the usual Freudenthal compactification. So, it do not depend of $\varphi$. As the special case of $X = G$ and $\varphi = L$, we denote the Freudenthal compactification of $G$ by $G+_{\partial_{F}}Ends(G)$. The existence of the Freudenthal compactification of a group is a immediate consequence of the representation theorem of perspective compactifications of the group with totally disconnected boundary due to Abels (Satz 1.8, \cite{Ab}).

\begin{prop}\label{universalperspectiveends}Let $X_{1},X_{2}$ be locally compact Hausdorff spaces and let $\varphi_{i}: G \curvearrowright X_{i}$ be properly discontinuous cocompact actions. If the functor $\Pi_{12}: Pers(\varphi_{1}) \rightarrow Pers(\varphi_{2})$ is an isomorphism of the categories that preserve remainders and its maps, then $\Pi_{12}(X_{1}+_{f_{1F}}Ends(\varphi_{1})) = X_{2}+_{f_{2F}}Ends(\varphi_{2})$.
\end{prop}

\begin{obs}The composition of Attractor-Sum functors have the property of the hypothesis.
\end{obs}

\begin{proof}Let $X_{2}+_{g}Y \in Pers(\varphi_{2})$ with $Y$ totally disconnected, $X_{1}+_{g'}Y = \Pi_{12}^{-1}(X_{2}+_{g}Y)$ and $X_{2}+_{f'}Ends(\varphi_{1}) = \Pi_{12}(X_{1}+_{f_{1F}}Ends(\varphi_{1}))$. Since $Y$ is totally disconnected, there exists a continuous equivariant application $id_{X_{1}}+\phi: X_{1}+_{f_{1F}} Ends(\varphi_{1}) \rightarrow X_{1}+_{g'}Y$. Hence, there exists a continuous equivariant map $id_{X_{2}}+\phi = \Pi_{12}(id_{X_{1}}+\phi): X_{2}+_{f'}Ends(\varphi_{1}) \rightarrow X_{2}+_{g}Y$. Let $id_{X_{2}}+\phi': X_{2}+_{f'}Ends(\varphi_{1}) \rightarrow X_{2}+_{g}Y$ be another continuous equivariant map. Then, $id_{X_{1}}+\phi' = \Pi_{12}^{-1}(id_{X_{2}}+\phi'): X_{1}+_{f_{1F}}Ends(\varphi_{1}) \rightarrow X_{1}+_{g'}Y$ is continuous and equivariant. But such a map is unique, which implies that $id_{X_{1}}+\phi' = id_{X_{1}}+\phi$, and then $\phi' = \phi$ and $id_{X_{2}}+\phi' = id_{X_{2}}+\phi$. So, $id_{X_{2}}+\phi$ is unique, which implies that $X_{2}+_{f'}Ends(\varphi_{1})$ is the Freudenthal compactification of $X_{2}$ with respect to $\varphi_{2}$. Thus, $\Pi_{12}(X_{1}+_{f_{1F}}Ends(\varphi_{1})) = X_{2}+_{f_{2F}}Ends(\varphi_{2})$.
\end{proof}

\begin{cor}Let $X$ be a locally compact Hausdorff topological space and let $\varphi: G \curvearrowright X$ be properly discontinuous cocompact actions. Then, the Freudenthal compactification $X+_{f_{F}}Ends(\varphi)$ exists.
\end{cor}

\begin{proof}We have that $G+_{\partial_{F}}Ends(G)$ exists and the attractor-sum functor $\Pi: Pers(G) \rightarrow Pers(\varphi)$ satisfies the conditions of the last proposition. Thus, $\Pi(G+_{\partial_{F}}Ends(G))$ is the Freudenthal compactification of $X$ with respect to $\varphi$.
\end{proof}

\begin{cor}Let $X_{1}$ and $X_{2}$ be locally compact Hausdorff spaces and $\varphi_{i}: G \curvearrowright X_{i}$ properly discontinuous cocompact actions. Then, $Ends(\varphi_{1}) \cong Ends(\varphi_{2})$.
\eod\end{cor}

As a special case, we get:

\begin{cor}(Hopf) Let $X_{1},X_{2}$ be connected, locally connected and locally compact Hausdorff spaces and $\varphi_{i}: G \curvearrowright X_{i}$ properly discontinuous cocompact actions. Then, $Ends(X_{1}) \cong Ends(X_{2})$.
\eod\end{cor}

\begin{obs}The Freudenthal compactification of a group do not need to be initial on the subcategory of $T_{2}Comp(G)$ where the remainders are totally disconnected. On the \textbf{Example \ref{exemplochave}}, $\Z+_{\partial}\{x_{1},x_{2},x_{3}\} \in T_{2}Comp(G)$ with the subspace $\{x_{1},x_{2},x_{3}\}$ totally disconnected, but there is no continuous map $id+\phi: \Z+_{\partial_{F}}Ends(\Z) \rightarrow \Z+_{\partial}\{x_{1},x_{2},x_{3}\}$, since it must be surjective and $\#Ends(\Z) = 2$.
\end{obs}

\begin{prop}Let $G$ be a finitely generated group, $\varphi: G \curvearrowright X$ properly discontinuous and cocompact, and $X+_{g}Y \in Pers(\varphi)$ such that $X$ is dense and $Y$ has more than $2$ connected components. Then $G$ splits as an amalgamated product under a finite subgroup or as a HNN extension under a finite subgroup.
\end{prop}

\begin{proof}There exists an object $G+_{\partial}Y \in Pers(G)$ with $G$ dense (because of the isomorphism between $Pers(G)$ and $Pers(\varphi)$). The quotient of $G+_{\partial}Y$ by the relation defined by the connected components is a Hausdorff compact space of the for $G+_{\partial'}Y'$ with $Y'$ totaly disconnected and $G$ dense. Since the quotient map is the identity on G and it is equivariant (considering the induced action of $G$ in $G+_{\partial'}Y')$, we have that $G+_{\partial'}Y'$ is perspective. Since $Y$ has more than two connected components, we have that $\#Y' > 2$, which implies that $\#Ends(G) \geqslant \# Y' > 2$. Thus, the result follows from Stallings' Theorem.
\end{proof}

\subsection{Spaces that are boundaries}

Unlike spaces where a group acts with the convergence property, the boundaries of groups with the perspective property have quite less topological restrictions.

First, we present an example of non finitely generated group that admits any compact Hausdorff space with countable basis as its boundary (in a compactification with the perspective property). Stallings shows on his book \cite{St} that this example has infinitely many ends.

Let $\Z_{n^{\infty}} = \{\frac{a}{n^{i}}\in \Q : a\in\Z, i \in \N\}/ \Z$. Let's fix $k \in \N$. Let, for $j \in \{0,...,k-1\}, \ A_{j,k} = \{0\} \cup \{\frac{a}{n^{i}}+\Z \in \Z_{n^{\infty}}: a\in \Z, \ n\nmid a, \ i \equiv j \ mod \ k\}$ and, for $m \in \N, \ A_{j,k}^{m} = \{\frac{a}{n^{i}}+\Z \in A_{j,k}: \ a\in \Z, \ n\nmid a, \ i \geqslant km+j\}$ .

We have that $\Z_{n^{\infty}} = \bigcup_{j}A_{j,k}$ and $\forall i\neq j, A_{i,k} \cap A_{j,k} = \{0\}$.

Let $f_{k}: Closed(\Z_{n^{\infty}}) \rightarrow Closed(Y_{k})$, where $Y_{k} = \{y_{0},...,y_{k-1}\}$ with the discrete topology, defined by $f_{k}(F) = \{y_{j} \in Y_{k}: \#F\cap A_{j,k} = \aleph_{0}\}$.

\begin{prop}$f_{k}$ is admissible.
\end{prop}

\begin{proof}We have that $f_{k}(\emptyset) = \{y_{j}: \# \emptyset \cap A_{j,k}= \aleph_{0}\} = \emptyset$. Let $F_{1},F_{2} \in Closed(\Z_{n^{\infty}})$. If $y_{j} \in f_{k}(F_{i})$, then $F_{i}\cap A_{j,k}$ is infinite, which implies that $(F_{1} \cup F_{2})\cap A_{j,k}$ is infinite and then $y_{j} \in f_{k}(F_{1}\cup F_{2})$. Hence $f_{k}(F_{1})\cup f_{k}(F_{2}) \subseteq f_{k}(F_{1}\cup F_{2})$. Let $y_{j} \notin f_{k}(F_{1})\cup f_{k}(F_{2})$. Then, $F_{1}\cap A_{j,k}$ and $F_{2}\cap A_{j,k}$ are finite, which implies that $(F_{1} \cup F_{2}) \cap A_{j,k}$ is finite, which implies that $y_{j}\notin f_{k}(F_{1}\cup F_{2})$. Hence, $f_{k}(F_{1})\cup f_{k}(F_{2}) = f_{k}(F_{1}\cup F_{2})$. Thus, $f_{k}$ is admissible.
\end{proof}

\begin{prop}$\Z_{n^{\infty}}+_{f_{k}}Y_{k}$ is compact.
\end{prop}

\begin{proof}Let $F \in Closed(\Z_{n^{\infty}})$ be non compact. If, $\forall j \in \{0,...,k-1\}, \ F \cap A_{j,k}$ is finite, then $F = (F \cap A_{0,k}) \cup ... \cup (F \cap A_{k-1,k})$ is finite, absurd. So, there exists $j \in \{0,,...,k-1\}: F\cap A_{j,k}$ is infinite, which implies that $y_{j} \in f_{k}(F)$. Thus, $\Z_{n^{\infty}}+_{f_{k}}Y_{k}$ is compact.
\end{proof}

\begin{prop}$\Z_{n^{\infty}}+_{f_{k}}Y_{k}$ is Hausdorff.
\end{prop}

\begin{proof}Let $i \neq j \in \{0,...,k-1\}$. Let $B_{i} = \bigcup_{i' \neq i}A_{i',k}$ and $B_{j} = \bigcup_{j' \neq j}A_{j',k}$. We have that $B_{i},B_{j} \in Closed(\Z_{n^{\infty}})$, $B_{i} \cup B_{j} = \Z_{n^{\infty}}$, $f_{k}(B_{i}) = Y - \{y_{i}\}$ and $f_{k}(B_{j}) = Y - \{y_{j}\}$. Hence, $y_{i} \notin f_{k}(B_{i})$ and $y_{j} \notin f_{k}(B_{j})$. Thus, $\Z_{n^{\infty}}+_{f_{k}}Y_{k}$ is Hausdorff.
\end{proof}

\begin{lema}Let $a \in \Z_{n^{\infty}}$, with $a = \frac{a'}{n^{i}}+\Z$ and $n\nmid a'$. Then, $\forall j \in\{0,...,k-1\}$, $\forall m \in \N:$ $km+j \geqslant i$, we have that $a + A_{j,k}^{m} = A_{j,k}^{m}$.
\end{lema}

\begin{proof}Let $b \in a+ A_{j,k}^{m}$. Then, $b = \frac{b'}{n^{km+j}}$, with $n\nmid b'$. We have that $a+b = \frac{a'n^{km+j-i}+b'}{n^{km+j}} \in A_{j,k}^{m}$, since $n\nmid a'n^{km+j-i}+b'$ (because $n \nmid b'$). So, $a+ A_{j,k}^{m} \subseteq A_{j,k}^{m}$. Then, $A_{j,k}^{m} = a - a+ A_{j,k}^{m} \subseteq a+A_{j,k}^{m}$. Thus, $a+ A_{j,k}^{m} = A_{j,k}^{m}$.
\end{proof}

\begin{prop}$L+id: \Z_{n^{\infty}} \curvearrowright \Z_{n^{\infty}}+_{f_{k}}Y_{k}$ is an action by homeomorphisms.
\end{prop}

\begin{proof}Let $a \in \Z_{n^{\infty}}, \ F \in Closed(\Z_{n^{\infty}})$ and $y_{j}\in f_{k}(F)$. Then, $F \cap A_{j,k}$ is infinite. We have that there is $m \in \N$ such that $-a+A_{j,k}^{m} = A_{j,k}^{m}$ and $A_{j,k} - A_{j,k}^{m}$ is finite, which implies that $F \cap A_{j,k}^{m} =  F \cap (-a+A_{j,k}^{m})$ is infinite and then $(a+F) \cap A_{j,k}^{m}$ is infinite. So, $y_{j} \in f_{k}(a+F)$ and then $f_{k}(F) \subseteq f_{k}(a+F)$. But $f_{k}(a+F) \subseteq f_{k}(-a+a+F) = f_{k}(F)$, which implies that $f_{k}(F) = f_{k}(a+F)$. Thus, $L(a,\_)+id: \Z_{n^{\infty}}+_{f_{k}}Y_{k} \rightarrow \Z_{n^{\infty}}+_{f_{k}}Y_{k}$ is a homeomorphism.
\end{proof}

Thus, $\Z_{n^{\infty}}+_{f_{k}}Y_{k}$ is perspective.

\begin{prop}If $k_{2} \mid k_{1}$, then $id+\pi: \Z_{n^{\infty}}+_{f_{k_{1}}}Y_{k_{1}} \rightarrow \Z_{n^{\infty}}+_{f_{k_{2}}}Y_{k_{2}}$ is continuous and $\Z_{n^{\infty}}$-equivariant, where $\pi: Y_{k_{1}} \rightarrow Y_{k_{2}}$ send $y_{j}$ to $y_{i}$ if $j \equiv i \ mod \ k_{2}$.
\end{prop}

\begin{proof}The map is equivariant since  the actions are both trivial on the boundary.

Let $F \in Closed(\Z_{n^{\infty}})$, $y_{j} \in f_{k_{1}}(F)$ and $y_{i} = \pi(y_{j})$. Then, $F\cap A_{j,k_{1}}$ is infinite. If $\frac{a}{n^{j'}}+\Z \in A_{j,k_{1}}$, then $j' \equiv j \ mod \ k_{1}$, which implies that $j' = n_{1}k_{1}+j$, with $n_{1} \in \N$. By hypothesis, we have that $k_{1} = n_{2} k_{2}$ and $j = n_{3}k_{2}+i$, with $n_{2},n_{3} \in \N$. So, $j' = n_{1}n_{2}k_{2}+n_{3}k_{2}+i$, which implies that $\frac{a}{n^{j'}}+\Z \in A_{i,k_{2}}$. Hence, $A_{j,k_{1}} \subseteq A_{i,k_{2}}$, which implies that $F\cap A_{i}$ is infinite, and then $\pi(y_{j}) \in f_{k_{2}}(F)$. Then, $\pi(f_{k_{1}}(F)) \subseteq f_{k_{2}}(F)$, which implies that $f_{k_{1}}(F) \subseteq \pi^{-1}(\pi(f_{k_{1}}(F))) \subseteq \pi^{-1}(f_{k_{2}}(F))$. So, we have the diagram:

$$ \xymatrix{ Closed(\Z _{n^{\infty}}) \ar[r]^{f_{k_{1}}} & Closed(Y_{k_{1}}) \\
            Closed(\Z_{n^{\infty}}) \ar[r]^{f_{k_{2}}} \ar[u]^{id} \ar@{}[ur]|{\subseteq} & Closed(Y_{k_{2}}) \ar[u]^{\pi^{-1}} } $$

Thus, the map $id+\pi$ is continuous.

\end{proof}

Let $\xi_{k_{1}k_{2}}: \Z_{n^{\infty}}+_{f_{k_{1}}}Y_{k_{1}} \rightarrow \Z_{n^{\infty}}+_{f_{k_{2}}}Y_{k_{2}}$ be the map defined on the last proposition.  We have that $\{\Z_{n^{\infty}}+_{f_{k}}Y_{k},\xi_{k_{1}k_{2}}\}$ is a directed set in $Pers(\Z_{n^{\infty}})$, which implies that it does have a limit $\Z_{n^{\infty}}+_{f}K$ and, since it is a directed set, $\Z_{n^{\infty}}$ is dense and $K \cong \lim\limits_{\longleftarrow} Y_{k}$, which is a Cantor set. Since $\Z_{n^{\infty}}$ is abelian, it follows that the action of $\Z_{n^{\infty}}$ on $\Z_{n^{\infty}}+_{f}K$ must be of the form $L+id$.

\begin{prop}Let $M$ be a compact Hausdorff space with countable basis. Then, there exists a perspective compact $\Z_{n^{\infty}}+_{f_{M}}M$ with $f_{M}(\Z_{n^{\infty}}) = M$.
\end{prop}

\begin{proof}There exists a quotient map $\pi: K \rightarrow M$. If $\sim$ is the equivalence relation in $K$ induced by $\pi$, we have that $\sim$ is closed in $K^{2}$, which implies that it is closed in $(\Z_{n^{\infty}}+_{f}K)^{2}$, which implies that $\sim' = \Delta (\Z_{n^{\infty}}+_{f}K) \cup \sim$ is a closed equivalence relation in $\Z_{n^{\infty}}+_{f}K$. So, it has a quotient of the form $\Z_{n^{\infty}}+_{f_{M}}M$ that is Hausdorff, since it is closed. By the \textbf{Proposition \ref{quotientaction}}, the map $L+id: G \curvearrowright \Z_{n^{\infty}}+_{f_{M}}M$ is continuous, which implies that it is perspective.
\end{proof}

We have also an example of finitely generated group with many compactifications with the perspective property:

\begin{ex}We have a space $\Z^{n}+_{\partial}S^{n-1}$, for $n > 1$, with the perspective property and $\partial(\Z^{n}) = S^{n-1}$ (that comes from the visual boundary of the $CAT(0)$ space $\R^{n}$). Let $X$ be any Peano space. Then, by the Hahn-Mazurkiewicz Theorem, there exists a quotient map $\pi: S^{n} \rightarrow X$. So, a similar argument of the last proposition says that there is a compactification $\Z^{n}+_{\partial'}X$ with the perspective property.
\end{ex}


\begin{thebibliography}{99}
\bibitem{Ab} {\sc H. Abels}, {\it Specker-Kompaktifizierungen von lokal kompakten topologischen Gruppen}. Math. Z., v. 135, p. 325-361, 1973/74.
\bibitem{Bo} {\sc F. Borceux}, {\it Handbook of Categorical Algebra 1 - Basic Category Theory}. Encyclopedia of Mathematics and its Applications, Cambridge University Press, Great Britain, 1994.
\bibitem{Bou} {\sc N. Bourbaki}, {\it Topologie Générale}. Hermann, Paris, 1965.
\bibitem{BH} {\sc M. R. Bridson and A. Haefliger}, {\it Metric spaces of non-positive curvature}. Springer, 1964.
\bibitem{DGGP} {\sc M. Dussaule, I. Gekhtman, V. Gerasimov, and L. Potyagailo}, {\it The Martin boundary of relatively hyperbolic groups with virtually abelian parabolic subgroups}. arXiv:1711.11307v2 [math.GR], 2018.
\bibitem{En} {\sc R. Engelking}, {\it General Topology}. Sigma Series in Pure Mathematics, Heldermann Verlag Berlin, 1989.
\bibitem{Fr} {\sc H. Freudenthal}, {\it $\ddot{U}$ber die Ender topologischer Räume und Gruppen}. Math Z, v. 33, p. 692-713, 1931.
\bibitem{Ge1} {\sc V. Gerasimov}, {\it Expansive convergence groups are relatively hyperbolic}. Geometric and Functional Analysis, v. 19, p. 137-169, 2009.
\bibitem{Ge2} {\sc V. Gerasimov}, {\it Floyd maps for relatively hyperbolic groups}. Geometric and Functional Analysis, v. 22, p. 1361-1399, 2012.
\bibitem{GP1} {\sc V. Gerasimov and L. Potyagailo}, {\it Non-finitely generated relatively hyperbolic groups and Floyd quasiconvexity}. Groups, Geometry and Dynamics, v. 9, n. 2, p. 369–434, 2015.
\bibitem{GP2} {\sc V. Gerasimov and L. Potyagailo}, {\it Similar relatively hyperbolic actions of a group}. International Mathematics Research Notices, v. 2016, n. 7, p. 2068-2103, 2016.
\bibitem{Ho} {\sc H. Hopf}, {\it Ender offener Räume und unendliche diskontinuierliche Gruppen}. Coment. Math Helv., v. 16, p. 81-100, 1943.
\bibitem{Ke} {\sc A. S. Kechris}, {\it Classical Descriptive Set Theory}. Springer-Verlag, 1995.
\bibitem{Sp} {\sc E. Specker}, {\it Endenverbände von Räumen und Gruppen}. Math. Ann., v. 122, p. 167-174, 1950.
\bibitem{St} {\sc J. Stallings}, {\it Group Theory and Three-Dimensional Manifolds}. James K. Wittemore Lectures in Mathematics, Yale University Press, New Haven and London, 1971.
\bibitem{To} {\sc C. Toromanoff}, {\it Convergence groups and quasi-Specker compactifications}. PhD Thesis to Université de Lille 1, 2018.
\bibitem{SW} {\sc S. Willard}, {\it General Topology}. Addison-Wesley Series in Mathematics, Addison-Wesley Publishing Company, 1968.
\end{thebibliography}
\end{document}